%% file: neurips_camera.tex
\documentclass{article}

     \PassOptionsToPackage{authoryear, compress, sort}{natbib}

     \usepackage[preprint]{neurips_2021}

\usepackage{natbib}
\usepackage[utf8]{inputenc} 
\usepackage[T1]{fontenc}    
\usepackage{url}            
\usepackage{booktabs}       
\usepackage{amsfonts}       
\usepackage{nicefrac}       
\usepackage{microtype}      
\usepackage{xcolor}         

\usepackage{algorithm}
\usepackage{algorithmic}

\usepackage{amsmath, amssymb}
\usepackage{amsthm}
\usepackage{bbm}
\usepackage{mathbbol}
\usepackage{float}
\usepackage{dirtytalk}
\usepackage{mathrsfs}
\DeclareMathOperator*{\argminA}{arg\,min}
\usepackage[english]{babel}

\usepackage{booktabs}
\usepackage{empheq}
\usepackage{mdframed}
\usepackage{pifont}
\usepackage{enumitem}
\usepackage[verbose]{wrapfig}
\usepackage{placeins}
\usepackage{empheq}
\usepackage{bm}

\definecolor{mydarkblue}{rgb}{0,0.08,0.45}
\usepackage[colorlinks=true,
    linkcolor=mydarkblue,
    citecolor=mydarkblue,
    filecolor=mydarkblue,
    urlcolor=mydarkblue,
    pdfview=FitH]{hyperref}

\usepackage{cancel}
\usepackage{array}
\usepackage{comment}

\usepackage{tikz}
\usetikzlibrary{shapes}
\usetikzlibrary{decorations.markings}
\usetikzlibrary{plotmarks}
\usetikzlibrary{patterns}
\usepackage{ctable}
\usepackage{multirow}
\usepackage{multicol}

\usepackage{color, colortbl}
\definecolor{mydarkblue}{rgb}{0,0.08,0.45}
    
\usepackage{empheq}

\definecolor{myteal}{RGB}{27,158,119}
\definecolor{myorange}{RGB}{217,95,2}
\definecolor{myred}{RGB}{231,41,138}
\definecolor{mypurple}{RGB}{152,78,163}
\definecolor{myblue}{RGB}{55,126,184}
\definecolor{mygreen}{RGB}{0,100,0}

\newtheorem{lemma}{Lemma}

\newtheorem{theorem}{Theorem}

\newtheorem{proposition}{Proposition}
\newtheorem{assumption}{Assumption}
\newtheorem{remark}{Remark}
\newtheorem{corollary}{Corollary}

\numberwithin{equation}{section}
\numberwithin{lemma}{section}
\numberwithin{proposition}{section}
\numberwithin{definition}{section}
\numberwithin{corollary}{section}

\DeclareMathOperator{\E}{\mathbb{E}}

\DeclareMathOperator{\Exp}{\mathbb{E}}

\include{macros}

\title{Dynamics of Stochastic Momentum Methods on Large-scale, Quadratic Models}

\author{%
  Courtney Paquette \thanks{Website
  \href{courtneypaquette.github.io}{courtneypaquette.github.io}
  .} \\
  Department of Mathematics and Statistics\\
  McGill University\\
  Montreal, Quebec H2Y 2M5 \\
  \texttt{courtney.paquette@mcgill.ca} \\
  \And
  Elliot Paquette \thanks{Website 
  \href{elliotpaquette.github.io}{elliotpaquette.github.io}
  .} \\
  Department of Mathematics and Statistics\\
  McGill University\\
  Montreal, Quebec H2Y 2M5 \\
  \texttt{elliot.paquette@mcgill.ca} \\
}

\begin{document}
\maketitle
\begin{abstract}
    We analyze a class of stochastic gradient algorithms with momentum on a high-dimensional random least squares problem. Our framework, inspired by random matrix theory, provides an exact (deterministic) characterization for the sequence of loss values produced by these algorithms which is expressed only in terms of the eigenvalues of the Hessian. This leads to simple expressions for nearly-optimal hyperparameters, a description of the limiting neighborhood, and average-case complexity. 
    
    As a consequence,
    we show that (small-batch) stochastic heavy-ball momentum with a fixed momentum parameter provides no actual performance improvement over SGD when step sizes are adjusted correctly. For contrast, in the non-strongly convex setting, it is possible to get a large improvement over SGD using momentum. By introducing hyperparameters that depend on the number of samples, we propose a new algorithm SDANA (stochastic dimension adjusted Nesterov acceleration) which obtains an asymptotically optimal average-case complexity while remaining linearly convergent in the strongly convex setting without adjusting parameters.
    
\end{abstract}

Methods that incorporate momentum and acceleration play an integral role in machine learning where they are often combined with stochastic gradients. Two of the most popular methods in this category are the heavy-ball method (HB) \citep{Polyak1962Some} and Nesterov's accelerated method (NAG) \citep{nesterov2004introductory}. These methods are known to achieve optimal convergence guarantees when employed with \textit{exact gradients} (computed on the full training data set), but in practice, these momentum methods are typically implemented with \textit{stochastic} gradients. In the influential work \citet{sutskever2013on}, the authors demonstrated empirical advantages of augmenting stochastic gradient descent (SGD) with the momentum machinery and, as a result, momentum methods are widely used for training deep neural networks.  Yet despite the popularity of these stochastic momentum methods, the theoretical understanding of these algorithms remains rather limited. 

In this paper, we study the dynamics of stochastic momentum methods (with batch size $1$ and constant step size) rooted in heavy-ball momentum and Nesterov's accelerated gradient algorithms on a least squares problem. Our approach uses a framework inspired by the phenomenology of random matrix theory (see \cite{paquetteSGD2021}), which gains explanatory power when the number of samples $(n)$ and features $(d)$ are large. A key contribution of this work is a simple description of the exact dynamics for a class of stochastic momentum methods in the \textit{high-dimensional limit}; we construct a smooth, deterministic function $\psi(t)$ such that $f(\xx_k) \to \psi(k/n)$ as $n \to \infty$. This function $\psi$ solves a Volterra integral equation:
\begin{equation} \label{eq:volterra_intro}
\psi(t) = F(t) + \int_0^t \mathcal{K}_s(t) \psi(s) \dif s.
\end{equation}

\begin{wrapfigure}[23]{r}{0.48\textwidth}
\vspace{-0.35cm}
 \centering
     \includegraphics[width = 1\linewidth]{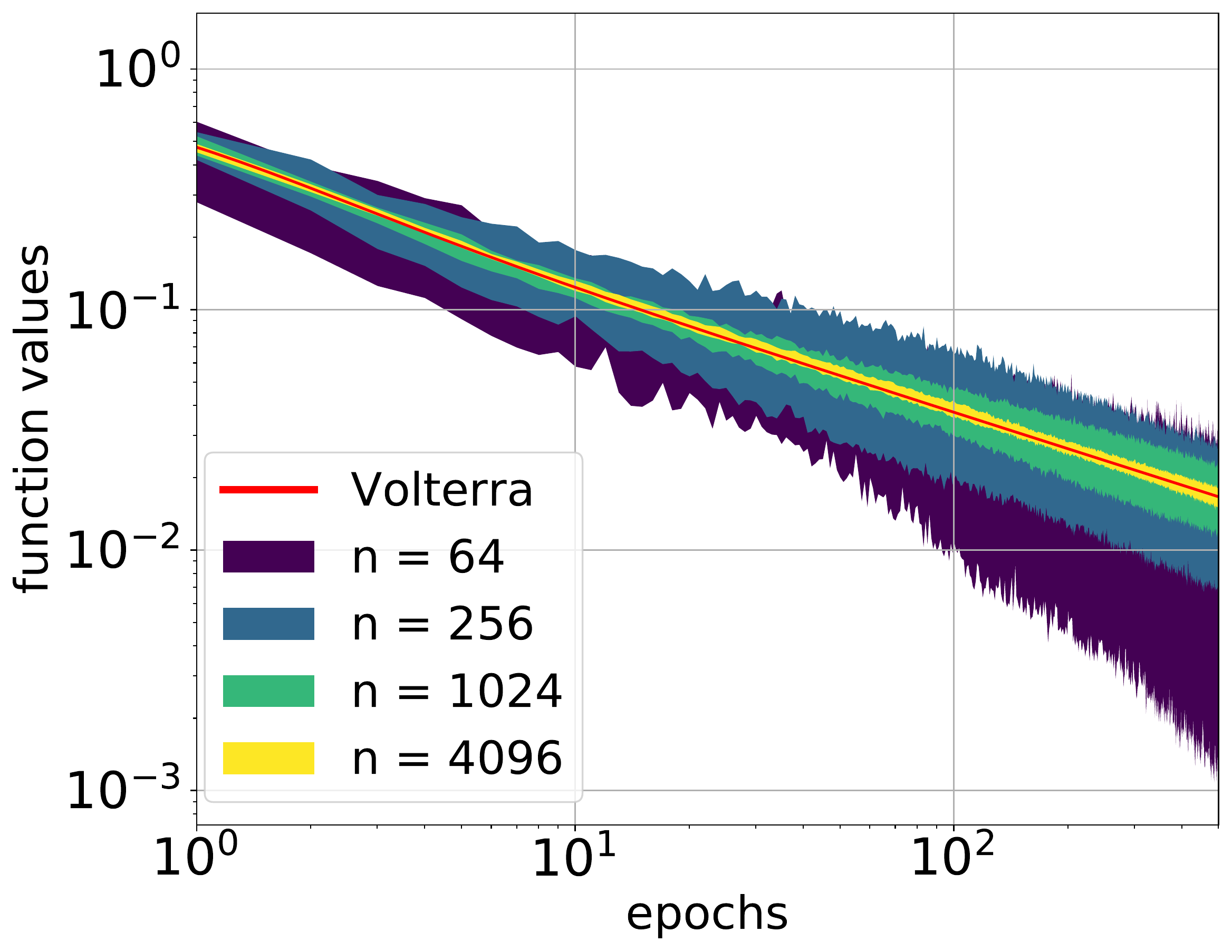}
     \vspace{-0.5cm}
     \caption{\textbf{Concentration of stochastic heavy-ball} (SHB) on a Gaussian random least squares problem (Sec.~\ref{sec:main_problem_setting}), $d = n$, an 80\% confidence interval (shaded region) over 10 runs for each $n$, the parameters for SHB (Table~\ref{table:stochastic_algorithms}) are $(\theta, \gamma) = (0.1,0.08)$. The random least squares problem becomes non-random in the large limit and all runs of SHB converge to a deterministic function $\psi(t)$ (red) given by our Volterra equation \eqref{eq:volterra_intro}.} \label{fig:concentration_SHB}
\end{wrapfigure}

Here $F(t)$ and $\mathcal{K}_s(t)$ are explicit, see Theorem~\ref{thm:hSGD} for a precise statement. This Volterra equation \eqref{eq:volterra_intro} gives an accurate prediction of the behavior of stochastic methods, see Figure \ref{fig:concentration_SHB}.  We then analyze these dynamics providing insight into step size and momentum parameter selections as well as providing both upper and lower average-case complexity (\textit{i.e.}, the complexity of an algorithm averaged over all possible inputs) for the last iterate.

As we show in this work, both theoretically and empirically, (small batch size) SGD with heavy-ball momentum (SHB) for any fixed momentum parameter does \textit{not} provide any acceleration over plain SGD on large-scale least square problems. 
We conclude under an identification of the parameters that $f(\xx^{\text{shb}}_k) = f(\xx^{\text{sgd}}_k)$ for all $k$ up to errors that vanish as $n$ grows large (upper bounds of this nature have been observed before: see \cite{zhang2019which,kidambi2018on,sebbouh2020almost}). Thus while SHB may provide a speed-up over SGD, it is only due to an effective increase in the learning rate, and this speed-up could be matched by appropriately adjusting the learning rate of SGD. 

The root of SHB's failure to provide meaningful acceleration is that a fixed momentum parameter is not aggressive enough when $n$ is large. 
We propose a new algorithm that uses a dimension-based modification of Nesterov (see Alg.~\ref{alg:class} and Table~\ref{table:stochastic_algorithms}). The resulting algorithm, SDANA, matches the average-case complexity of SGD when the least-squares problem is strongly convex and obtains an average-case complexity of $1/k^3$ in the convex setting.    

\begin{algorithm}[t!]
        \caption{Generic stochastic momentum method.} \label{alg:class}
       \begin{algorithmic} \STATE \textbf{Given}: step sizes $\Gamma_1, \Gamma_2 > 0$ and momentum parameter $\Delta(k) > 0$\\
          \STATE \textbf{Initialize}: $\xx_0 \in \mathbb{R}^d$ and $\yy_0 = \bm{0}$\\
          \STATE \textbf{for} $k \ge 1$, Select $i_k \in [n]$ uniformly and update
          \begin{align}
               \yy_k = (1-\Delta(k) ) \yy_{k-1} + \Gamma_1 \nabla f_{i_k}(\xx_k)  \qquad \text{and} \qquad \xx_{k+1} = \xx_k - \Gamma_2 \nabla f_{i_k}(\xx_k) -\yy_k
          \end{align}
        \end{algorithmic}  
\end{algorithm}

\section{Motivation and related work} \label{sec:motivation} We consider the large finite-sum setting
\[ \min_{\xx \in \mathbb{R}^d}~ \bigg \{  f(\xx) =  \frac{1}{n} \sum_{i=1}^n f_i(\xx) = \frac{1}{2} \sum_{i=1}^n (\aa_i \xx-\bb_i)^2 = \frac{1}{2} \|\AA \xx - \bb\|^2 \bigg \}, \]
for data matrix $\AA \in \mathbb{R}^{n \times d}$ whose $i$-th row is denoted by $\aa_i \in \mathbb{R}^{d \times 1}$ and target vector $\bb \in \mathbb{R}^n$ (detailed in Section~\ref{sec:main_problem_setting}). We make the convention that the matrix $\AA$ has max row norm equal to $1$.
Note we absorb some $n$--dependence into $\AA$ and $\bb$ by setting $\tfrac{1}{n}f_i(\xx) = \tfrac{1}{2}( \aa_i \xx - b_i)^2$. 
We investigate a generic class of stochastic momentum algorithms (see Alg.~\ref{alg:class} and Table~\ref{table:stochastic_algorithms}). Particularly, we introduce a sub-class, denoted by SDA$(\gamma_1, \gamma_2, \Delta)$, of Alg.~\ref{alg:class} which has parameters that are appropriately adjusted for large problems (large number of samples $n$ and large model size $d$); we refer to the \textit{dimension} of the problem as $n$. The class SDA$(\gamma_1, \gamma_2, \Delta)$ is defined by setting in Alg.~\ref{alg:class}
\begin{equation}\label{eq:SDA}
    \Gamma_1 = \frac{\gamma_1}{n}, \quad \Gamma_2 = \gamma_2, \quad \text{and} \quad \Delta(k) \defas \tfrac{1}{n}(\log \varphi)'(\tfrac{k}{n}),
\end{equation}
where $\gamma_1, \gamma_2 > 0$ are step sizes and $\varphi$ is a smooth function that represents a momentum schedule. 
Although we develop some theory for general $\varphi$, we are principally interested in the two cases:
\begin{equation} \begin{gathered}
(\text{SDANA})
\, \, 
\Delta(k) = \tfrac{\theta}{k+n}
\leftrightarrow
\varphi(t) = (1+t)^{\theta}
 \quad\text{and}\quad (\text{SDAHB}) \, \, 
 \Delta(k) = \tfrac{\theta}{n}
 \leftrightarrow
 \varphi(t)=e^{\theta t}.
\end{gathered} \end{equation}



\begin{wrapfigure}[33]{r}{0.55\textwidth}
 \centering
     \includegraphics[scale = 0.28]{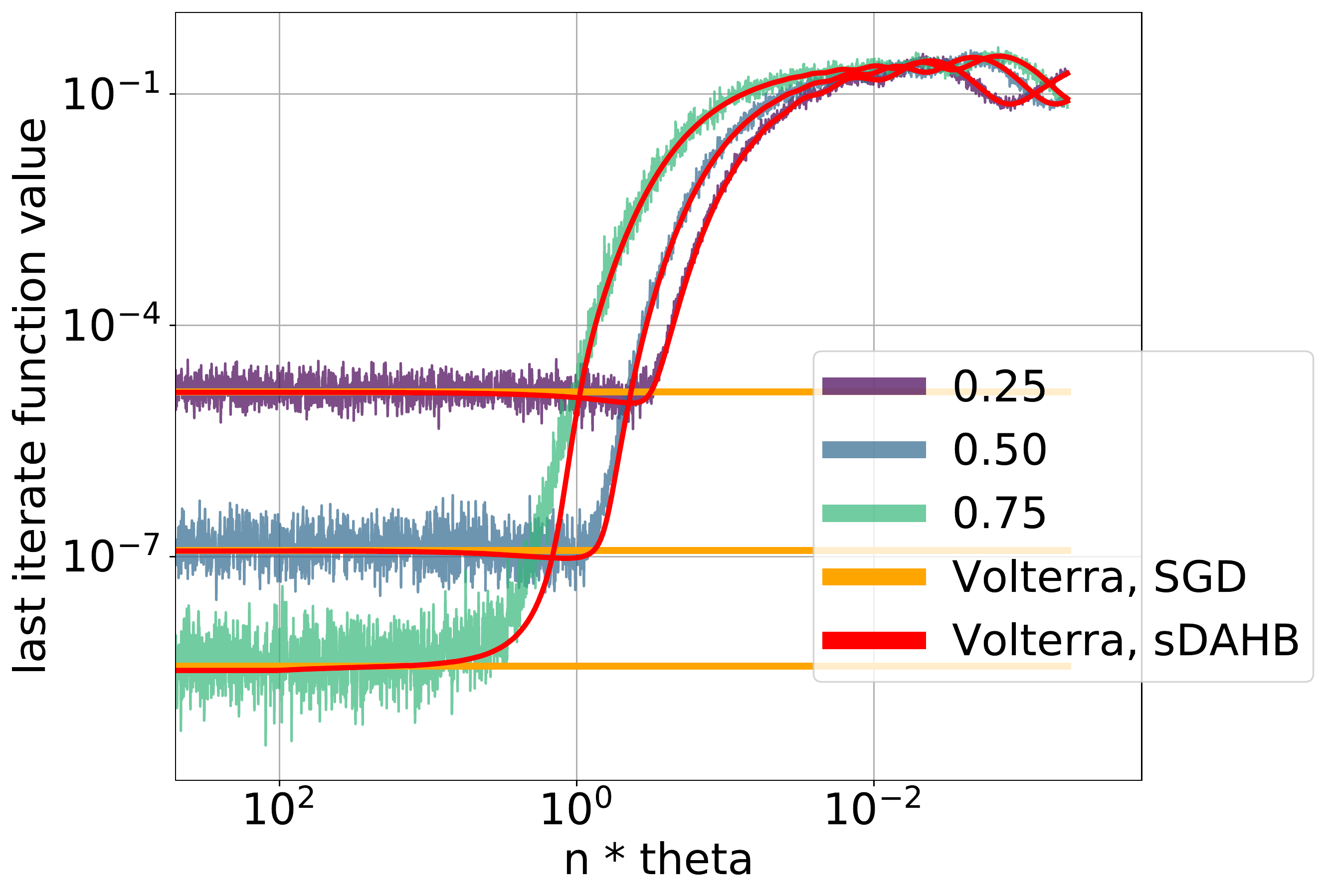}
     \vspace{-0.15cm}
     \caption{\textbf{Equivalence of SGD and stochastic Heavy-Ball.},
     For every $\gamma^{\text{sgd}} \in \{0.25,0.50,0.75\}$, we select a pair of parameters $(\gamma^{\text{shb}}, \theta^{\text{shb}})$ so that $\gamma^{\text{sgd}}=\tfrac{\gamma^{\text{shb}}}{\theta^{\text{shb}}}$. We run SHB 3000 times with varying $\theta^{\text{shb}}$ on $\eqref{eq:lsq}$ with $d = 500, n = 1000$, and plot the value of the last iterate after 50 epochs. Small $\theta^{\text{shb}}$ matches SGD (orange, theory), illustrating their equivalence. With $n \cdot \theta^{\text{shb}} \approx 1,$ a change is observed, giving a small improvement for some values of $\gamma^{\text{sgd}}$.  Plotted against theory for SDAHB (red Volterra, see Thm.~\ref{thm:hSGD} and App~\ref{sec:dahb}), which is the same algorithm as SHB after a change of parameters.}
     \label{fig:SHB_equals_SGD}
\end{wrapfigure}

To avoid confusion between different algorithms, we add superscripts indicating the algorithm (\textit{e.g.}, we denote $\Gamma_2 = \gamma^{\text{sgd}}$, the step size parameter for SGD). For all these algorithms, we are interested in:
\begin{enumerate}[leftmargin=*]
    \item An expression for the (deterministic) dynamics of these algorithms when \textit{multiple passes} on the data set are allowed. This contrasts with the "streaming" or "online" setting where at each iteration one generates an independent never-before-used data point.
    \item A formula for choosing the hyperparameters and a discussion of the dependence of these hyperparameters on number of features and samples.  
    \item Upper and lower bounds on the average-case complexity of the \textit{last iterate} to a neighborhood; this neighborhood disappears entirely in the overparameterized regime, while in the underparameterized regime the limiting distance to optimality concentrates in the high-dimensional limit.
\end{enumerate}

\renewcommand{\arraystretch}{1.1}
\ctable[notespar,
caption = {{\bfseries Summary of the parameters for a variety of stochastic momentum algorithms} that fit within the framework of Alg.~\ref{alg:class}, denote the normalized trace by $m \defas n^{-1}\sum_{i=1}^n \|\aa_i\|^2$. The default parameters are chosen so that its linear rate is no slower, by a factor of $4$ than the fastest possible rate for an algorithm having optimized over all step size choices. \vspace{-0.5em}
} ,label = {table:stochastic_algorithms},
captionskip=2ex,
pos =!t
]{c c c c c}{}{
\toprule
\multirow{2}{5em}{\textbf{Methods}} & \multicolumn{3}{c}{  \textbf{Alg.~\ref{alg:class} Parameters}}  & \multirow{2}{10em}{ \centering \textbf{Default Parameters}} \\
& $\Gamma_1$ & $\Gamma_2$ & $\Delta(k)$ &\\
\midrule
\begin{minipage}{0.38\textwidth} \begin{center} Stochastic gradient descent: \textbf{SGD$(\gamma)$}
\end{center}
\end{minipage} &   $0$& $ \displaystyle \gamma$ & $1$ & \begin{minipage}{0.2\textwidth} \begin{center} $\gamma = \frac{1}{m}$, \, 
{\footnotesize (Prop.~\ref{propG:sGDdefault})}
\end{center} \end{minipage}\\
\midrule
\begin{minipage}{0.38\textwidth} \begin{center} Stoch. gradient descent w/ momentum: \textbf{SHB$(\gamma, \theta)$}  \end{center}
\end{minipage} & 
$\displaystyle \gamma$ & $0$ & $\theta$ & (see Fig.~\ref{fig:SHB_equals_SGD})\\
\midrule
\begin{minipage}{0.38\textwidth} \begin{center}
Stoch. dimension-adjusted heavy-ball: \textbf{SDAHB$(\gamma, \theta)$} \\ \textcolor{purple}
{(This paper)}\\
 \end{center} \end{minipage}& $\displaystyle \frac{\gamma}{n}$ & $0$ & $\displaystyle \frac{\theta}{n}$ &\begin{minipage}{0.2\textwidth} \begin{center} $\gamma = \frac{\theta}{m}$, \, \, $\theta = 2$ {\footnotesize (Prop.~\ref{propG:default})} \end{center} \end{minipage}\\
\midrule
\begin{minipage}{0.38\textwidth} \begin{center}
Stoch. dimension-adjusted Nesterov's accel. method: \textbf{SDANA$(\gamma_1, \gamma_2, \theta)$}\\
\textcolor{purple}{(This paper)}\\
 \end{center} \end{minipage} & $\displaystyle \frac{\gamma_1}{n}$ & $\displaystyle \gamma_2$ & $\displaystyle \frac{\theta}{k+n}$ & \begin{minipage}{0.3\textwidth} \begin{center} $\gamma_1 = \tfrac{1}{4 m}$, \, \, $\gamma_2 = \frac{1}{m} $, \\ $\theta=4$ {\footnotesize (Cor.~\ref{corE:stepsize})} \end{center} \end{minipage}\\
 \bottomrule
}

\paragraph{Why divide by n? A negative result.} Throughout the literature, there are examples for which (small batch size) stochastic momentum methods such as SHB and stochastic Nesterov's accelerated method (SNAG) achieve performances equal to (or even worse) than small batch size SGD (see \textit{e.g.}, \cite{sebbouh2020almost,kidambi2018on,zhang2019which} for heavy-ball and \cite{Liu2020accelerating,assran2020on,zhang2019which} for Nesterov). We also observe this phenomenon (see \eqref{eq:diffusion_volterra} and App.~\ref{sec:SDAHBsGD}, Thm~\ref{thmG:degeneration}), and we illustrate this in Fig.~\ref{fig:SHB_equals_SGD}. The stochastic heavy-ball method for \textit{any} fixed step size and momentum parameters (Table~\ref{table:stochastic_algorithms}) has the \textit{exact} same dynamics as vanilla SGD, which is to say, by setting the step size parameter in SGD to be $\gamma^{\text{sgd}} = \frac{\gamma^{\text{shb}}}{\theta^{\text{shb}}}$, the two algorithms have the same loss values provided the number of samples is sufficiently large, \textit{i.e.}, $f(\xx_k^{\text{sgd}}) = f(\xx_k^{\text{shb}})$.  

As a consequence of this, the average-case complexity of SHB equals the last iterate complexity of SGD (This was observed in \citep{sebbouh2020almost} with an upper bound, but our result shows an exact equivalence between last iterate SGD and SHB). Although App.~\ref{sec:SDAHBsGD}, Thm~\ref{thmG:degeneration} (see also \eqref{eq:diffusion_volterra}) gives an unsatisfactory answer to stochastic heavy-ball with fixed $\theta$ and small batch-size, our analysis illuminates a path forward. Particularly, \textit{one must choose \textbf{dimension-dependent} parameters to achieve dynamics which differ from SGD.}

\begin{wrapfigure}[23]{r}{0.5\textwidth}
    \centering
    \vspace{-0.5cm}
        \includegraphics[scale = 0.265]{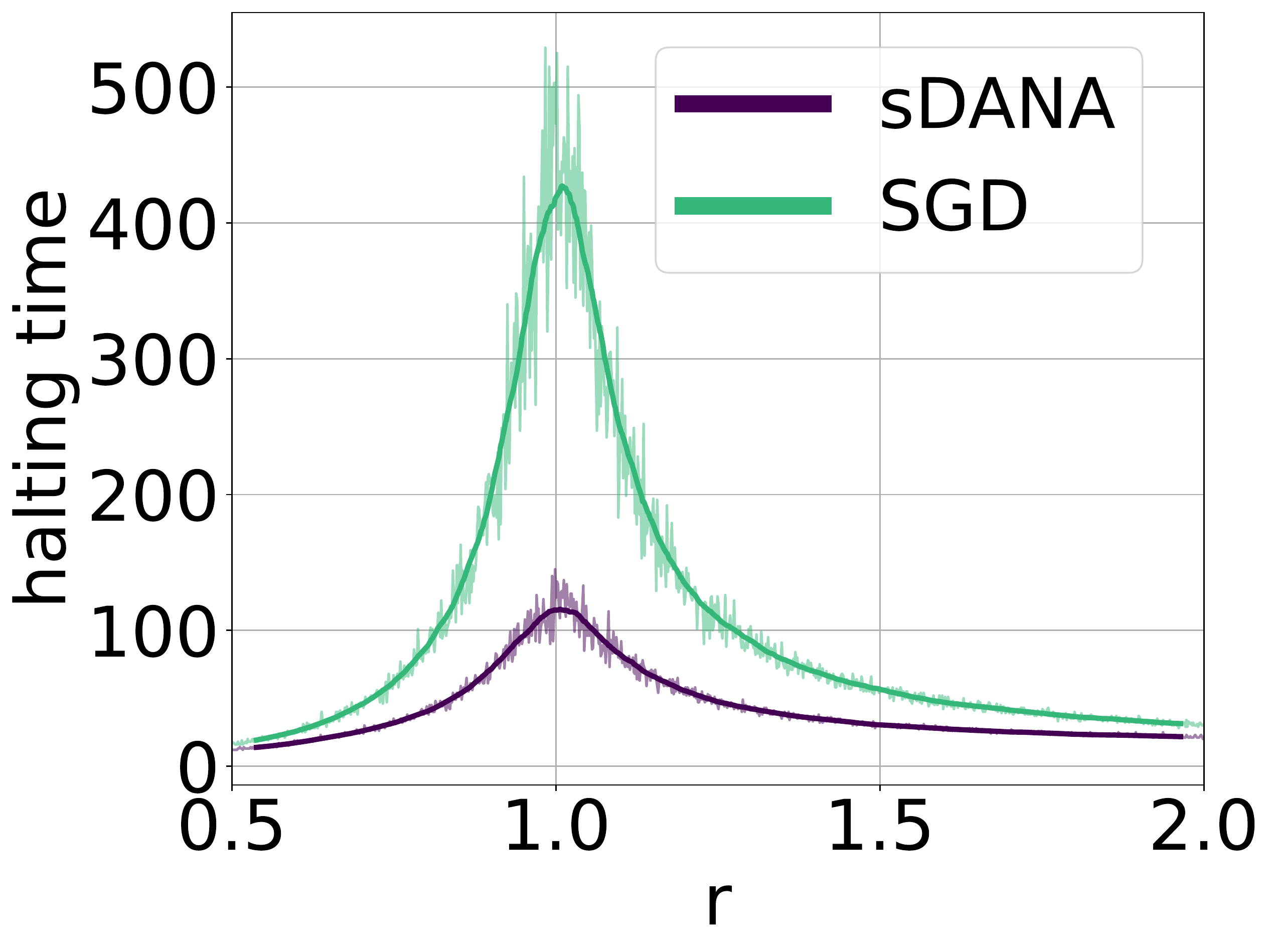} \vspace{-0.2cm}
	\caption{\textbf{Convergence of SDANA.} Halting time of SDANA vs SGD with default parameters on the Gaussian random least squares problem \eqref{eq:lsq} with varying $d$ and $n=1024$. When the ratio $r = d/n \to 1$ (in which case $\max\{\lambda_{\min}(\AA^T\AA),\lambda_{\min}(\AA\AA^T)\} \to 0$, SDANA requires significantly fewer iterations to reach a loss of $10^{-5}$. 
        As the ratio $r$ moves away from $1$, the performance of SDANA matches SGD.
	~
      }    \label{fig:SDANA_faster}
\end{wrapfigure} 

\paragraph{Why divide by n? A positive result.}
Adapting SHB for dimension, we arrive at stochastic dimension adjusted heavy ball (SDAHB).  While formally equivalent to SHB, we include the dimension parameters to emphasize that any improvement in its performance for large $n$ requires it.  Nonetheless, the speed-up for heavy ball is modest (see Fig.~ \ref{fig:SHB_equals_SGD}).

On the other hand, we show that a dimension adapted version of Nesterov acceleration, SDANA, has a large improvement in the non-strongly convex case. Moreover, with a simple parameter choice (see the default parameters in Table \ref{table:stochastic_algorithms}), it will perform linearly in the strongly convex case, and competitively with learning-rate-tuned SGD (or SHB), while performing orders-of-magnitude faster ($k^{-3}$ as compared to SGD $k^{-1}$) for the non-strongly convex setting (see Fig. \ref{fig:SDANA_faster} and Table~\ref{tab:complexity_main}).  We believe this gives SDANA promise as an algorithm outside of the least squares context, in situations in which loss landscapes can range between alternately curved and very flat, frequently observed in neural network settings (see \cite{ghorbani2019investigation,li2018measuring,sagun2016eigenvalues}).

\paragraph{Related work.} Recent works have established convergence guarantees for SHB in both strongly convex and non-strongly convex setting \citep{flammarion2015from,gadat2016stochastic,orvieto2019role,yan2018unified,sebbouh2020almost}; the latter references having established almost sure convergence results. Specializing to the setting of minimizing quadratics, the iterates of SHB converge linearly (but not in $L^2$) under an exactness assumption \citep{loizou2017momentum} while under some additional assumptions on the noise of the stochastic gradients, \citep{kidambi2018on,can2019accelerated} show linear convergence to a neighborhood of the solution. 

Convergence results for stochastic Nesterov's accelerated method (SNAG), under both strongly convex and non-strongly setting, have also been established. The works \citep{kulunchakov2019generic,assran2020on,aybat2018robust,can2019accelerated} showed that SNAG converged at the optimal accelerated rate to a neighborhood of the optimum. Under stronger assumptions, convergence to the optimum at an accelerated rate is guaranteed.   
Examples include the strong growth condition \citep{vaswani2019fast} and additive noise on the stochastic gradients \citep{laborde2019nesterov}. 

The lack of general convergence guarantees showing acceleration for existing momentum schemes, such as heavy-ball and NAG, in the stochastic setting, has led many authors to design alternative acceleration schemes \citep{ghadimi2012optimal,ghadimi2013optimal, allen2017katyusha, kidambi2018on,kulunchakov2019generic,Liu2020accelerating}.

\section{Random least squares problem}
\label{sec:main_problem_setting} 
To formalize the analysis of a high--dimensional, typical least squares problem, we define the \textit{random least squares problem}: 
 \begin{equation}\label{eq:lsq}
    \argminA_{\xx\in\mathbb{R}^d} \Big\{ f(\xx) =  \frac{1}{n} \sum_{i=1}^n f_i(\xx)  \overset{\mathrm{def}}{=} \frac{1}{2}\sum_{i=1}^n (\aa_i \xx - b_i)^2 \Big\} , \quad \text{with $\bb \defas \AA \widetilde{\xx} + \eeta$.}
\end{equation}
The data matrix $\AA$ is random
and we shall introduce assumptions on $\AA$ as they are needed, but we suggest as a central example the \emph{Gaussian random least squares} where each entry of $\AA$ is sampled independently from a standard normal distribution with variance $\tfrac{1}{d}$. We always make the assumption that each row $\aa_i \in \RR^{d\times 1}$ is centered and is normalized so that $\max_i~\{ \EE[\|\aa_i\|^2] \} = 1$.


As for the target $\bb = \AA \widetilde{\xx} + \eeta$, we assume it comes from a generative model corrupted by noise, where $\widetilde{\xx}$ is signal and $\eeta$ is noise. 
\begin{assumption}[Initialization, signal, and noise] \label{assumption: Vector} 
The initial vector $\xx_0 \in \RR^d$ is chosen so that $\xx_0-\widetilde{\xx}$ is independent of the matrix $\AA$.  The noise $\eeta$ is centered and has i.i.d. entries, independent of $\AA$.  The signal and noise are normalized so that
\[
\Exp \|\xx_0 - \widetilde{\xx}\|^2_2 = R\tfrac{d}{n}
\quad\text{and}\quad
\EE[\|\eeta\|_2^2] = \widetilde{R}.
\]
\end{assumption}
\noindent 

Note that deterministic $\xx_0-\widetilde{\xx}$ satisfies this assumption.  The vectors $\xx_0-\widetilde{\xx}$ and $\eeta$ arise as a result of preserving a constant signal-to-noise ratio in the generative model. Such generative models with this scaling have been used in numerous works \citep{mei2019generalization,hastie2019surprises,gerbelot2020asymptotic}. 

 For the data matrix $\AA$ we introduce the Hessian matrix $\widetilde{\HH} = \AA^T \AA$ and its symmetrization ${\HH} = \AA \AA^T$.  Let $\lambda_1 \ge \ldots \ge \lambda_n$  be the eigenvalues of the matrix ${\HH}$.  Up to appending zeros, this is the same ordered sequence of eigenvalues as those of the Hessian.  Define
the \emph{empirical spectral measure} (ESM) of $\HH$, $\mu_{\HH}$ by the formula
\begin{equation}\label{eq:ESM}
\int g(\lambda) \mu_{\HH}(\dif \lambda) \defas \frac{1}{n} \sum_{i=1}^n g( \lambda_i)
\quad\text{for any continuous function } g : \RR \to \RR.
\end{equation}
This gives the interpretation for the empirical spectral measure as the distribution of an eigenvalue of $\HH$ chosen uniformly at random.

\paragraph{Diffusion approximation.}
Our analysis will use a diffusion approximation to analyze the SDA($\gamma_1, \gamma_2, \Delta$) class of stochastic momentum methods (see \eqref{eq:SDA}) on the random least squares setup \eqref{eq:lsq}.  We call the approximation \emph{homogenized SGD}:
\begin{equation}\label{eqF:HSQD}
\dif \XX_t \defas
-\gamma_2 
\dif \ZZ_t
-\frac{\gamma_1}{\varphi(t)}
\int_0^t \varphi(s) \dif \ZZ_t,
\, \,
\text{where}
\, \,
\dif \ZZ_t \defas \nabla f(\XX_t) \dif t
+ \sqrt{\tfrac{2}{n}f(\XX_t)\nabla^2(f)}\dif \BB_t,
\end{equation}
and with initial conditions given by $\XX_0 = \xx_0$.  The process $(\BB_t : t \geq 0)$ is a $d$--dimensional standard Brownian motion.  Here time is scaled in such a way that $t=1$ represents one pass over the dataset or $n$ calls to the stochastic oracle.  Similar SDEs have appeared frequently in the theory around SGD, see \textit{e.g.}, \cite{mandt2016variational,li2017stochastic,li2019stochastic}.

The advantage of the homogenized SGD diffusion is that we are able to give an explicit representation of the expected loss values on a least squares problem, even at finite $n$.
\begin{theorem}[Volterra dynamics at finite $n$]\label{thm:hSGD}
Let $\Exp_{\HH}[ \cdot] $ be the conditional expectation where $\HH$ is held fixed.
    There are non-negative functions 
    $F(t)$ and $\mathcal{K}_s(t)$
    for $s,t \geq 0$ depending on the spectrum of $\HH$ so that
    for all $t \geq 0$
    \begin{equation} \label{eq:diffusion_volterra}
    \Exp_{\HH}[ f(\XX_t)]
    = F(t) + \int_0^t \mathcal{K}_s(t)   \Exp_{\HH}[ f(\XX_s)]\dif s,
    \quad\text{for all}\quad
    t \geq 0.
    \end{equation}
    The forcing function $F$ and kernel $\mathcal{K}$ are given by
    \[
    F(t) = \frac{1}{n}\sum_{i=1}^n (R\lambda_i + \widetilde{R}) G^{(\lambda_i)}(t)
    \quad\text{and}\quad
    \mathcal{K}_s(t) 
    =
    \frac{1}{n}\sum_{i=1}^n K^{(\lambda_i)}_s(t).
    \]
    The functions $G^{(\lambda)}$ and $K^{(\lambda)}$ are solutions of an initial value problem with a $3$-rd order ODE which depend on the hyperparameters $(\gamma_1,\gamma_2,\Delta)$ (Note, there is a 1-to-1 relationship with $\varphi$, see \eqref{eq:SDA}).
\end{theorem}
\noindent We refer to the supplemental materials for the explicit third--order ODE (see Theorem \ref{thm:bighSGD} for full details).  The expression in \eqref{eq:diffusion_volterra} is a Volterra integral equation, which can be analyzed explicitly, and has a relatively simple theory, especially in the case that the kernel is of convolution type (i.e.\ $\mathcal{K}_s(t) = \mathcal{I}(t-s)$ for some function $\mathcal{I}$); see Table~\ref{table:convolution_kernel} for kernels. We also note that in the case of SGD$({\gamma})$ ($\varphi$ is unused), the functions $G$ and $K$ become particularly simple
\[
G^{(\lambda)}(t) = e^{-2\gamma \lambda t}
\quad\text{and}\quad
K^{(\lambda)}_s(t) = \gamma^2 \lambda^2 e^{-2\gamma \lambda (t-s)}.
\]

\paragraph{Comparing homogenized SGD to the SDA class.} When $\AA$ is a random matrix, we can compare the diffusion \eqref{eqF:HSQD} to SDA \eqref{eq:SDA} when $n$ and $d$ are large.  The argument is based on the results of \cite{paquetteSGD2021}, and we do it only in the case of SGD:
\begin{theorem}[Concentration of SGD]\label{thm:SGDthm}
  Suppose that $\AA$ is a \emph{left-orthogonally invariant} random matrix, meaning that for any orthogonal matrix $\OO \in \mathbb{R}^{n \times n},$
    \(\OO \AA \law \AA\). 
Suppose further that the noise vector $\eeta$ is independent of $\AA$ and that it satisfies
\[
\quad \EE[\|\eeta\|_\infty^p] = \mathcal{O}( n^{\epsilon-p/2}) \quad \text{for any } \epsilon, p >0.
\]
Fix $\gamma < 2n (\tr \HH)^{-1}$, the convergence threshold of SGD($\gamma$).
There is an absolute constant $\varepsilon >0$ and a constant $c(T,\lambda_H^+)$ so that with $p =\min\{d,n\},$
\[
\Pr(
\sup_{0 \leq t \leq T} |\Exp_{\HH}[f(\XX_t)] - f(\xx_{[nt]})|
> c(T,\lambda_H^+) p^{-\varepsilon}~\vert~\lambda_H^+) \leq p^{-\varepsilon}.
\]
\end{theorem}

We expect that this theorem can be generalized, to include the entire SDA class.  We also expect that the orthogonal invariance assumption can be relaxed somewhat (for example to include classes of non--Gaussian isotropic features matrices), but not entirely: the left singular vectors need to have some degree of isotropy for the result to hold.  The numerical results show very good general agreement with theory and demonstrate the validity of the approximation: see Figures \ref{fig:concentration_SHB},  \ref{fig:SHB_equals_SGD}, and \ref{fig:concentration} as well as Figure~\ref{fig:MNIST} on real data.  Nonetheless, it is of great theoretical interest to establish the theorem in greater generality. We show a heuristic derivation in App.~\ref{sec:hSGDlsq}.

\renewcommand{\arraystretch}{1.1}
\ctable[notespar,
caption = {{\bfseries Summary of the convolution kernel for the Volterra equations} \eqref{eq:diffusion_volterra} for all algorithms considered. The convolution kernel (below) for SDANA is an approximation to the true kernel. The forcing terms $G^{(\lambda)}(t)$ for SGD and SDAHB are similar to the kernel whereas the forcing term for SDANA is defined only by solving a 3rd-order ODE.} ,label = {table:convolution_kernel},
captionskip=1ex,
pos =!t
]{c c c c}{ }{
\toprule
\textbf{Methods} & \textbf{Kernel}, $K^{(\lambda)}_s(t)$ \\
\midrule
\begin{minipage}{0.15\textwidth} \begin{center} \textbf{SGD$({\gamma})$}
\end{center}
\end{minipage}& $\gamma^2 \lambda^2 e^{-2\gamma \lambda (t-s)}$ &---\\
\midrule
\begin{minipage}{0.15\textwidth} \begin{center}
\textbf{SDAHB$(\gamma, \theta)$} \\ \textcolor{purple}
{(This paper)}\\
 \end{center} \end{minipage}& $\frac{2 \gamma^2 \lambda^2}{\omega} e^{-(t-s)\theta}(1-\cos((t-s) \sqrt{\omega}))$ &\begin{minipage}{0.2\textwidth} 
 $\omega = 4 \lambda \gamma -\theta^2$\end{minipage} \\
\midrule
\begin{minipage}{0.15\textwidth} \begin{center}
\textbf{SDANA$(\gamma_1, \gamma_2, \theta)$}\\
\textcolor{purple}{(This paper)}\\
 \end{center} \end{minipage} & \begin{minipage}{0.45\textwidth} \begin{center} $\frac{\lambda}{\omega} e^{-\lambda \gamma_2 (t-s)} \left (1-\cos((t-s)\sqrt{\lambda \omega} + \vartheta) \right )$ \end{center} \end{minipage} & \begin{minipage}{0.3\textwidth} \begin{center} $\tan(\vartheta) = \frac{(\omega-2\gamma_1) \sqrt{4 \gamma_1-\omega}}{(\omega-2\gamma_1)^2-2 \gamma_1^2}$\\ $\omega = 4 \gamma_1-\gamma_2^2 \lambda$ \end{center} \end{minipage}\\
 \bottomrule
}

\section{Main results} \label{sec:main_results}
In this section, and in light of the Thm.~\ref{thm:hSGD}, we outline how to use this Volterra equation \eqref{eq:diffusion_volterra} to produce average-case analysis, nearly optimal hyperparameters, and exact expressions for the neighborhood and convergence thresholds. For additional details, see Supplementary Materials. 

\subsection{Convolution Volterra convergence analysis: convergence threshold and neighborhood}

For all algorithms considered (SDANA, SDAHB, SHB, SGD), the Volterra equation in Theorem \ref{thm:hSGD} can be expressed in a simpler form, that is, as a \textit{convolution--type Volterra equation}
\begin{equation}\label{eqF:conv}
\Exp_{\HH}[f( \XX_t)] = F(t) + \int_0^t \mathcal{I}(t-s) \Exp_{\HH}[ f(\XX_s)]\dif s
\quad\text{for all}
\quad t \geq 0.
\end{equation}
The forcing function $F$ and the convolution kernel $\mathcal{I}$ are non-negative functions that depend on the spectrum of $\HH$ and SDA parameters (see Table~\ref{table:convolution_kernel} for the kernels of various algorithms).  In the case of SDANA, the kernel is in fact not a convolution Volterra equation, but it can be approximated by one so that it matches the non--convolution equation as $t \to \infty$. 

First, the forcing function $F$ will in all cases be bounded, and in fact it will converge as $t \to \infty$ to a deterministic value,
\begin{equation}\label{eqF:force}
F(t) \underset{t\to\infty}{\longrightarrow} 
{\frac{\widetilde{R}\mu_{\HH}(\{0\})}{2}}
=
\frac{\widetilde{R}\dim(\ker(\HH))}{2n}.
\end{equation}
Here $\mu_{\HH}$ is the empirical spectral measure \eqref{eq:ESM} which exists for even non-random matrices. 
It follows that the solution of \eqref{eqF:conv} remains bounded if the 
norm $\|\mathcal{I}\| = \int_0^\infty \mathcal{I}(t) \dif t$ is less than $1$.  
\begin{theorem}[Convergence threshold and limiting loss] \label{thmF:norm}
    If the norm $\|\mathcal{I}\| < 1$, the algorithm is convergent in that 
    \[
    \Exp_{\HH} [f(\XX_t)] \underset{t\to\infty}{\longrightarrow}
    \frac{\widetilde{R}\dim(\ker(\HH))}{2n(1-\|\mathcal{I}\|)} \qquad \text{(limiting loss)}.
    \]
\end{theorem}
\begin{wrapfigure}[21]{r}{0.4\textwidth}
 {\centering 
   \vspace{-1.0cm}
     \includegraphics[width = 0.9\linewidth]{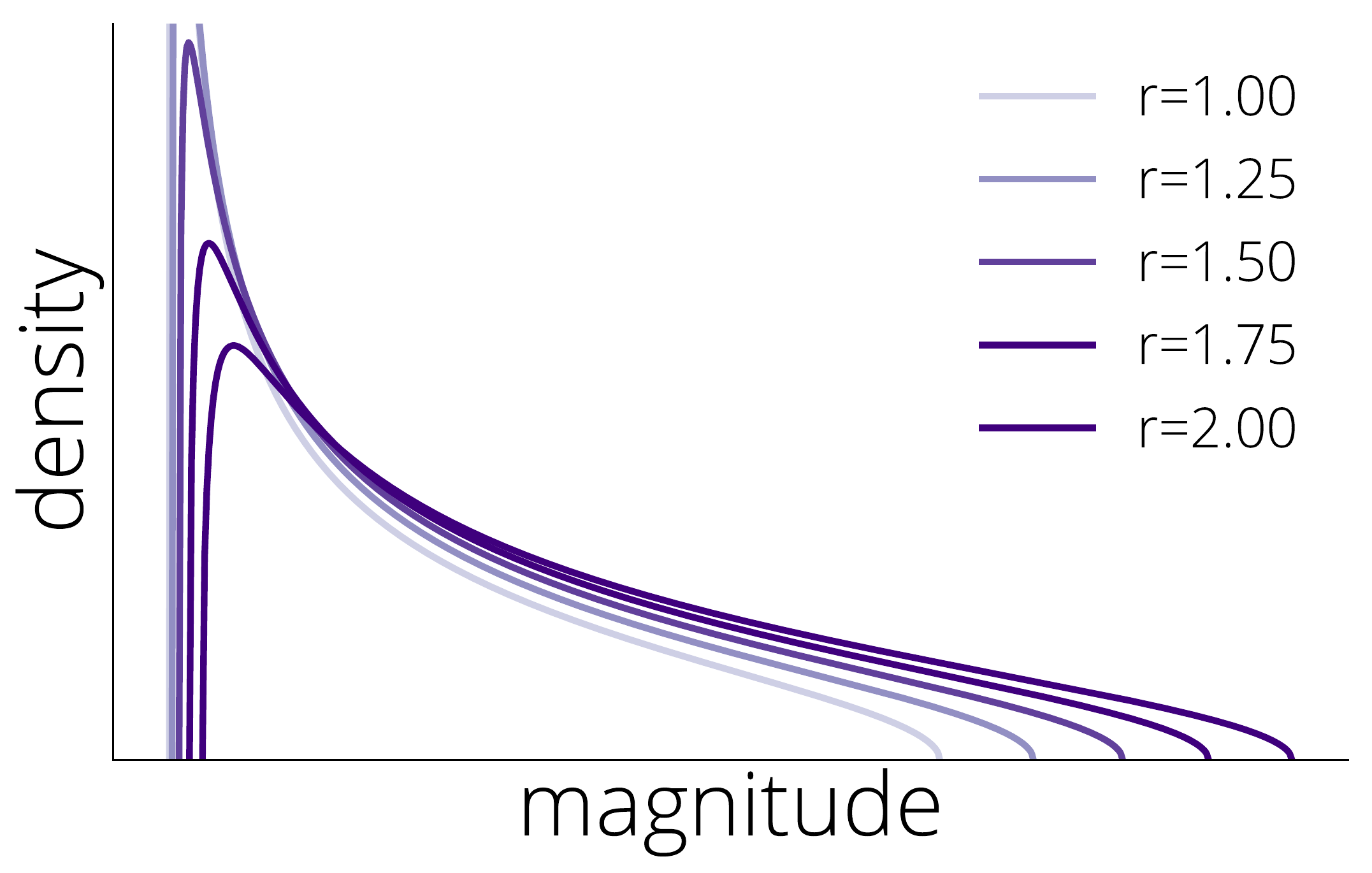}
   \vspace{-0.3cm}
     \caption{The \emph{Marchenko-Pastur law}($r$). Varying $r = d/n$.}}
    \label{fig:MP}
\end{wrapfigure}
This theorem gives a convergence threshold for all algorithms in Table~\ref{table:stochastic_algorithms} based only on the norm of the kernel of the Volterra equation, which is easily computable (see Table~\ref{tab:complexity_main}). 

\begin{table}[t!]
    \centering
        \caption{\textbf{Asymptotic average-case convergence guarantees} for $\Exp_{\HH}[f(\XX_t)] \vspace{-0em} - \frac{\widetilde{R}\dim(\ker(\HH))}{2n(1-\|\mathcal{I}\|)}$ (last iterate) under default parameters (see Table~\ref{table:stochastic_algorithms}) for the isotropic features model. The norm of the kernel is controlled by two values: the normalized trace of the matrix $m = \sum_{i=1}^n \|\aa_i\|^2 / n$ and the mass of the spectral measure (empirical or limiting) at $0$ which we denote by $p = \text{dim}(\text{ker}(\HH)) / n$. Average-case complexity is strictly better than the worst-case complexity, in some cases by a factor $\gamma$ vs. $\gamma^2$. As in \citet{paquetteSGD2021}, the worst-case rates in non-strongly convex setting have dimension dependent constants due to the distance to the optimum $\|\xx^\star-\xx_0\|^2 \approx d$ which appears in the bounds. SDANA obtains an accelerated average-case rate in the non-strongly convex case over SGD while matching the average-case rate of SGD in strongly convex regime. These rates are achieved without changing hyperparameters in SDANA. For worst-case rates, see \cite[Theorem 4.6]{ bottou2018optimization} \cite[Theorem 2.1]{ghadimi2013stochastic}; $\lambda^{+}$ can be replaced by the max-$\ell^2$-row-norm. }
    \label{tab:complexity_main}
    \begin{tabular}[]{lccc}
      \toprule
      & \textbf{Kernel, $ \bm \| \mathcal{I} \|$} & \begin{minipage}{0.27\textwidth} \begin{center} \textbf{Strongly convex} \end{center} \end{minipage} & \begin{minipage}{0.2\textwidth} \begin{center} \textbf{Non-strongly convex} \vspace{0.1em} \end{center} \end{minipage} \\
      \midrule        
      \textbf{SGD}$(\gamma)$ \quad 
      \begin{minipage}{0.055\textwidth} Worst\vspace{0.2em}
      \end{minipage}
      & 
      & 
      \begin{minipage}{0.3\textwidth} \begin{center} $\text{exp}(-\gamma t \lambda^- + \tfrac{\gamma^2}{2} (\lambda^+)^2 t)$ \vspace{0.2em} \end{center}\end{minipage} 
      &  \begin{minipage}{0.2\textwidth} \begin{center} $(R + \widetilde{R} \cdot \textcolor{purple}{d}) \cdot \frac{1}{t}$\vspace{0.2em}
	\end{center} \end{minipage} \\
      \textbf{SGD}$(\gamma)$ \quad 
	Avg
      &  \begin{minipage}{0.13\textwidth} \begin{center} 
	  $\frac{ \gamma}{2} m$ \\
	  {\footnotesize (Eq.~\eqref{q:sgd}) }
	  \end{center}\end{minipage}
      & \begin{minipage}{0.3\textwidth} \begin{center}
	  $\text{exp}(-\gamma t \lambda^-)$ \\
	  {\footnotesize (Lem.~\ref{q:default})}
	   \end{center}\end{minipage} 
      &  
      \begin{minipage}{0.2\textwidth} \begin{center}
	  $Rt^{-3/2} + \widetilde{R} t^{-1/2}$ 
	  \\
	  {(\footnotesize Eq.~\eqref{q:sgdht})}
	  \end{center}\end{minipage}
	  \\
      \midrule

      \textbf{SDAHB}$(\gamma, \theta)$ 
      & \begin{minipage}{0.13\textwidth} \begin{center} $\frac{\gamma}{2\theta} m$ \vspace{0.4em}\\
	  {\footnotesize (Eq.~\eqref{eqG:norm}) }\end{center} \end{minipage} 
      & \begin{minipage}{0.21 \textwidth} \begin{center}$\exp( -t \frac{\gamma \lambda^- \theta}{2\gamma \lambda^- + \theta^2} )$ \vspace{0.2em} \\ {\footnotesize (Prop.~\ref{propG:default}) } \vspace{0.4em} \end{center} \end{minipage} 
      & \begin{minipage}{0.2\textwidth} \begin{center}
	  $Rt^{-3/2} + \widetilde{R} t^{-1/2}$ 
	  \\
	  {(\footnotesize Eq.~\eqref{eqG:sdahb})}
	\end{center}\end{minipage}
      \\
      \midrule
      \textbf{SDANA}$(\gamma_1, \gamma_2, \theta)$ 
      & \begin{minipage}{0.2\textwidth} \begin{center}  $\frac{\gamma_1}{2 \gamma_2} (1-p) + \frac{\gamma_2}{2} m$\\ {\footnotesize (Eq. \eqref{eqE:convergence}) } \end{center} \end{minipage} 
      & \begin{minipage}{0.2\textwidth} \begin{center} $\exp( -t\frac{3 \gamma_1\gamma_2 \lambda^{-}}{2\gamma_2^2\lambda^{-} + 4{\gamma_1}})$\\ {\footnotesize (Cor.~\ref{corE:stepsize})} \vspace{0.4em} \end{center} \end{minipage} 
      & \begin{minipage}{0.2\textwidth} \begin{center}
	  $Rt^{-3} + \widetilde{R} t^{-1}$ 
	  \\
	  {(\footnotesize Prop.~\ref{prop:sqasymp})}
	\end{center}\end{minipage}
      \\
      \bottomrule
    \end{tabular}
\end{table}
\subsection{Average case analysis} 
\paragraph{Limiting spectral measures.} Average-case complexity looks at the typical behavior of an algorithm when some of its inputs are chosen at random.  
To formulate an average case analysis that is representative of what is seen in a large scale optimization problem, we will take a limit of the empirical spectral measure as $n$ and $d$ are taken to infinity.  So, we suppose that the following holds:
\begin{assumption}[Spectral limit] \label{assumption: spectral_density}
Let $\AA$ be an $n \times d$ matrix drawn from a family of random matrices such that the number of features, $d$, tends to infinity proportionally to the size of the data set, $n$, so that $\tfrac{d}{n} \to r \in (0, \infty)$; and suppose these random matrices satisfy the following. 
    
    1. The eigenvalue distribution of $\HH=\AA \AA^T$ converges to a deterministic limit $\mu$ with compact support.  Formally, the empirical spectral measure (ESM) converges weakly to $\mu$, in that for all bounded continuous $g : \RR \to \RR$
    \begin{equation} \label{eq:ESM_convergence}
    \frac{1}{n}\sum_{i=1}^n g(\lambda_i)
    \Prto[n]
    \int_0^\infty g(\lambda)\mu(\dif \lambda).
    \end{equation}
    2. The largest eigenvalue $\lambda_{\HH}^+$ of $\HH$ converges to the largest element $\lambda^+$ in the support of $\mu$, i.e. $\lambda_{\HH}^+ \Prto[n] \lambda^+.$
\end{assumption}

\begin{figure}[t!]
    \centering
    \includegraphics[scale = 0.16]{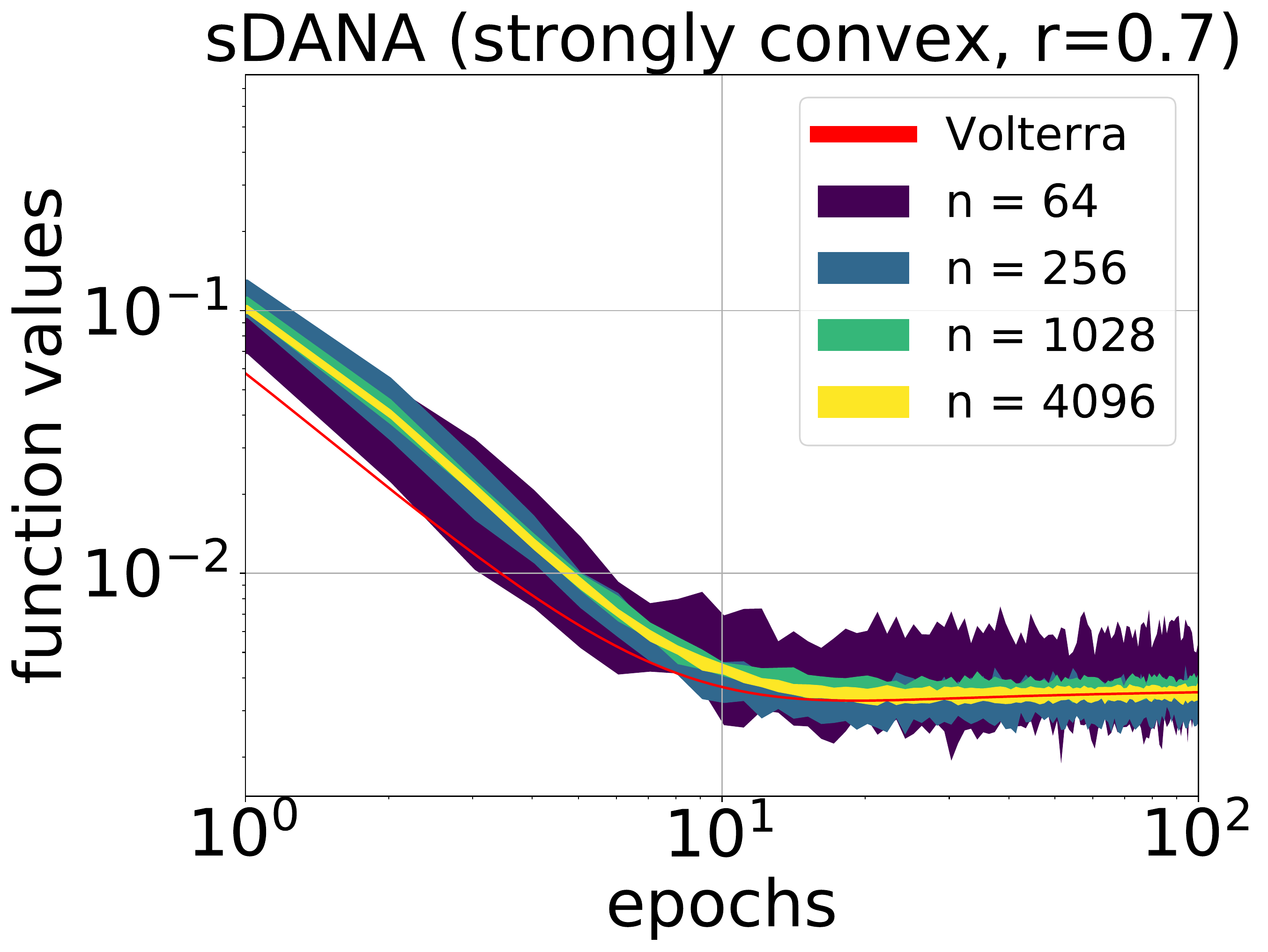}
    \includegraphics[scale = 0.16]{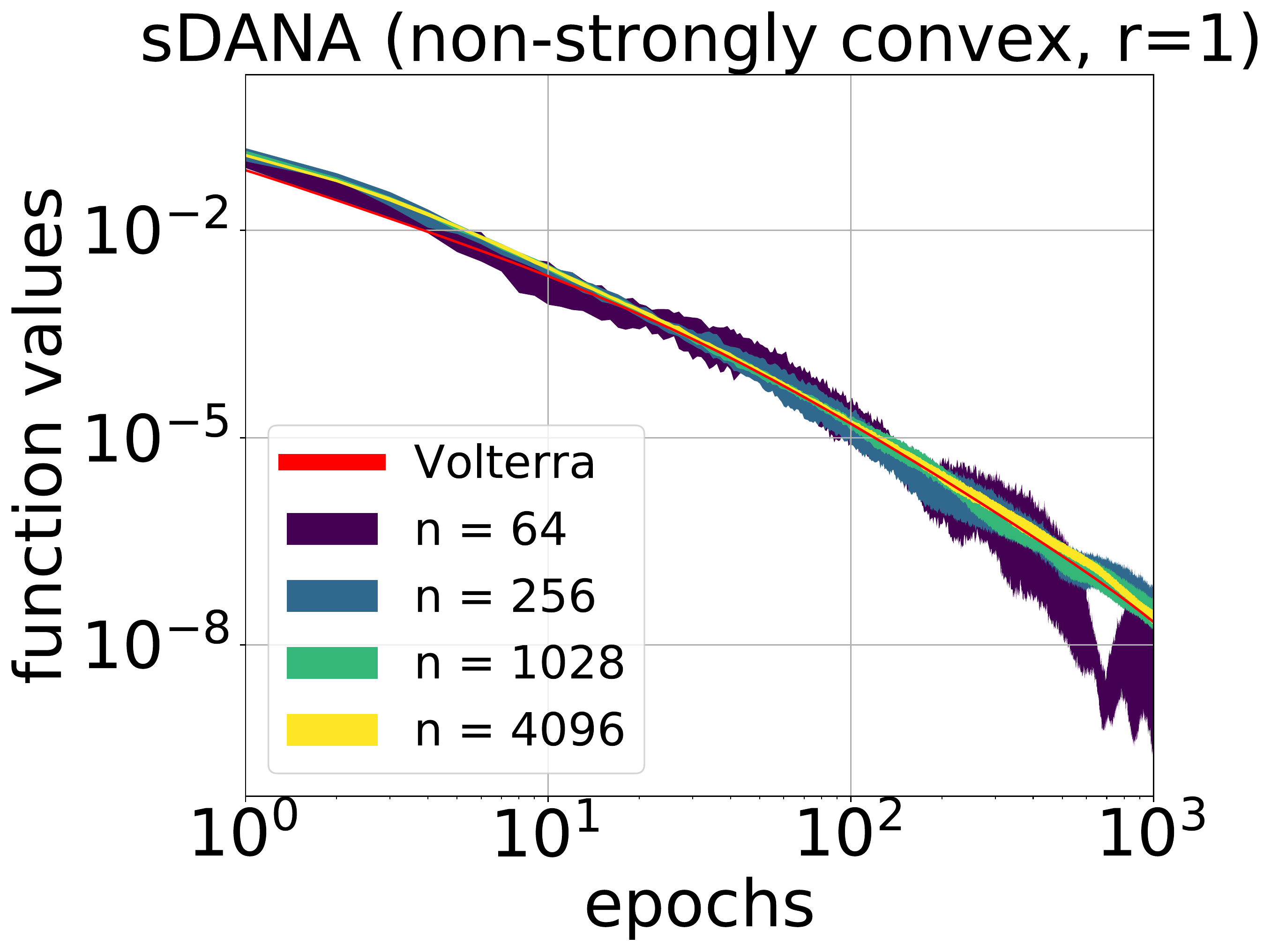}
    \includegraphics[scale = 0.16]{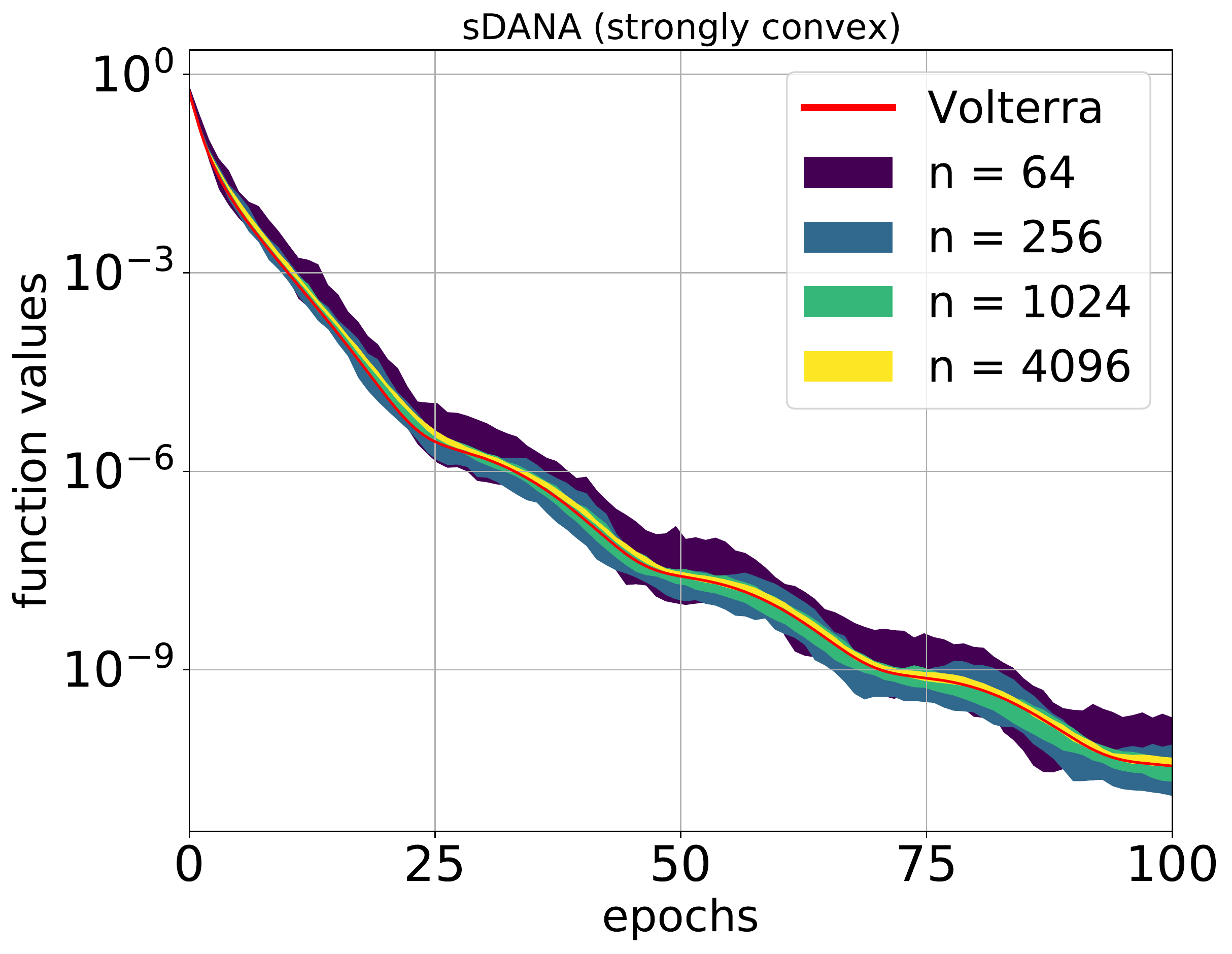}  
    \caption{\textbf{Concentration of SDANA.} 80\% confidence interval on 10 runs with default parameters on Gaussian random least squares problem \eqref{eq:lsq}, $d/n = r$, with noise $\widetilde{R}=0.01$ and signal $R = 1$. The convolution-type Volterra equation (red, \eqref{eqF:conv}) predicts the performance of SDANA and it reflects the oscillatory trajectories typically seen in momentum methods due to overshooting. Because the convolution Volterra is only an approximation to the kernel, there is always an initial mismatch between actual runs of SDANA and the Volterra solution. As $t \to \infty$, the convolution-type Volterra equation better approximates SDANA. For more details on numerical simulations see App.~\ref{sec: numerical_simulation}.}
    \label{fig:concentration}
\end{figure}

This assumption is typical in random matrix theory.  An important example is \emph{the isotropic features model}, which is a random $n \times d$ matrix $\AA$ whose every entry is sampled from a common, mean $0$, variance $\tfrac{1}{d}$ distribution with fourth moment $\mathcal{O}(d^{-2})$, such as a Gaussian $N(0,\tfrac{1}{d})$.  In this case, the ESM $\mu_{\HH}$ of $\HH = \AA\AA^T$ converges to the
 Marchenko-Pastur law (see Figure \ref{fig:MP}):
\begin{equation} \begin{gathered} \label{eq:MP} \dif \MP(\lambda) \defas \delta_0(\lambda) \max\{1-{r}, 0\} + \frac{r\sqrt{(\lambda-\lambda^-)(\lambda^+-\lambda)}}{2 \pi \lambda} 1_{[\lambda^-, \lambda^+]}\,,\\
\text{where} \qquad \lambda^- \defas (1 - \sqrt{\tfrac{1}{r}})^2 \quad \text{and} \quad \lambda^+ \defas (1+ \sqrt{\tfrac{1}{r}})^2\,.
\end{gathered} \end{equation}


More generally, the convergence of the spectral measure of matrices drawn from a consistent ensemble is well studied in random matrix theory, and for many random matrix ensembles the limiting spectral measure is known. In the machine learning literature, it has been shown that the spectrum of the Hessians of neural networks share characteristics with the limiting spectral distributions found in classical random matrix theory \citep{dauphin2014identifying, papyan2018the, sagun2016eigenvalues, behrooz2019investigation,martin2018implicit,pennington2017geometry, liao2020Random, granziol2020learning}.

\paragraph{Complexity analysis.} The forcing function and the convolution kernel both converge under Assumption \ref{assumption: spectral_density}, and the result is that
\[
\lim_{n \to \infty} \Exp_{\HH}[f(\XX_t)] = \psi(t)
\quad\text{where}
\quad
\psi(t) = F_\mu(t) + \int_0^\infty \mathcal{I}_\mu(t-s) \psi(s) \,\dif s
\quad\text{for all}
\quad t \geq 0.
\]
The forcing function and interaction kernel are given as integrals against the limit measure (such as \eqref{eq:MP} in the case of isotropic features) and
\[
F_\mu(t) \defas \int_0^\infty 
(R\lambda +\widetilde{R}) G^{(\lambda)}(t) \mu( \dif \lambda )
\quad\text{and}\quad
\mathcal{I}_\mu(t) 
\defas \int_0^\infty 
K^{(\lambda)}(t) \mu(\dif \lambda).
\]
The kernel norm still determines the convergence properties of the Volterra equation.  In particular, to have convergence of the algorithm, we need that $\|\mathcal{I}_\mu\| < 1$ and just like in the finite-$n$ case (see Lemma \ref{q:limit})
\[
\psi(t)
\underset{t \to \infty}{\longrightarrow}
\psi(\infty)\defas
\frac{\widetilde{R}\mu(\{0\})}{2(1-\|\mathcal{I}_\mu\|)}.
\]

To discuss average-case rates, we use the function $\psi(t)$. We consider separately the regimes when the problem is strongly convex and not.
Having taken the limit, we say the problem is strongly convex if the intersection of the support of $\mu$ with $(0,\infty)$ is closed.  Intuitively, this says there is a gap between $0$ and the next smallest eigenvalue $\lambda^-$ of the hessian (strictly speaking it allows a vanishing fraction of the eigenvalues to approach $0$). 

\begin{wrapfigure}[36]{r}{0.5\textwidth}
 \centering
 \vspace{-1.1cm}
     \includegraphics[scale = 0.23]{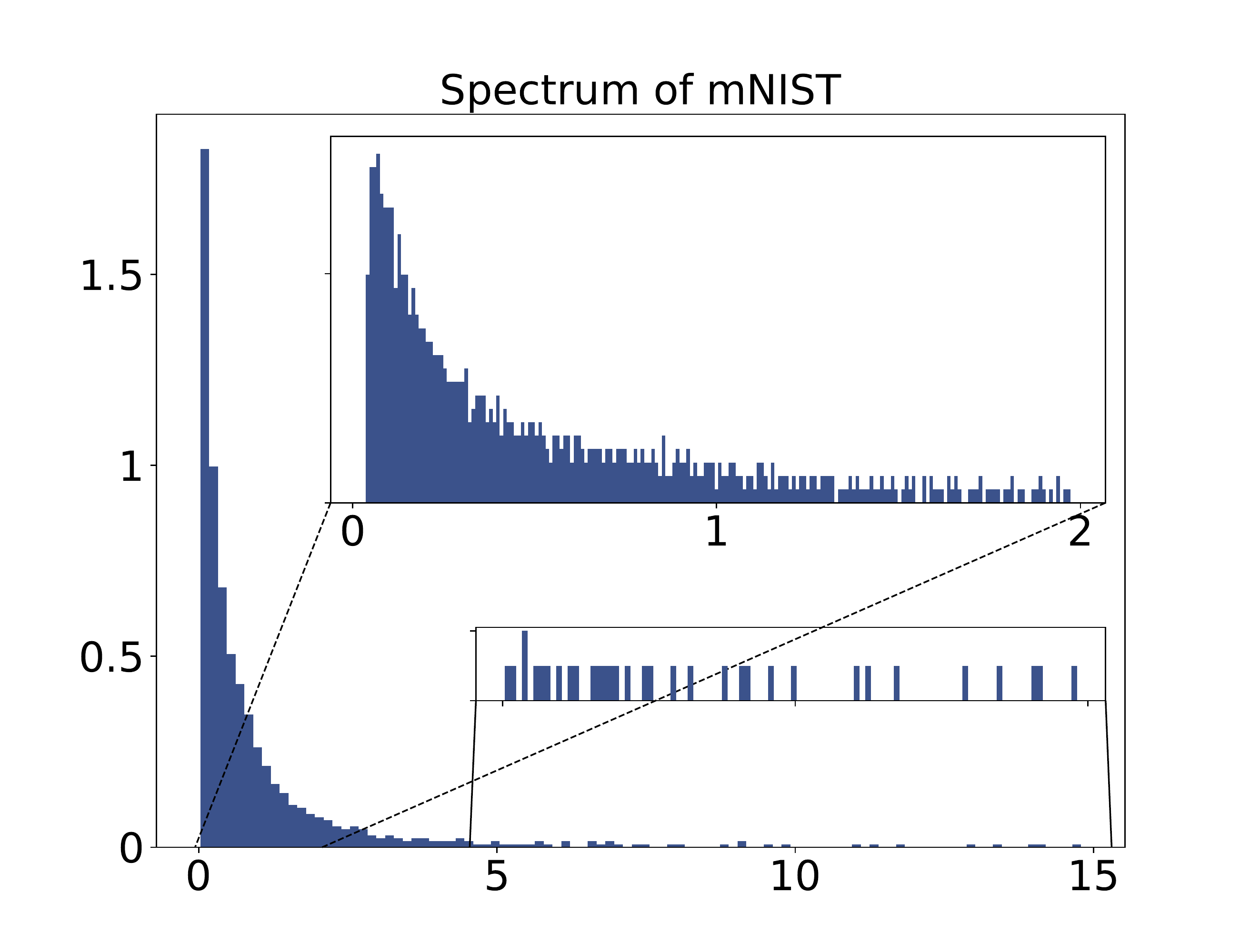}\vspace{-0.5cm}
     \includegraphics[scale = 0.23]{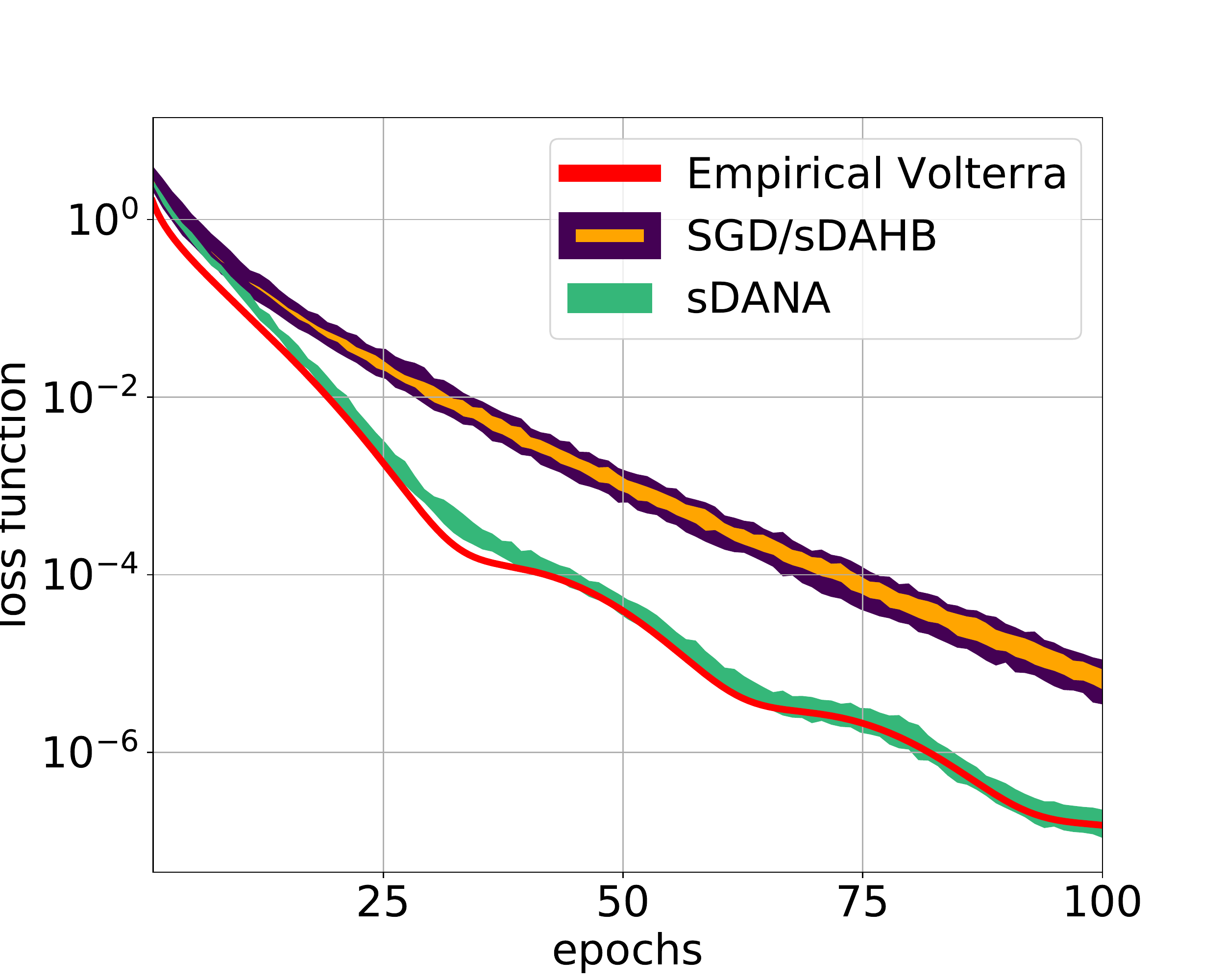}
     \vspace{-0.1cm}
     \caption{\textbf{SDANA \& SGD vs Theory on MNIST.} 
     MNIST ($60000\times 28 \times 28$ images) \citep{lecun2010mnist} is reshaped into $10$ matrices of dimension $1000\times 4704$, representing $1000$ samples of groups of $6$ digits (preconditioned to have centered rows of norm-1).  First digit of each 6 is chosen to be the target $\bb$. Algorithms were run 10 times with default parameters (without tuning) to solve \eqref{eq:lsq}. 80\%--confidence interval is displayed. Volterra (SDANA) is generated with eigenvalues from the first MNIST data matrix (top pane, $\lambda^- = 0.041$). Volterra predicts the convergent behavior of SDANA in this non-idealized setting. SDANA outperforms equivalent SGD/SDAHB. See also Appendix \ref{sec: numerical_simulation}.
     }
     \label{fig:MNIST}
\end{wrapfigure}

In the non-strongly convex case the average-case complexity is relatively simple to compute. The rate of convergence of $\psi(t)-\psi(\infty) \to 0$ 
is only determined by the rate of of convergence of the forcing function 
$F_\mu(t)-F_\mu(\infty)$.
This decays like $t^{-\beta}$ where $\beta$ in turn is controlled by the exponent $\alpha$ at which $\mu( (0,\varepsilon] ) \asymp
\epsilon^\alpha$ as $\varepsilon \to 0$ (see Lemma \ref{q:rvrate}). Particular if there are more small eigenvalues, the rate is slowed.  In Table~\ref{tab:complexity_main}, the rates are reported for $\alpha = \tfrac 12$. This is the typical behavior for random matrix distributions with a ``hard-edge,'' such as Marchenko--Pastur with aspect ratio $r=1$.

In the strongly convex case, $\lambda^{-} > 0,$ the kernel $\mathcal{I}_\mu$ plays a larger role, in that it may slow down the convergence rate.  In particular, if it exists, we define the \emph{Malthusian exponent} $\lambda^*$ as the solution of 
\[
\int_0^\infty e^{\lambda^* t} \mathcal{I}_\mu(t) = 1.
\]

The rate of convergence of $\psi$ to $0$ (at exponential scale) will then be the slower of $F_\mu(t)$ and $e^{-\lambda^* t}$ (see Lemma \ref{q:errate}).  In the case of SGD on Marchenko--Pastur, the exact value of $\lambda^*$ is worked out in exact form in \cite{paquetteSGD2021}.  By bounding these Malthusian exponents, we produce the rate guarantees in Table \ref{tab:complexity_main} (see Cor.~\ref{corE:stepsize} and Prop.~\ref{propG:default} for the bounds in the Appendix). The \textbf{default parameters} are chosen so that its linear rate is no slower, by a factor of $4$ than the fastest possible rate for an algorithm having optimized over all step size choices. This is achieved by lower bounding the Malthusian exponent at the default parameters and upper bounding the optimal rate by minimizing $F_\mu(t)-F_\mu(\infty)$ over all convergent parameters.

\paragraph{Conclusions from the analysis of homogenized SGD.} 
The SHB algorithm is a special instance of SDAHB with parameter choices
\(
n \theta^{\text{shb}} = \theta^{\text{sdahb}}
\quad\text{and}\quad
n\gamma^{\text{shb}}
=
\gamma^{\text{sdahb}}.
\)
By evaluating the kernels for SDAHB in the large $n$ limit, it is easily seen that the homogenized SGD equations for $\XX^{\text{shb}}_t$ and $\XX^{\text{sgd}}_t$ satisfy
\begin{equation}\label{eq:shbsgd}
|\Exp_{\HH}[ f(\XX^{\text{shb}}_t)] -\Exp_{\HH}[f(\XX^{\text{sgd}}_t)]| 
\underset{n \to \infty}{\longrightarrow} 0,
\end{equation}
see (Thm.~\ref{thmG:degeneration}).  On the other hand SDAHB with default parameters is always strictly faster (for sufficiently large $\theta$) than tuned SGD, but its linear rate is never more than a factor of $2$ faster than SGD (Prop.~\ref{prop:SDAHB}).  It also does not substantially improve over SGD in non-strongly convex case. In contrast, the dimension adjusted Nesterov acceleration (SDANA) greatly (and provably, using homogenized SGD) improves over SGD in  non-strongly convex case (see Prop.~\ref{prop:sqasymp}), while remaining linear (and nearly as fast as SGD) in the convex case (Cor.~\ref{corE:stepsize}).  Furthermore, the predictions of homogenized SGD are born out even on real data (Fig.~\ref{fig:MNIST}), which is a non-idealized setting that does not verify the assumptions we imposed for the theoretical analysis.

\paragraph{Future directions.} We would like to explore the applicability of homogenized SGD to other datasets and other convex losses as well as generalizing the theoretical setting under which homogenized SGD applies (see the discussion below Thm.~\ref{thm:SGDthm}). Moreover, we would like to test and extend SDANA to non-convex problems and extend homogenized SGD to non-convex settings.

\paragraph{Funding Transparency Statement.} C. Paquette's research was supported by CIFAR AI Chair, MILA. Research by E. Paquette was supported by a Discovery Grant from the Natural Science and Engineering Council (NSERC). Additional revenues related to this work: C. Paquette has part-time employment at Google Research, Brain Team, Montreal, QC. 

\bibliographystyle{plainnat}
\bibliography{reference}

\newpage
\appendix
\begin{center}
\LARGE{Dynamics of Stochastic Momentum Methods on Large-scale, Quadratic Models}\\
\vspace{0.5em}\Large{Supplementary material \vspace{0.5em}}

\end{center}
The appendix is organized into five sections as follows:
\begin{enumerate}
    \item Appendix~\ref{sec:hSGD} derives the Volterra equation and proves the main result for the homogenized SGD (Theorem~\ref{thm:hSGD}).
    \item We show in Appendix~\ref{sec:hSGDlsq} a heuristic derivation of the homogenized SGD approximation to the SDA class of algorithms on the least squares problem and we show that SGD and homogenized SGD are close under orthogonal invariance (Theorem~\ref{thm:SGDthm}).
    \item We give in Appendix~\ref{sec:rates} a general overview of the analysis of a convolution Volterra equation of the type that arises in the SDA class.
    \item Appendix~\ref{sec:SDANA} details the analysis of the homogenized SGD for SDANA, including average-case analysis and near optimal parameters.
    \item Appendix~\ref{sec:dahb} has the details showing equivalence of SDAHB with SHB as well as general average-case complexity and parameter selections. 
    \item Appendix~\ref{sec: numerical_simulation} contains details on the simulations.
\end{enumerate}
Unless otherwise stated, all the results hold under Assumptions~\ref{assumption: Vector} and \ref{assumption: spectral_density}. We include all statements from the previous sections for clarity. 

\paragraph{Potential societal impacts.} The results presented in this paper concern the analysis of existing methods and a new method that is a variant of an existing method. The results are theoretical and we do not anticipate any direct ethical and societal issues. We believe the results will be used by machine learning practitioners and we encourage them to use it to build a more just, prosperous world.  

\section{Analysis of the Homogenized SGD evolution}
\label{sec:hSGD}

\subsection{Homogenized SGD}

We recall that the diffusion model is given by
\[
\dif \XX_t = 
-\gamma_2 
\dif \ZZ_t
-\frac{\gamma_1}{\varphi(t)}
\int_0^t \varphi(s) \dif \ZZ_t,
\quad
\text{where}
\quad
\dif \ZZ_t = \nabla f(\XX_t) \dif t
+ \sqrt{\tfrac{2}{n}f(\XX_t)\nabla^2(f)}\dif \BB_t.
\]
To connect these diffusions to SGD on the least squares problem \eqref{eq:lsq}
\[
f(\xx)
= \frac{1}{2}\|\AA \xx - \bb\|^2,
\]
we will use the singular value decomposition of $\UU \SSigma \VV^T$ of $\AA$.  We order the singular values $\sigma_1 \geq \sigma_2 \geq \sigma_3 \cdots$ in decreasing order.  We then let $\nnu_t = \VV^T (\XX_t -\widetilde{\xx})$, where we recall that $\bb = \AA \widetilde{\xx} +\eeta$.  We recall that 
\[
\nabla f( \XX_t)
= \AA^T (\AA \XX_t - \bb)
\quad
\text{and}
\quad
\nabla^2 f = \AA^T \AA.
\]
Hence, we may change the basis to write
\[
  \begin{aligned}
    &\dif (\VV^T \XX_t) = 
    -\gamma_2 
    \dif (\VV^T \ZZ_t)
    -\frac{\gamma_1}{\varphi(t)}
    \int_0^t \varphi(s) \dif (\VV^T \ZZ_t), \\
    &\dif (\VV^T \ZZ_t) = \SSigma^T (\SSigma \nnu_t- \UU^T \bb)\dif t +\sqrt{\tfrac{2}{n}f(\XX_t) \SSigma^T \SSigma} \dif (\VV^T \BB_t).
  \end{aligned}
\]
The loss values we may also represent in terms of $\nnu$
\[
f(\XX_t)
= 
\frac{1}{2}
\| \AA \XX_t - \bb\|^2
=
\frac{1}{2}
\| \SSigma \nnu_t - \UU^T\eeta\|^2
=
\frac{1}{2}\sum_{j=1}^d
( \sigma_j \nu_{t,j} - (\UU^T \bb)_j)^2
.
\]
We let $\dif \WW_t = \sqrt{\tfrac{2}{n}f(\XX_t) \SSigma^T \SSigma} \dif (\VV^T \BB_t),$ so that $\{W_{t,j} : t \geq 0, j \in 1,2,\dots, n\}$ are a family of continuous martingales with quadratic variation
\begin{equation}\label{eqF:W}
\dif \langle W_{t,j},W_{t,i}\rangle
= \delta_{i,j} \frac{2\sigma_j^2}{n} f(\XX_t) \dif t
\end{equation}
for all $i$ and $j$, with $1 \leq i,j \leq n$.
Finally, we conclude that
\begin{equation} \begin{aligned} \label{eq:dnu}
    \dif \nu_{t,j} = -\gamma_2 \dif \xi_{t,j} - \frac{\gamma_1}{\varphi(t)}\int_0^t \varphi(s) \dif \xi_{s,j}
    \quad\text{where}
    \quad
    \dif \xi_{t,j}
    \coloneqq
    \dif W_{t,j} + \sigma_j^2 \bigl(\nu_{t,j}-\tfrac{(\UU^T\eeta)_j}{\sigma_j}\bigr)\dif t.
\end{aligned}
\end{equation}
As in \eqref{eqF:W}, the quadratic variation of $W_{t,j}$ and $\xi_{t,j}$ is 
\begin{equation}
    \dif \, \langle \xi_{t,j} \rangle 
    =
    \gamma_2^{-2}\dif \, \langle \nu_{t,j} \rangle 
    =
    \dif \, \langle W_{t,j} \rangle 
    = \frac{2 \sigma_j^2 f(\XX_t)}{n}\dif t.
\end{equation}

\subsection{Mean behavior of the homogenized SGD}
We derive a description for the mean of the loss values $\Exp_{\HH} f(\XX_t)$.
We define the following functions of time
\begin{equation}\label{eqE:JL} \begin{gathered}
J \defas \Exp_{\HH}\biggl[ \bigl(\nu_{t,j}-\tfrac{(\UU^T\eeta)_j}{\sigma_j}\bigr)^2\biggr]
\quad\text{and}\quad
N \defas \sigma_j^{-2}\Exp_{\HH} \biggl[ 
\biggl(\int_0^t \varphi(s) d\xi_{s,j}\biggr)^2
-
\int_0^t \varphi^2(s)\dif \langle \xi_{s,j}\rangle
\biggr].
\end{gathered}
\end{equation}
We will compute the derivatives of these expressions in time.
Using It\^o's rule, 
\begin{equation} \begin{aligned} \label{eqE:dnu_squared}
    \dif \, \bigl(\nu_{t,j}-\tfrac{(\UU^T\eeta)_j}{\sigma_j}\bigr)^2 
    &= 2 \bigl(\nu_{t,j}-\tfrac{(\UU^T\eeta)_j}{\sigma_j}\bigr) \, \dif \nu_{t,j} + \dif \langle \nu_{t,j} \rangle \\
    &= 2 \bigl(\nu_{t,j}-\tfrac{(\UU^T\eeta)_j}{\sigma_j}\bigr) \biggl(-\gamma_2 \dif \xi_{t,j} - \frac{\gamma_1}{\varphi(t)}\int_0^t \varphi(s) \dif \xi_{s,j} \biggr)
    + \gamma_2^2\dif \langle \xi_{t,j} \rangle.
\end{aligned} \end{equation}
Since $\dif W_{t,j}$ is a martingale increment, the expectation of the $\dif \xi_{t,j}$ term simplifies.  We may do a similar computation with $N$ and conclude that:
\[
\begin{aligned}
&J^{(1)} 
=-2\gamma_2 \sigma_j^2 J
-\frac{2\gamma_1}{\varphi(t)}
\Exp_{\HH} \biggl[
\bigl(\nu_{t,j}-\tfrac{(\UU^T\eeta)_j}{\sigma_j}\bigr)
\int_0^t \varphi(s) \dif \xi_{s,j}
\biggr]+ \gamma_2^2
\Exp_{\HH}\dif \langle \xi_{t,j}\rangle,
\\
&N^{(1)}
=\sigma_j^{-2}\Exp_{\HH}\biggl[
2\varphi(t)\dif \xi_{t,j} \int_0^t \varphi(s) \dif \xi_{s,j}
\biggr]
=\Exp_{\HH}\biggl[
2\varphi(t)\bigl(\nu_{t,j}-\tfrac{(\UU^T\eeta)_j}{\sigma_j}\bigr) \int_0^t \varphi(s) \dif \xi_{s,j}
\biggr]
\end{aligned}
\]
In summary, we may express $J$ in terms of $N$ by
\begin{equation} \begin{gathered} \label{eqE:J}
J^{(1)} = -2 \gamma_2 \sigma_j^2 J - \frac{\gamma_1}{\varphi^2(t)} N^{(1)}
+ \gamma_2^2\dif \langle \xi_{t,j}\rangle
\quad
\text{with} \quad J(0) = \EE_{\HH}\biggl[ \bigl(\nu_{0,j}-\tfrac{(\UU^T\eeta)_j}{\sigma_j}\bigr)^2\biggr].
\end{gathered}
\end{equation}
Now we write a differential equation for $N$, using the product rule for stochastic calculus, and conclude
\[
\begin{aligned}
\varphi\tfrac{\dif}{\dif t}{\bigl(N^{(1)}/\varphi\bigr)}
=
&-
2\varphi(t)
\Exp_{\HH}\biggl[
\biggl( \gamma_2 \dif \xi_{t,j} + \frac{\gamma_1}{\varphi(t)}\int_0^t \varphi(s) \dif \xi_{s,j} \biggr) \int_0^t \varphi(s) \dif \xi_{s,j}
\biggr] \\
&+
2\varphi(t)
\Exp_{\HH}\biggl[
\biggl( \nu_{t,j}-\tfrac{(\UU^T\eeta)_j}{\sigma_j} \biggr) \varphi(t) \dif \xi_{t,j}
\biggr] 
+
2\varphi^2(t) 
\Exp_{\HH}\langle \dif \nu_{t,j}, \dif \xi_{t,j}\rangle \\
&=
-\gamma_2\sigma_j^2 N^{(1)}
-2\gamma_1 \bigl( \sigma_j^2  N+\int_0^t \varphi^2 \dif \langle \xi_{t,j}\rangle\bigr)
+2\varphi^2(t) \sigma_j^2 J - 2\gamma_2\varphi^2(t) 
\Exp_{\HH}\dif \langle \xi_{t,j}\rangle,
\end{aligned}
\]
with initial conditions $N(0)=N^{(1)}(0)=0$.
We will use 
\begin{equation}\label{eqE:hat}
\widehat{J} = \varphi^2 J / \gamma_1
\quad\text{and}\quad\widehat{\psi} = 2 \sigma^2_j \varphi^2 \Exp_{\HH} f(\XX_t)/n = \varphi^2 \Exp_{\HH} \dif \langle \xi_{t,j}\rangle.
\end{equation}
From these definitions we can also record, by evaluating the previous displayed equation at $0$ that $N^{(2)}(0)=2\gamma_1\sigma_j^2 \widehat{J}(0) - 2\gamma_2\widehat{\psi}(0)$.
The $\widehat{J}$ can be expressed as
\[
\widehat{J} \bigl( -2 \Phi + 2\gamma_2 \sigma_j^2\bigr) + \widehat{J}^{(1)} = -N^{(1)}
+ \tfrac{\gamma_2^2}{\gamma_1} \widehat{\psi} \qquad \text{where \qquad $\Phi \defas \frac{\varphi'(t)}{\varphi(t)}$.}
\]
We then differentiate the display equation  above to produce
\[
N^{(3)}
+N^{(2)}\bigl(-\Phi + \gamma_2 \sigma_j^2\bigr)
+N^{(1)}\bigl(-\Phi' + 2\gamma_1 \sigma_j^2\bigr)
-2\gamma_1 \sigma_j^2 \widehat{J}^{(1)}
=
-2\gamma_1 \widehat{\psi}
-2\gamma_2 \widehat{\psi}^{(1)}.
\]
On substituting $\widehat{J}$, we arrive at the third-order differential equation
\begin{equation}\label{eqE:hatJ_genera}
    \begin{aligned}
    &\widehat{J}^{(3)} + \left (  -3\Phi + 3 \gamma_2 \sigma_j^2 \right ) \widehat{J}^{(2)} + \left ( -5\Phi^{(1)} + 2 \Phi^2 -4 \gamma_2 \sigma_j^2\Phi +4\gamma_1\sigma_j^2 + 2 \gamma_2^2 \sigma_j^4 \right ) \widehat{J}^{(1)} \\
    &+ \left ( -2\Phi^{(2)}+4\Phi \Phi^{(1)} - 4\gamma_2\sigma_j^2 \Phi^{(1)} - 4 \gamma_1 \sigma_j^2 \Phi + 4 \gamma_1 \gamma_2 \sigma_j^4 \right ) \widehat{J}\\
    &=
    \tfrac{\gamma_2^2}{\gamma_1}
    \widehat{\psi}^{(2)}
    +\bigl(
    2\gamma_2
    + 
    \bigl(-\Phi + \gamma_2\sigma_j^2\bigr)\tfrac{\gamma_2^2}{\gamma_1}
    \bigr)
    \widehat{\psi}^{(1)}
    +\bigl(
    2\gamma_1
    + 
    \bigl( -\Phi^{(1)} + 2\gamma_1\sigma_j^2\bigr)\tfrac{\gamma_2^2}{\gamma_1}
    \bigr)
    \widehat{\psi}.
    \end{aligned}
\end{equation}
The initial conditions are given by
\begin{equation}\label{eqE:JIC_general}
\begin{gathered}
\widehat{J}(0) = \gamma_1^{-1}\EE\biggl[ \bigl(\nu_{0,j}-\tfrac{(\UU^T\eeta)_j}{\sigma_j}\bigr)^2\biggr],
\quad
 \widehat{J}^{(1)}(0) =
 \tfrac{\gamma_2^2}{\gamma_1} \widehat{\psi}(0)
-\widehat{J}(0) \bigl( -2\Phi(0) + 2\gamma_2 \sigma_j^2\bigr),
\quad\text{and} \\
\widehat{J}^{(2)}(0)
 =  \tfrac{\gamma_2^2}{\gamma_1} \widehat{\psi}^{(1)}(0)+ 2\gamma_2 \widehat{\psi}(0) - 2 \gamma_1 \sigma_j^2 \widehat{J}(0) + ( 2\Phi(0) - 2 \gamma_2 \sigma_j^2) \widehat{J}^{(1)}(0) + 2 
 \Phi^{(1)}(0) \widehat{J}(0).
\end{gathered}
\end{equation}

\paragraph{Two special cases for $\Delta(k,n)$.} In this section, we record the ODE for two special cases of the function $\varphi(t)$. When $\Delta(k,n) = \frac{\theta}{k+n}$ and thus $\varphi(t) = (1+t)^{\theta}$ with $\Phi(t) = \frac{\theta}{1+t}$, the corresponding ODE is precisely  
\begin{equation}
    \begin{aligned} \label{eq:DE_SDANA}
    &\widehat{J}^{(3)} - \left ( \tfrac{3 \theta}{(1+t)} - 3 \gamma_2 \sigma_j^2 \right ) \widehat{J}^{(2)} - \left ( -\tfrac{5 \theta + 2 \theta^2}{(1+t)^2} + \tfrac{4 \gamma_2 \sigma_j^2 \theta}{ (1+t)} -4\gamma_1\sigma_j^2 - 2 \gamma_2^2 \sigma_j^4 \right ) \widehat{J}^{(1)} \\
    &- \left ( \tfrac{4 \theta + 4 \theta^2}{ (1+t)^3} - \tfrac{4 \gamma_2 \sigma_j^2 \theta }{ (1+t)^2} + \tfrac{4 \gamma_1 \sigma_j^2 \theta}{ (1+t) } - 4 \gamma_1 \gamma_2 \sigma_j^4 \right ) \widehat{J}\\
    &=
    \tfrac{\gamma_2^2}{\gamma_1}
    \widehat{\psi}^{(2)}
    +\bigl(
    2\gamma_2
    + 
    \bigl(\tfrac{-\theta}{(1+t)} + \gamma_2\sigma_j^2\bigr)\tfrac{\gamma_2^2}{\gamma_1}
    \bigr)
    \widehat{\psi}^{(1)}
    +\bigl(
    2\gamma_1
    + 
    \bigl(\tfrac{\theta}{(1+t)^2} + 2\gamma_1\sigma_j^2\bigr)\tfrac{\gamma_2^2}{\gamma_1}
    \bigr)
    \widehat{\psi}.
    \end{aligned}
\end{equation}
and the initial conditions are given by
\begin{equation}
\begin{gathered}
\widehat{J}(0) = \gamma_1^{-1}\EE\biggl[ \bigl(\nu_{0,j}-\tfrac{(\UU^T\eeta)_j}{\sigma_j}\bigr)^2\biggr],
\quad
 \widehat{J}^{(1)}(0) =
 \tfrac{\gamma_2^2}{\gamma_1} \widehat{\psi}(0)
-\widehat{J}(0) \bigl( -2\theta + 2\gamma_2 \sigma_j^2\bigr),
\quad\text{and} \\
\widehat{J}^{(2)}(0)
 =  \tfrac{\gamma_2^2}{\gamma_1} \widehat{\psi}^{(1)}(0)+ 2\gamma_2 \widehat{\psi}(0) - 2 \gamma_1 \sigma_j^2 \widehat{J}(0) + ( 2 \theta - 2 \gamma_2 \sigma_j^2) \widehat{J}^{(1)}(0) - 2\theta \widehat{J}(0).
\end{gathered}
\end{equation}
The other case is when $\Delta(k,n) = \frac{\theta}{n}$, or $\varphi(t) = \exp(\theta t)$. We call this the general SDAHB; one recovers SDAHB when $\gamma_1 = \gamma, \gamma_2 = 0, \, \text{and} \, \theta = \theta$. In this setting, the log-derivative $\Phi(t) = \alpha$ and the ODE reduces to 
\begin{equation}
    \begin{aligned} \label{eq:JEQ_SHB}
    &\widehat{J}^{(3)} + \left (  -3\theta + 3 \gamma_2 \sigma_j^2 \right ) \widehat{J}^{(2)} + \left ( 2 \theta^2 -4 \gamma_2 \sigma_j^2\theta +4\gamma_1\sigma_j^2 + 2 \gamma_2^2 \sigma_j^4 \right ) \widehat{J}^{(1)} \\
    &+ \left ( - 4 \gamma_1 \sigma_j^2 \theta + 4 \gamma_1 \gamma_2 \sigma_j^4 \right ) \widehat{J}\\
    &=
    \tfrac{\gamma_2^2}{\gamma_1}
    \widehat{\psi}^{(2)}
    +\bigl(
    2\gamma_2
    + 
    \bigl(-\theta + \gamma_2\sigma_j^2\bigr)\tfrac{\gamma_2^2}{\gamma_1}
    \bigr)
    \widehat{\psi}^{(1)}
    +\bigl(
    2\gamma_1
    +  2\gamma_1\sigma_j^2 \tfrac{\gamma_2^2}{\gamma_1}
    \bigr)
    \widehat{\psi}.
    \end{aligned}
\end{equation}
The initial conditions are given by
\begin{equation}
\begin{gathered}
\widehat{J}(0) = \gamma_1^{-1}\EE\biggl[ \bigl(\nu_{0,j}-\tfrac{(\UU^T\eeta)_j}{\sigma_j}\bigr)^2\biggr],
\quad
 \widehat{J}^{(1)}(0) =
 \tfrac{\gamma_2^2}{\gamma_1} \widehat{\psi}(0)
-\widehat{J}(0) \bigl( -2\theta + 2\gamma_2 \sigma_j^2\bigr),
\quad\text{and} \\
\widehat{J}^{(2)}(0)
 =  \tfrac{\gamma_2^2}{\gamma_1} \widehat{\psi}^{(1)}(0)+ 2\gamma_2 \widehat{\psi}(0) - 2 \gamma_1 \sigma_j^2 \widehat{J}(0) + ( 2\theta - 2 \gamma_2 \sigma_j^2) \widehat{J}^{(1)}(0).
\end{gathered}
\end{equation}
We note that the ODE in \eqref{eq:JEQ_SHB} is constant coefficient and therefore can be solved by finding the characteristic polynomial, that is,
\begin{align*}
    0 &= \lambda^3 + (3 \gamma_2 \sigma^2_j - 3 \theta) \lambda^2 + (2\theta^2 -4 \gamma_2 \sigma_j^2 \theta + 4 \gamma_1 \sigma_j^2 + 2 \gamma_2^2 \sigma_j^4) \lambda + 4\gamma_1 \gamma_2 \sigma_j^4 - 4\gamma_1 \sigma_j^2 \theta, \\
    0 &= (\lambda + \sigma_j^2 \gamma_2 - \theta) (\lambda^2 + (2\sigma_j^2 \gamma_2 - 2 \theta) \lambda + 4 \sigma_j^2 \gamma_1 ), \\
    \lambda &= \theta - \sigma_j^2 \gamma_2 \quad \text{and} \quad \lambda = -(\sigma_j^2 \gamma_2 - \theta) \pm \sqrt{(\sigma_j^2 \gamma_2-\theta)^2-4 \sigma_j^2 \gamma_1}.
\end{align*}

\subsection{Inhomogeneous IVP in \eqref{eqE:hatJ_genera}} We simplify the problem in \eqref{eqE:hatJ_genera} by considering the inhomogeneous ODE 
\begin{equation} \begin{gathered} \label{eq:inhomogeneous_ODE}
    L[\widehat{J}]\coloneqq \widehat{J}^{(3)} + p(t) \widehat{J}^{(2)} + q(t) \widehat{J}^{(1)}+ r(t) \widehat{J} = C \widehat{\psi}^{(2)} + f(t) \widehat{\psi}^{(1)} + g(t) \widehat{\psi}
    \eqqcolon R[\widehat{\psi}],
\end{gathered} \end{equation}
where $L[\widehat{J}]$ and $R[\widehat{\psi}]$ are differential operators.  Let $J_0(t)$ be the solution to the homogeneous ODE in \eqref{eq:inhomogeneous_ODE} (\textit{i.e.} $L[J_0] = 0$) with initial conditions given by $J_0(0) = d_0$, $J_0^{(1)}(0) = d_1$ and $J_0^{(2)}(0) = d_2$. We let $\widehat{K}_s(t)$ solve
\begin{equation} \begin{gathered}
\text{$\widehat{K}_s(t) = 0$ for $t < s$}, \,
    \text{ $L[\widehat{K}_s(t)] = 0$ for $t \neq s$}, \, 
    \text{and $\widehat{K}_s(s) = c_0$,\, $\widehat{K}_s^{(1)}(s)=c_1$, \, $\widehat{K}_s^{(2)}(s) = c_2$.}
\end{gathered} \end{equation}
Here the initial conditions are chosen so that $L[\widehat{K}_s(t)] = R^*[\delta_s(t)],$ with $R^*$ the adjoint differential operator, \textit{i.e.},
\begin{align*}
    L \big [ \int_0^\infty \widehat{K}_s(t) \widehat{\psi}(s) \, \dif s \big ](t) = \int_0^\infty L[\widehat{K}_s(t)] \widehat{\psi}(s) \, \dif s =  \int_0^\infty R^*[\delta_s(t)] \widehat{\psi}(s) \, \dif s = R[\widehat{\psi}](t).
\end{align*}
We now just need to determine the initial conditions $c_0$, $c_1$, and $c_2$. First, we define $H_s(t)$ to be the Heaviside function with a jump at $s$ and note the following classical results for derivatives of $H_s(t)$:
\begin{align*}
    \partial_t \big ( \tfrac{(t-s)^2}{2} H_s(t) \big ) = (t-s) H_s(t), \quad &\partial_t ( (t-s) H_s(t) ) = H_s(t)\\
    \partial_t H_s(t) = \delta_s(t), \quad &\partial^2_t H_s(t) = \delta_s'(t).
\end{align*}
We now define the following operator where the derivatives are taken with respect to $t$
\begin{align*} R^*[\delta_s](t) &\defas C \delta_s''(t) + f(t) \delta_s'(t) + g(t) \delta_s(t)\\
\int_0^\infty R^*[\delta_s(t)] \widehat{\psi}(s) \, \dif s &= \frac{d}{dt^2} \int_0^\infty C \delta_s(t) \widehat{\psi}(s) \dif s + f(t) \frac{d}{dt} \int_0^\infty \delta_s(t) \widehat{\psi}(s) \dif s + g(t) \int_0^\infty \delta_s(t) \widehat{\psi}(s) \dif s\\
&=C \widehat{\psi}^{(2)}(t) + f(t) \widehat{\psi}^{(1)}(t) + g(t) \widehat{\psi}(t) = F(t).
\end{align*}
We now decompose $\widehat{K}_s(t) = \widehat{K}_s^1(t) + \widehat{K}_s^2(t) + \widehat{K}_s^3(t)$ and find initial conditions for each of these terms separately, that is, we will find
\begin{align*}
    L[\widehat{K}_s^1(t)] = C \delta_s''(t), \quad L[\widehat{K}_s^2(t)] = f(t) \delta_s'(t), \quad \text{and} \quad L[\widehat{K}_s^3(t)] = g(t) \delta_s(t).
\end{align*}
We recall for clarity that $f(t) \delta_s'(t) = f(s) \delta_s'(t) - f'(s) \delta_s(t)$. We can write $\widehat{K}_s^i = \widetilde{K}_s^i + \widetilde{H}_s^i$ where $\widetilde{K}_s^i$ is $0$ at $s$ and $C^2$ and
\[\widetilde{H}_s^i \defas H_s^i(t) \big (c_0^i + c_1^i(t-s) + c_2^i \tfrac{(t-s)^2}{2} \big ).\]
It follows that $L[\widetilde{H}_s^1] = C \delta_s''(t) + \{\text{continuous functions on $t \ge s$}\}$. To find $c_0^1, c_1^1$, and $c_2^1$, we see that
\begin{align*}
    L[\widetilde{H}_s^1] &= c_0^1 \delta_s''(t) + c_1^1 \delta_s'(t) + c_2^1 \delta_s(t) + p(t) [c_0^1 \delta_s'(t) + c_1^1 \delta_s(t) + c_2^1 H_s(t) ]\\
    &\qquad + q(t) \big [ c_0^1 \delta_s(t) + c_1^1 H_s(t) + c_2^1 \int H_s(t) \dif t \big ]\\
    &\qquad + r(t) \big [ c_0^1 H_s(t) + c_1^1 \int H_s(t) \dif t + c_2^1 \int \int H_s(t) \dif t \dif t \big ]\\
    &= c_0^1 \delta_s''(t) + c_1^1 \delta_s'(t) + c_2^1 \delta_s(t) + c_0^1 \big ( p(s) \delta_s'(t) - p'(s) \delta_s(t) \big ) + c_1^1 p(s) \delta_s(t)\\
    & \qquad + c_0^1 q(s) \delta_s(t) + \text{continuous terms}   
\end{align*}
As we want $L[\widetilde{H}_s^1] = C \delta_s''(t)$, then we need to solve the system
\begin{align*}
    \begin{pmatrix} C \\ 0 \\ 0 \end{pmatrix} = \begin{pmatrix} 
    1 & 0 & 0\\
    p(s) & 1 & 0\\
    q(s) - p'(s) & p(s) & 1
    \end{pmatrix} \begin{pmatrix}
    c_0^1\\
    c_1^1\\
    c_2^1
    \end{pmatrix}.
\end{align*}
We can know solve this system to get that
\begin{equation} \label{eq:IC_1}
c_0^1 = C, \quad c_1^1 = -C p(s), \quad \text{and} \quad c_2^1 = Cp^2(s) - C(q(s) - p'(s)).  
\end{equation}
Next we want to solve $L[\widetilde{H}_s^2] = f(t) \delta_s'(t) = f(s) \delta_s'(t) - f'(s) \delta_s(t)$. Using a similar argument as before, we deduce that
\begin{align*}
    \begin{pmatrix} 0 \\ f(s) \\ -f'(s) \end{pmatrix} = \begin{pmatrix} 
    1 & 0 & 0\\
    p(s) & 1 & 0\\
    q(s) - p'(s) & p(s) & 1
    \end{pmatrix} \begin{pmatrix}
    c_0^2\\
    c_1^2\\
    c_2^2
    \end{pmatrix}.
\end{align*}
Solving this system,
\begin{equation}\label{eq:IC_2}
    c_0^2 = 0, \quad c_1^2 = f(s), \quad \text{and} \quad c_2^2 = -f'(s) - p(s) f(s).
\end{equation}
Lastly we want to solve $L[\widetilde{H}_s^3] = g(s) \delta_s(t)$ or equivalently,
\begin{align*}
    \begin{pmatrix} 0 \\ 0 \\ g(s) \end{pmatrix} = \begin{pmatrix} 
    1 & 0 & 0\\
    p(s) & 1 & 0\\
    q(s) - p'(s) & p(s) & 1
    \end{pmatrix} \begin{pmatrix}
    c_0^3\\
    c_1^3\\
    c_2^3
    \end{pmatrix},
\end{align*}
that is 
\begin{equation}
    c_0^3 = 0, \quad c_1^3 = 0, \quad \text{and} \quad c_2^3 = g(s).
\end{equation}
Putting this all together, we need to solve for $\widehat{K}_s(t)$ such that
\begin{equation} \begin{gathered} \label{eq:general_kernal}
 L[\widehat{K}_s(t)] = 0\\
\text{where} \quad \widehat{K}_s(s) = C, \quad \widehat{K}_s'(s) = f(s)-Cp(s),\\
\text{and} \quad \widehat{K}_s''(s) = Cp^2(s) - C(q(s)-p'(s)) -f'(s)-p(s) f(s) + g(s) 
\end{gathered} 
\end{equation}

\begin{proposition}[Kernel representation, general] \label{prop:kernel_general} Consider the inhomogeneous ODE in \eqref{eqE:hatJ_genera}. Let $\widehat{K}_s(t)$ and $\widehat{J}_0(t)$ solve the homogeneous ODE in \eqref{eqE:hatJ_genera}, that is, 
\begin{enumerate}
    \item $L[\widehat{K}_s(t)] = 0 \quad \text{\rm where} \quad \widehat{K}_s(s) = \frac{\gamma_2^2}{\gamma_1}, \quad \widehat{K}_s'(s) = 2 \gamma_2 + \frac{2\gamma_2^2}{\gamma_1}\left (\Phi(s)-\gamma_2 \sigma_j^2 \right ), \text{\rm and}\\ \quad \widehat{K}_s''(s) = 2(\gamma_1 + 3 \gamma_2 \Phi(s) - 4\gamma_2^2 \sigma_j^2) + \frac{2 \gamma_2^2}{\gamma_1} \big [ \Phi^{(1)}(s) + 2\Phi^2(s) - 4 \gamma_2 \sigma_j^2 \Phi(s) + 2 \gamma_2^2 \sigma_j^4 \big ] $.
    \item $L[\widehat{J}_0(t)] = 0 \, \, \text{\rm where} \, \, \widehat{J}_0(0) = \frac{1}{\gamma_1} \mathbb{E}\big [ \big (\nu_{0,j}- \frac{(\UU^T\eeta)_j}{\sigma_j} \big )^2 \big ], \, \, \widehat{J}^{(1)}_0(0) = (2 \Phi(0)-2\gamma_2 \sigma_j^2) \widehat{J}_0(0)\\
    \text{\rm and} \quad \widehat{J}_0^{(2)}(0) = \big ( (2\Phi(0) - 2 \gamma_2 \sigma_j^2)^2 -2\gamma_1\sigma_j^2 + 2\Phi^{(1)}(0)  \big ) \widehat{J}_0(0) 
    $.
\end{enumerate}
Then the solution to the inhomogeneous ODE in \eqref{eqE:hatJ_genera} is given by 
\[ \widehat{J}(t) = \widehat{J}_0(t) + \int_0^t \widehat{K}_s(t) \widehat{\psi}(s) \dif s. \]
\end{proposition}
\begin{proof} This is a direct application of \eqref{eq:general_kernal} with coefficients defined by \eqref{eq:inhomogeneous_ODE} to the ODE in \eqref{eqE:hatJ_genera}.
\end{proof}

This leads immediately to a general representation of the kernel and forcing terms for homogenized SGD, which we summarize in the following theorem.
\begin{theorem}\label{thm:bighSGD}
The homogenized SGD diffusion loss values satisfy 
\[
\Exp_{\HH}[f(\XX_t)] = F(t) + \int_0^t \mathcal{K}_s(t) \Exp_{\HH}[f(\XX_s)] \dif s
\quad\text{for all}
\quad t \geq 0.
\]
The forcing function $F$ and the kernel $\mathcal{K}$ are given by
\[
F(t) = \frac{1}{n}\sum_{i=1}^n (R\lambda_i + \widetilde{R}) G^{(\lambda_i)}(t)
\quad\text{and}\quad
\mathcal{K}_s(t) 
=
\frac{1}{n}\sum_{i=1}^{n} K^{(\lambda_i)}_s(t).
\]
The function $G^{(\lambda)}(t)$ and $K^{(\lambda)}_s(t)$ are solutions of a differential equation, where if $\lambda=0$ then $G^{(\lambda)}(t)=1$ and $K^{(\lambda)}_s(t)=0$.  Define the differential operator
\[
\begin{aligned}
L^{(\lambda)}[ \widehat{J}] = 
    &\widehat{J}^{(3)} + \left (  -3\Phi + 3 \gamma_2 \lambda \right ) \widehat{J}^{(2)} + \left ( -5\Phi^{(1)} + 2 \Phi^2 -4 \gamma_2 \lambda\Phi +4\gamma_1\lambda + 2 \gamma_2^2 \lambda^2 \right ) \widehat{J}^{(1)} \\
    &+ \left ( -2\Phi^{(2)}+4\Phi \Phi^{(1)} - 4\gamma_2\lambda \Phi^{(1)} - 4 \gamma_1 \lambda \Phi + 4 \gamma_1 \gamma_2 \lambda^2 \right ) \widehat{J}.
    \end{aligned}
\]
Then the interaction kernel is given by
\[\begin{gathered}
K_s^{(\lambda)}(t) = \tfrac{\lambda^2\varphi^2(s) \widehat{K}_s(t)}{\varphi^2(t)}
\quad\text{where}
\quad
L^{(\lambda)}[ \widehat{K}_s](t) = 0,\,\, t \geq s,\,\,\widehat{K}_s(t) =0,\,\,t < s,
\quad\text{and}\quad \\
\quad \widehat{K}_s(s) = \gamma_2^2, \quad \widehat{K}_s'(s) = 2 \gamma_2\gamma_1 + {2\gamma_2^2}\left (\Phi(s)-\gamma_2 \lambda \right ),\quad \text{and}\\ \quad \widehat{K}_s''(s) = 2\gamma_1(\gamma_1 + 3 \gamma_2 \Phi(s) - 4\gamma_2^2 \lambda) 
+{2 \gamma_2^2} \big [ \Phi^{(1)}(s) + 2\Phi^2(s) - 4 \gamma_2 \lambda \Phi(s) + 2 \gamma_2^2 \lambda^2 \big ].
\end{gathered}
\]
The forcing kernel is given by
\[
\begin{gathered}
G^{(\lambda)}(t) = \tfrac{\widehat{J}_0(t)}{2\varphi^2(t)}
\quad\text{where}
\quad
L^{(\lambda)}[ \widehat{J}_0] = 0,\quad t \geq 0, 
\quad\text{and}\quad \\
\widehat{J}_0(0) = 1, \, \, \widehat{J}^{(1)}_0(0) = (2 \Phi(0)-2\gamma_2 \lambda),\,\,\text{and}\,\,\, \widehat{J}_0^{(2)}(0) = \big ( (2\Phi(0) - 2 \gamma_2 \lambda)^2 -2\gamma_1\lambda + 2\Phi^{(1)}(0)  \big ).
\end{gathered}
\]
\end{theorem}
\begin{proof}
Using the results derived so far, we now formulate the autonomous Volterra equation for the loss under homogenized SGD $\Exp f(\XX_t)$.
We recall that for the least squares problem we have taking expectation (conditioning on the singular values $\SSigma$)
\[
\Exp_{\HH}[f(\XX_t)]
=
\frac{1}{2}
\Exp_{\HH}\| \SSigma \nnu_t - (\UU^t\eeta)\|^2
=
\frac{1}{2}
\sum_{j=1}^{n}
\sigma_j^2
\Exp_{\HH}  \bigl(\nu_{t,j} - \tfrac{(\UU^t\eeta)_j}{\sigma_j}\bigr)^2
+
\frac{1}{2}
\sum_{j = d}^n
\Exp_{\HH}
(\UU^T \eeta)^2_j,
\]
where the second sum is empty when $n < d$.  Recall that $J=J^{(\sigma_j^2)}$ \eqref{eqE:JL} gives the expectation of $\Exp_{\HH}  \bigl(\nu_{t,j} - \tfrac{(\UU^t\eeta)_j}{\sigma_j}\bigr)^2$ and hence
\[
\Exp_{\HH} f(\XX_t)
=
\frac{1}{2}
\sum_{j=1}^n
\sigma_j^2
J^{(\sigma_j^2)}(t)
+ 
\frac{\widetilde{R}\min\{ n-d, 0\}}{2n}.
\]
Using \eqref{eqE:hat}
\begin{equation*}
\widehat{J} = \varphi^2 J / \gamma_1
\quad\text{and}\quad\widehat{\psi} = 2 \sigma^2_j \varphi^2 \Exp f(\XX_t)/n.
\end{equation*}
The term $n\sigma_j^2\widehat{J}$ has as initial conditions $\sigma_j^2 R + \widetilde{R}$ (when $\sigma_j^2 = 0,$ the process $\widehat{J}$ is constant).
We conclude that
\[
\Exp_{\HH} f(\XX_t)
=
\frac{1}{n}
\sum_{j=1}^n
\frac{\gamma_1 n\sigma_j^2\widehat{J}^{(\sigma_j^2)}}{2\varphi^2(t)}.
\]
From Proposition \ref{prop:kernel_general}, 
\[
\Exp_{\HH} f(\XX_t)
=
\frac{1}{n}
\sum_{j=1}^n
\frac{\gamma_1 n\sigma_j^2\widehat{J}_0^{(\sigma_j^2)}(t)}{2\varphi^2(t)}
+
\int_0^t
\frac{1}{n}
\sum_{j=1}^n
\frac{\gamma_1\sigma_j^4\varphi^2(s)\widehat{K}_s^{(\sigma_j^2)}(t)\cdot \Exp_{\HH} f(\XX_s)}{\varphi^2(t)}
\dif s.
\]
After defining $G^{(\lambda)}$ and $K^{(\lambda)}$ as in the statement of the Theorem, this completes the proof.


\end{proof}

We now give an explicit expressions for the kernel in two specific cases. 

\begin{corollary}[Kernel representation, SDANA] \label{cor:kernel} Consider the inhomogeneous ODE in \eqref{eq:DE_SDANA}. Define the differential operators with $\lambda \defas \sigma_j^2$
\begin{enumerate}
\item $
    L[\widehat{K}_s](t) = 0, \, \, t \ge s, \, \, \widehat{K}_s(t) = 0, \, \,  t < s \\
    \text{\rm where} \quad \widehat{K}_s(s) = \gamma_2^2, \quad \widehat{K}_s'(s) = 2\gamma_2 \gamma_1 + 2\gamma_2^2 \left ( \frac{\theta}{1+s} - \gamma_2 \lambda \right ), 
      \text{\rm and}\\ \quad \widehat{K}_s''(s) = \gamma_2^2 \left [ \frac{4\theta^2-2\theta}{(1+s)^2} - \frac{8\theta \gamma_2 \lambda}{1+s} + 4 \lambda^2 \gamma_2^2  \right ] +\gamma_1 (2 \gamma_1- 8 \gamma_2^2 \lambda+ \frac{6 \theta \gamma_2}{1+s} ).\\
$
\item $L[\widehat{J}_0(t)] = 0 \\
  \textrm{\rm where} \, \, \widehat{J}_0(0) = 1, \quad \widehat{J}_0^{(1)}(0) = 2 \theta - 2 \gamma_2 \lambda, \quad 
    \textrm{\rm and} \quad \widehat{J}^{(2)}_0(0) = (2\theta - 2 \gamma_2 \lambda)^2-2 \gamma_1 \sigma_j^2 -2\theta .
 $
\end{enumerate}
Then one has that
\[ K_s^{(\lambda)}(t) = \frac{\lambda^2 (1+s)^{2 \theta} \widehat{K}_s(t)}{(1+t)^{2 \theta}} \quad \text{and} \quad G^{(\lambda)}(t) = \frac{ \widehat{J}_0(t)}{2(1+t)^{2\theta}}.\]
\end{corollary}

\begin{corollary}[Kernel representation, general SDAHB] \label{cor:kernel_aleph} Consider the inhomogeneous ODE in \eqref{eq:JEQ_SHB}. Define the differential operators with $\lambda \defas \sigma_j^2$
\begin{enumerate}
    \item $L[\widehat{K}_s](t) = 0, \, \, t \ge s, \, \, \widehat{K}_s(t) = 0, \, \,  t < s \\ \text{\rm where} \quad \widehat{K}_s(s) = \gamma_2^2, \quad \widehat{K}_s'(s) = 2 \gamma_2 \gamma_1 + 2\gamma_2^2\left (\theta-\gamma_2 \lambda \right ), \quad \text{\rm and}\\ \quad \widehat{K}_s''(s) = 2\gamma_1(\gamma_1 + 3 \gamma_2 \theta - 4\gamma_2^2 \lambda) + 2 \gamma_2^2 \big [ 2\theta^2 - 4 \gamma_2 \lambda \theta + 2 \gamma_2^2 \lambda^2 \big ] $.
    \item $L[\widehat{J}_0(t)] = 0\\ \text{\rm where} \, \, \widehat{J}_0(0) = 1, \, \, \widehat{J}^{(1)}_0(0) = 2 \theta-2\gamma_2 \lambda, \, \, \text{\rm and} \, \,  \widehat{J}_0^{(2)}(0) =  (2\theta - 2 \gamma_2 \lambda)^2 -2\gamma_1\lambda 
    $.
\end{enumerate}
Then one has that
\[ K_s^{(\lambda)}(t) = \lambda^2 e^{2 \theta (s-t)} \widehat{K}_s(t) \quad \text{and} \quad G^{(\lambda)}(t) = \frac{e^{-2\theta t} \widehat{J}_0(t)}{2}.\]
\end{corollary}

\section{Relating Homogenized SGD to SGD on the random least squares problem}\label{sec:hSGDlsq}

\subsection{ Heuristic reduction}
In this section, we give a nonrigorous derivation of the homogeneous sGD which holds in general.  In the next section, we give a proof of Theorem \ref{thm:SGDthm} which applies in the case of $\gamma_1 =0$ using the results from \cite{paquetteSGD2021}.

We are considering the SDA class of algorithms \eqref{eq:SDA} which, for $\xx_1 \in \mathbb{R}^d$ and $\yy_0 = 0$,
\begin{equation}
    \begin{gathered}
    \yy_k = \left (1- \Delta(k) \right ) \yy_{k-1} + \frac{\gamma_1}{n} \nabla f_{i_k} (\xx_k) \quad \text{and} \quad 
    \xx_{k+1} = \xx_k - {\gamma_2} \nabla f_{i_k}(\xx_k) - \yy_k,
    \end{gathered}
\end{equation}
Here $\gamma_1, \gamma_2 > 0$ are step sizes and $\Delta$ is a function of the iteration $k$ and number of samples $n$ such that 
\[
\Delta(k)
\defas
\Delta(k,n) \defas \tfrac{1}{n}(\log \varphi)'(\tfrac{k}{n}).
\]
Recall that our two motivating cases are SDANA for which $\Delta(k,n) = \frac{\theta}{k + n}$ and SDAHB for which $\Delta(k,n) = \frac{\theta}{n}$.

Recall that we consider the normalized least squares problem
\[
f(\xx)
= \frac{1}{2}\|\AA \xx - \bb\|^2,
\]
and we use the singular value decomposition of $\UU \SSigma \VV^T$ of $\AA$, with singular values $\sigma_1 \geq \sigma_2 \geq \sigma_3 \cdots$ in decreasing order.  We then let $\nnu_t = \VV^T (\XX_t-\widetilde{\xx})$.  We recall that 
\[
\nabla f( \XX_t)
= \AA^T (\AA \XX_t - \bb)
\quad
\text{and}
\quad
\nabla^2 f = \AA^T \AA.
\]
Hence, we may change the basis to write 
\[
    \begin{aligned}
    &\dif (\VV^T \XX_t) = 
    -\gamma_2 
    \dif (\VV^T \ZZ_t)
    -\frac{\gamma_1}{\varphi(t)}
    \int_0^t \varphi(s) \dif (\VV^T \ZZ_t), \\
    &\dif (\VV^T \ZZ_t) = \SSigma^T (\SSigma \nnu_t- \UU^T \bb)\dif t +\sqrt{\tfrac{2}{n}f(\XX_t) \SSigma^T \SSigma} \dif (\VV^T \BB_t).
  \end{aligned}
\]
The loss values we may also represent in terms of $\nnu$
\[
f(\XX_t)
= 
\frac{1}{2}
\| \AA \XX_t - \bb\|^2
=
\frac{1}{2}
\| \SSigma \nnu_t - \UU^T\eeta\|^2
=
\frac{1}{2}\sum_{j=1}^n
( \sigma_j \nu_{t,j} - (\UU^T \bb)_j)^2
.
\]

We let $\widehat{\ww}_k = \tfrac{n}{\gamma_1}\VV^T \yy_k$ and $\widehat{\nnu}_k = \VV^T(\xx_k-\widetilde{\xx})$ (with $\ww_0 = \bm{0}$ and $\widehat{\nnu}_1 \in \mathbb{R}^d$, $k \ge 1$) so that for a random rank-1 coordinate projection matrix $\PP_k$
\begin{align}
    &\widehat{\ww}_k =  \big (1-\Delta(k,n) \big ) \widehat{\ww}_{k-1} + \SSigma^T \UU^T \PP_k (\UU \SSigma \widehat{\nnu}_k-\eeta), \quad\text{and} \quad\\
    &
    \widehat{\nnu}_{k+1} = \widehat{\nnu}_k - \gamma_2 \SSigma^T \UU^T \PP_k (\UU \SSigma \widehat{\nnu}_k-\eeta) - \tfrac{\gamma_1}{n} \widehat{\ww}_k.
\end{align}
By unraveling the recurrence for $\widehat{\ww}$, a simple computation shows that  
\begin{equation}
    \widehat{\ww}_k = \sum_{\ell = 1}^{k} \prod_{i=\ell}^{k-1} \big [ 1- \Delta(i,n) \big ] \SSigma^T \UU^T \PP_{\ell} (\UU \SSigma \widehat{\nnu}_{\ell}-\eeta).
\end{equation}
We now create a continuous time version of the vector $\widehat{\ww}_k$ so that $t$ and $s$ correspond to one pass over the data set. In doing so, we can approximate the product $\prod_{i=\ell}^{k-1} 1-\Delta(i,n)$ by first taking logarithms and then approximating the sum with a Riemann integral. If we let $\ell = ns$ and $k = nt$, 
\begin{align*} \prod_{i = \ell}^{k - 1} \big [1 - \Delta(i,n) \big ] &= \prod_{i=ns}^{nt} \big [1- \Delta(i,n) \big ] = \exp \bigg ( \sum_{i=ns}^{nt} \log\big (1-\Delta(i,n) \big ) \bigg )\\
&\approx \exp \bigg ( -\sum_{i=ns}^{nt} \Delta(i,n) \bigg ) = \exp \bigg (-\frac{1}{n} \sum_{i=ns}^{nt} (\log \varphi)'( \tfrac{i}{n} ) \bigg )\\
&\approx \exp \left ( -\int_s^t (\log \varphi)'(u) \, \dif u \right  ) =  \frac{\varphi(s)}{\varphi(t)}. 
\end{align*}
We are trying to isolate the martingale term in $\widehat{\ww}_k$ so we need to find the mean behavior of $\widehat{\ww}$. As such, 
\begin{align*}
    \widehat{w}_{k,j} &= \sum_{\ell = 1}^{k} \prod_{i=\ell}^{k-1} \big [1 - \Delta(i,n) \big ] \big ( \mathbb{e}_j^T \SSigma^T \UU^T \PP_{\ell} (\UU \SSigma \nnu_{\ell}-\eeta) - \big ( \tfrac{\sigma_j^2}{n} \widehat{\nu}_{\ell,j} - \tfrac{\sigma_j}{n} (\UU^T \eeta)_j \big ) \big )\\
    &+ \sum_{\ell=1}^{k} \prod_{i=\ell}^{k-1} \big [1- \Delta(i,n) \big ] \big ( \tfrac{\sigma_j^2}{n} \widehat{\nu}_{\ell,j} - \tfrac{\sigma_j}{n} (\UU^T \eeta)_j \big ).
\end{align*}
Define the martingale increment $\Delta \widehat{M}_{\ell,j} \defas \mathbb{e}_j^T\SSigma^T \UU^T \PP_{\ell} (\UU \SSigma \nnu_{\ell} -\eeta) - \big ( \tfrac{\sigma_j^2}{n} \widehat{\nu}_{\ell,j} - \tfrac{\sigma_j}{n} (\UU^T \eeta)_j \big )$. Then 
\begin{align*}
    \widehat{w}_{k,j} = \sum_{\ell = 1}^{k} \prod_{i = \ell}^{k-1} \big [ 1- \Delta(i,n) \big ] \Delta \widehat{M}_{\ell,j} + \sum_{\ell = 1}^{k} \prod_{i=\ell}^{k-1} \big [1-\Delta(i,n) \big ] \big (\tfrac{\sigma_j^2}{n} \widehat{\nu}_{\ell,j} - \tfrac{\sigma_j}{n} (\UU^T \eeta)_j \big )
\end{align*}
We now pass to the continuous time by letting $k \sim n t$.  So we define a continuous time, purely discontinuous martingale $M_t$ with jumps at times $\N/n$ which are given by
\[
(\Delta M)_{{\ell}/{n},j}
\defas
\mathbb{e}_j^T\SSigma^T \UU^T \PP_{\ell} (\UU \SSigma \nnu_{\ell/n} -\eeta) - n^{-1}\big ( {\sigma_j^2} {\nu}_{\ell/n,j} - {\sigma_j} (\UU^T \eeta)_j \big )
\]
In terms of this martingale, we define c\`adl\`ag processes $w_{t,j}$ and $\nu_{t,j}$ as approximations for $\widehat{\ww}_{k,j}$
and $\widehat{\nnu}_{k,j}.$
For $\ww$ this is given by
\[ w_{t,j} = \frac{1}{\varphi(t)} \int_0^t \varphi(s)
\bigl(
\dif M_{s,j} + \bigl(\sigma_j^2 \nu_{s,j} - \sigma_j (\UU^T\eeta)_j\bigr)\dif s\bigr).\]
As for $\nnu,$ we must compute also compute the change in $\widehat{\nu}_{k,j}$:
\begin{align*}
    \widehat{\nu}_{k+1, j}-\widehat{\nu}_{k,j} &= -\mathbb{e}_j^T \gamma_2 \SSigma^T \UU^T \PP_k (\UU \SSigma \widehat{\nnu}_{k}-\eeta) - \frac{\gamma_1}{n} \widehat{w}_{k,j}.
\end{align*}
Again on scaling time to be like $k \sim nt,$ we arrive at a continuous time stochastic evolution
\begin{align*}
\dif \nu_t    &= - \gamma_2 \big (\dif M_{t,j} + {\sigma_j^2} \nu_{t,j} - {\sigma_j} (\UU^T \eeta)_j \big )
-\frac{\gamma_1}{\varphi(t)} \int_0^t \varphi(s)
    \bigl(
    \dif M_{s,j} + \bigl(\sigma_j^2 \nu_{s,j} - \sigma_j (\UU^T\eeta)_j\bigr)\dif s\bigr).
\end{align*} 
Thus this is exactly the homogenized SGD \eqref{eq:dnu}, but with the martingales $\{ W_{t,j}\}$ replaced by $\{M_{t,j}\}$.

The martingales $\{M_{t,j}\}$ are purely discontinuous.  Their predictable quadratic variations are given by (ignoring errors induced by smoothing the indexing)
\[
\dif \langle M_{t,j}, M_{t,i}
\rangle
=
\sigma_j^2 
\sum_{\ell=1}^n
U_{\ell,j}
U_{\ell,i}
\bigl((\UU \SSigma {\nnu}_{t}-\eeta)_\ell\bigr)^2
-
n^{-1}
\big ( {\sigma_j^2} {\nu}_{t,j} - {\sigma_j} (\UU^T \eeta)_j \big )
\big ( {\sigma_i^2} {\nu}_{t,i} - {\sigma_i} (\UU^T \eeta)_i \big ).
\]
The latter term is too small to recover and so disappears in the large-$n$ limit.  Note that in the first sum, if $(\UU \SSigma {\nnu}_{t}-\eeta)_\ell$ could be decoupled from $U_{\ell,j}
U_{\ell,i}$ and if $(U_{\ell,j}^2: 1 \leq\ell\leq n)$ is sufficiently delocalized, then we would arrive at
\[
\dif \langle M_{t,j}, M_{t,i}
\rangle
\approx
\delta_{i,j}
{\sigma_j^2 }
\sum_{\ell=1}^n
U^2_{\ell,j}
\bigl((\UU \SSigma {\nnu}_{t}-\eeta)_\ell\bigr)^2
\approx
\delta_{i,j}
\frac{2\sigma_j^2}{n} f(\XX_t),
\]
from the fact that $U^2_{\ell,j} \approx \tfrac{1}{n}$ on average.
This is the homogenized SGD.  The main input is sufficiently strong input information on the eigenvector matrix $\UU$.  In \cite{paquetteSGD2021}, it is assumed that this is independent of the spectra and Haar orthogonally distributed.  We expect it remains true under weaker assumptions, but note some type of eigenvector assumption is needed.  If for example $\AA$ is diagonal, the resulting coordinate processes decouple entirely, as opposed to interacting through the loss values $\psi$.

\subsection{Proof of correspondence for SGD}

We give a proof of Theorem \ref{thm:SGDthm}, or rather show how \cite{paquetteSGD2021} (which contains the argument) may be adapted to show the statement claimed.  The starting point is an embedding of the discrete problem into continuous time.  That is we create a homogeneous Poisson process $N_t$ with rate $n$ (so that in one unit of time, in expectation $n$ Poisson points arrive).  We then introduce the notation 
\begin{equation} \begin{aligned} \label{eq:psidef}
    \psi_{\varepsilon}(t) \defas f(\xx_{N_t})
    &=
    \frac{1}{2}\| \AA (\xx_{N_t}-\widetilde{\xx}) -\, \eeta \|^2
    =
    \frac{1}{2}\| \SSigma \nnu_t - \UU^T\eeta\|^2.\\
    &=\frac{1}{2}
\sum_{j=1}^d \sigma_j^{2}\nu_{t,j}^2
-\sum_{j=1}^{n \wedge d} \sigma_j\nu_{t,j}\, (\UU^T\eeta)_j 
+\frac{1}{2}\|\eeta\|^2.
\end{aligned}
\end{equation}
By partially integrating the equation, we can rewrite this equation (see the derivation \cite[Equation (41)]{paquetteSGD2021} through \cite[Lemma 21]{paquetteSGD2021} -- note we are using batchsize $\beta=1$).
\begin{equation}\label{eq:psivolterra}
\begin{aligned}
\psi_{\varepsilon}(t)
&=
\frac{1}{2}
\sum_{j=1}^n \sigma_j^{2}
\biggl(e^{-2t\gamma\sigma_j^2}\nu_{0,j}^2 + \int_0^t e^{-2(t-s)\gamma\sigma_j^2}\gamma^2\frac{2\sigma_j^2\psi_{\varepsilon}(s)}{n}\dif s\biggr) \\ 
&+
\frac{1}{2}
\sum_{j=1}^n
\int_0^t 
e^{-2(t-s)\gamma\sigma_j^2}
\gamma 2\sigma_j^3\nu_{s,j}(\UU^T\eeta)_j \dif s
+\frac{1}{2}\|\eeta\|^2
-\sum_{j=1}^{n \wedge d} \sigma_j\nu_{t,j}\, (\UU^T\eeta)_j \\
&+
\varepsilon^{(n)}_{\operatorname{KL}}(t)
+
\varepsilon^{(n)}_{\operatorname{M}}(t)
\end{aligned}
\end{equation}
The two error terms $\varepsilon^{(n)}_{\operatorname{KL}}(t)$ and $\varepsilon^{(n)}_{\operatorname{M}}(t)$ are (first) due to the eigenvectors not being perfectly delocalized and (second) due to the randomness of SGD.  Note that we can write this in terms of the notation for Theorem \ref{thm:bighSGD} by
\begin{equation}\label{eq:newpsivolterra}
\begin{aligned}
\psi_{\varepsilon}(t) &= 
F_{\HH}(t) + \int_0^t \mathcal{K}_s(t) \psi_{\varepsilon}(s) \dif s
+
\varepsilon^{(n)}_{\operatorname{KL}}(t)
+
\varepsilon^{(n)}_{\operatorname{M}}(t)
+
\varepsilon^{(n)}_{\operatorname{xtra}}(t),\quad\text{where} \\
\varepsilon^{(n)}_{\operatorname{xtra}}(t)
&= 
\frac{1}{2}
\sum_{j=1}^n \sigma_j^{2}
\biggl(e^{-2t\gamma\sigma_j^2}\bigl( (\nu_{0,j}-(\UU^T\eeta)_j/\sigma_j)^2 - R/n - \widetilde{R}/(n\sigma_j^2)\bigr) \\
&+
\frac{1}{2}
\sum_{j=1}^n
\int_0^t 
e^{-2(t-s)\gamma\sigma_j^2}
\gamma 2\sigma_j^3\nu_{s,j}(\UU^T\eeta)_j \dif s\\
\end{aligned}
\end{equation}
The main errors are controlled directly using the results of \cite{paquetteSGD2021}.
\begin{proposition}\label{prop:SGDconcentration}
There is an $\delta > 0$ so that 
for any $T>0$ there is a constant $C(T,\lambda_{\HH}^+) >0$ so that
\[
\Pr[
\sup_{0 \leq t \leq T}
\bigl\{
|\varepsilon^{(n)}_{\operatorname{KL}}(t)|
+
|\varepsilon^{(n)}_{\operatorname{M}}(t)|
\bigr\}
\geq C(T)n^{-\delta}~\vert~ \SSigma
] \leq n^{-\delta}.
\]
\end{proposition}
\begin{proof}
In short this the combination of 
\cite[Lemma 13]{paquetteSGD2021},
\cite[Lemma 14]{paquetteSGD2021},
\cite[Proposition 16]{paquetteSGD2021}.  We note that the event that dominates the probability is the application of \cite[Lemma 13]{paquetteSGD2021}, which is simply Markov's inequality applied to the loss.
\end{proof}
The extra errors are controlled by \cite[Proposition 19]{paquetteSGD2021} and by concentration of the initial conditions.  Thus by \eqref{eq:newpsivolterra}, we have that the true loss of SGD satisfies an approximate Volterra equation.  Using the stability of Volterra equations with respect to perturbation \cite[Proposition 11]{paquetteSGD2021}, we can conclude that $\psi_\epsilon$ is close to a solution of the Volterra equation with $\varepsilon = 0$, with the claimed probability.

\section{Analysis of convolution Volterra equations: convergence and rates}\label{sec:rates}

In what follows, we give an analysis of a class of convolution Volterra equations that appear naturally in the SDA context: our analysis will give convergence guarantees, convergence rates and limiting losses (in the underparameterized context.  Ultimately, for all of SGD, SDAHB and SDANA, we will have the task of describing the evolution of the training loss $L(t)=\E_{\HH}[f(X_t)]$ which satisfies
\begin{equation}\label{q:V}
  \begin{aligned}
    &L(t) = F(t) + \int_0^t \mathcal{I}(t-s) L(s),
  \quad
  &\text{for all $t \geq 0$,} \\
  &F(t) = \int_0^\infty (R\lambda + \widetilde{R})G^{(\lambda)}(t) \mu(d\lambda),
  \quad
  \text{and}
  \quad
  \mathcal{I}(t) = \int_0^\infty K^{(\lambda)}(t) \mu(d\lambda),
  &\text{for all $t \geq 0$.} \\
  \end{aligned}
\end{equation}
We refer to $F$ as the forcing function and $\mathcal{I}$ as the convolution kernel.
The measure $\mu$ is the limiting spectral measure of the Hessian problem (some parts of the analysis also hold with $\mu$ the actual empirical spectral measure of the problem).
In all cases, operate under the following assumptions.
\begin{assumption}\label{q:genV}
  The functions 
  $(\lambda,t) \mapsto G^{(\lambda)}(t)$
  and 
  $(\lambda,t) \mapsto K^{(\lambda)}(t)$
  are non-negative, continuous and bounded on bounded sets of $\lambda$.  Assume further that for each $\lambda > 0$, $K$ and $G$ tend to $0$ as $t \to \infty$.  At $\lambda = 0,$ on the other hand $K^{(\lambda)} \equiv 0$ and $G^{(\lambda)} \equiv 1.$
\end{assumption}

In this section, we shall also do the analysis for SGD.  Much of SGD analysis appeared already in \cite{paquetteSGD2021}, but it serves as an instructive and simple example.  Recall that for the case of SGD, we have that
\begin{equation}\label{q:SGDk}
  G_{\text{sgd}}^{(\lambda)}(t)
  =
  e^{-2\gamma \lambda t}
  \quad
  \text{and}
  \quad
  K_{\text{sgd}}^{(\lambda)}(t)
  = \gamma^2 \lambda^2 e^{-2\gamma \lambda t}
\end{equation}

The actual convergence analysis of $L$ in this setup is relatively simple.  As a consequence of dominated convergence, we have $F$ converges as $t \to \infty$, and in fact
\[
  \lim_{t \to \infty}
  F(t) = \mu(\{0\}) \widetilde{R}.
\]
The important input to ensure convergence is that the norm of the convolution kernel is controlled.
\begin{equation}\label{q:kernel}
  \|\mathcal{I}\|
  =
  \int_0^\infty \mathcal{I}(t) dt
  =
  \int_0^\infty
  \int_0^\infty  K^{(\lambda)}(t) dt \mu(d\lambda)
\end{equation}

Thus by dominated convergence, we have the following:
\begin{lemma}\label{q:limit}
  Suppose Assumption \ref{q:genV} and
  suppose $\|\mathcal{I}\| < 1$.  Then 
  \[
    \lim_{t \to \infty}
    L(t) = \frac{\mu(\{0\}) \widetilde{R}}{1- \|\mathcal{I}\|}.
  \]
\end{lemma}
For SGD, in particular that means (using \eqref{q:kernel})
\begin{equation}\label{q:sgd}
  \|\mathcal{I}_{\text{sgd}}\|
  =
  \int_0^\infty
  \int_0^\infty  K_{\text{sgd}}^{(\lambda)}(t) dt \mu(d\lambda)
  =
  \int_0^\infty
  \frac{\gamma \lambda}{2}
  \mu(d\lambda)
  =
  \frac{\gamma \mathfrak{tr}(\mu)}{2},
\end{equation}
where $\mathfrak{tr}(\mu)$ is the limiting normalized trace of the Hessian, i.e.\ the first moment of the measure $\mu.$

\paragraph{Rates (heavy-tailed case)} 
The rate analysis is divided into two cases, according to the behavior of the forcing function $F$.  If $F$ converges exponentially quickly to its limit (which occurs in our applications when the spectrum of $\mu$ is separated from $0$), then the forcing function converges exponentially.  On the other hand, if $\mu$ has a density in a neighborhood of $0$, then the rate is subexponential, and we will suppose further that $F$ and $\mathcal{I}$ are both tending slowly to $0$.
\begin{assumption}\label{q:rv}
  The function $F$ dominates $\mathcal{I}$ and
  $\mathcal{I}/\|\mathcal{I}\|$ defines a subexponential distribution: that is
  \[
    F(t)/\mathcal{I}(t) \to \infty
    \quad
    \text{and}
    \quad
    \frac{\int_T^\infty \int_0^t \mathcal{I}(t-s)\mathcal{I}(s)ds dt}
    {
      \int_T^\infty \mathcal{I}(s)ds
    }\to 2\|\mathcal{I}\|.
  \]
\end{assumption}
A simple sufficient condition for the latter of the two conditions is that $\mathcal{I}(t) t^{\beta} \to c$ for some $c>0,\beta>0$ as $t \to \infty.$
In this case, the rate of convergence of $L$ to its limit
is 
\begin{lemma}\label{q:rvrate}
  Suppose Assumptions \ref{q:genV} and \ref{q:rv}
  and
  suppose $\|\mathcal{I}\| < 1$.  Then 
  \[
    L(t) - \frac{\mu(\{0\}) \widetilde{R}}{1- \|\mathcal{I}\|}
    \sim_{t\to\infty}
    \frac{F(t) - \mu(\{0\})\widetilde{R}}{1- \|\mathcal{I}\|}
    .
  \]
\end{lemma}
In other words, the rate is completely dominated by whichever rate $F$ takes. 
\begin{proof}
  If we subtract the limiting behavior from $L$, we have
  \[
    \begin{aligned}
    L(t)-
    \frac{\mu(\{0\}) \widetilde{R}}{1- \|\mathcal{I}\|}
    &=
    F(t) - \mu(\{0\}) \widetilde{R}
    +
    \int_0^t \mathcal{I}(t-s)
    \biggl(L(s)-
    \frac{\mu(\{0\}) \widetilde{R}}{1- \|\mathcal{I}\|}\biggr)
    ds
  \end{aligned}
  \]
  Thus if we set 
  \[
    \widehat{L} = L-\frac{\mu(\{0\}) \widetilde{R}}{1- \|\mathcal{I}\|}
    \quad\text{and}\quad
    \widehat{F} = F-{\mu(\{0\}) \widetilde{R}},
  \]
  we conclude that
  \[
    \widehat{L}(t) = \widehat{F}(t) + \int_0^t \mathcal{I}(t-s)\widehat{L}(s)ds.
  \]
  This is a defective renewal equation, and moreover it is a defective renewal equation in which $\mathcal{I}/\|\mathcal{I}\|$ defines a \emph{subexponential distribution}.  Thus from \cite[(7.8)]{Asmussen}, we conclude the claim.
\end{proof}

To apply this to SGD, we suppose that
\begin{equation}\label{q:le}
\mu((0,\epsilon]) \underset{\epsilon \to 0}{\sim} \ell \epsilon^{\alpha}.
\end{equation}
In this case, we conclude that
\[
  \mathcal{I}_{\text{sgd}}(t) \underset{t \to \infty}{\sim}
  \frac{\Gamma(2+\alpha)\alpha \ell \gamma^2}{ (2\gamma t)^{2+\alpha}}
  \quad\text{and}\quad
  F_{\text{sgd}}(t)
  -\mu(\{0\})\widetilde{R} \underset{t \to \infty}{\sim}
  \frac{R\Gamma(1+\alpha)\alpha \ell}{ (2\gamma t)^{1+\alpha}}
  +\frac{\widetilde{R}\Gamma(\alpha)\alpha \ell}{ (2\gamma t)^{\alpha}}.
\]
Thus we conclude the rate for the loss of SGD (using \eqref{q:SGDk} and Lemma \ref{q:rvrate}
\begin{equation}\label{q:sgdht}
\E_{\HH}[f(X_t)]- 
\frac{\mu(\{0\}) \widetilde{R}}{1-\frac{\gamma \mathfrak{tr}(\mu)}{2}}
\underset{t \to \infty}{\sim}
\frac{1}{1-\frac{\gamma \mathfrak{tr}(\mu)}{2}}
\biggl(
\frac{R\Gamma(1+\alpha)\alpha \ell}{ (2\gamma t)^{1+\alpha}}
+\frac{\widetilde{R}\Gamma(\alpha)\alpha \ell}{ (2\gamma t)^{\alpha}}
\biggr).
\end{equation}

\paragraph{Rates (exponential case)} 
We now consider the case $F$ and $\mathcal{I}$ tend to $0$ exponentially, as is the case when $\mu$ has support $\{0\} \cup [\lambda_-,\lambda_+]$ for positive $\lambda_-.$  We enforce these assumptions by assuming that both $G^{(\lambda)}$ and $K^{(\lambda)}$ behave well in a neigbhorhood of the spectral edge.
\begin{assumption}\label{q:er}
  The support of $\mu$ is contained in $\{0\} \cup [\lambda_-,\lambda_+]$ for some $\lambda_- > 0,$ and $\lambda_-$ is in the support.
  The kernels satisfy for some positive strictly increasing function $f$ on some $[\lambda_-,\lambda_-+\delta]$
\[
  G^{(\lambda)}(t) = e^{-(f(\lambda)+o(1))t}
  \quad
  \text{and}
  \quad
  K^{(\lambda)}(t) = e^{-(f(\lambda)+o(1))t}
\]
as $t \to \infty$.  Both $G^{(\lambda)}(t)$ and $K^{(\lambda)}(t)$ are bounded by $e^{-(f(\lambda)+o(1))t}$ for larger $\lambda$ as $t \to \infty$.
\end{assumption}

To estimate the rate, we need to introduce the Laplace transform of the kernel 
\begin{equation}\label{q:LT}
  \mathcal{F}(x)
  \defas
  \int_0^\infty e^{xt} \mathcal{I}(t)dt
  =
  \int_0^\infty 
  \int_0^\infty e^{xt} K^{(\lambda)}(t) dt \mu(dx)
\end{equation}
We define the Malthusian exponent $\lambda^*$, if it exists, as the solution of
\begin{equation}\label{q:ME}
  \mathcal{F}(\lambda^*)=1.
\end{equation}
The Malthusian exponent gives the right behavior, on exponential scale for the rate of convergence, when it exists.  Otherwise it is simply $e^{-f(\lambda_-)t}$:
\begin{lemma}\label{q:errate}
  Suppose Assumptions \ref{q:er} and \ref{q:genV} and suppose $\|\mathcal{I}\|<1$.  
  Then 
  \[
    L(t)-
    \frac{\mu(\{0\}) \widetilde{R}}{1- \|\mathcal{I}\|}
    =\begin{cases}
      e^{-(\lambda^*+o(1))t}, & \text{if $\lambda^*$ exists}, \\
      e^{-(f(\lambda_-)+o(1))t}, & \text{otherwise},
    \end{cases}
  \]
  as $t\to\infty.$
  Furthermore $\lambda^* \leq f(\lambda_-)$ if it exists.
\end{lemma}
\begin{proof}
  First, if the Malthusian exponent exists, we observe it must be positive, since by assumption
  \[
    \mathcal{F}(0) = \int_0^\infty \mathcal{I}(t) dt = \|\mathcal{I}\| < 1,
  \]
  and the function $x \mapsto \mathcal{F}(x)$ is increasing (and continuously differentiable on $[0,\lambda_-)$).
  Furthermore, the Laplace transform $\mathcal{F}(x)=\infty$ for any $x > f(\lambda_-)$ as $\lambda_-$ is in the support of $\mu$ and $f$ is increasing.  Hence $\lambda^*$ if it exists is in $(0,f(\lambda_-)].$

  It follows that if the Malthusian exponent does not exist, then $\mathcal{F}(f(\lambda_-)) < 1.$
  In that case we can transform \eqref{q:V} by taking
  \[
    \widehat{L}(t) = e^{f(\lambda_-)t}\biggl(L(t)-\frac{\mu(\{0\}) \widetilde{R}}{1- \|\mathcal{I}\|}\biggr)
    \quad\text{and}\quad
    \widehat{L}(t) = e^{f(\lambda_-)t}\biggl(F(t)-{\mu(\{0\}) \widetilde{R}}\biggr),
  \]
  which therefore satisfies
  \[
    \widehat{L}(t)
    =
    \widehat{F}(t)
    +
    \int_0^t
    e^{f(\lambda_-)s}
    \mathcal{I}(s)
    \widehat{L}(t-s)
    ds.
  \]
  The function $\widehat{F}(t)$ grows subexponentially in $t$ by hypothesis, and since $e^{f(\lambda_-)s}
    \mathcal{I}(s)$ has norm less than $1$ (its norm being $\mathcal{F}(f(\lambda_-))$) it follows that
    \[
    \widehat{L}(t)
    =
    \widehat{F}(t)
    +
    \int_0^t
    R(s)
    \widehat{F}(t-s)
    ds,
    \]
    for the $\text{L}^1$-resolvent kernel of $e^{f(\lambda_-)s}\mathcal{I}(s)$, which is the infinite series of convolution powers of this kernel. Hence $\widehat{L}(t)$ grows at most subexponentially and at least as fast as $\widehat{F},$ from which we conclude that $L(t)$ behaves like $e^{-(f(\lambda_-)+o(1))t}$.

    If $\mathcal{F}(f(\lambda_-)) \geq 1,$ we instead have a nontrivial Malthusian exponent inside of $(0,f(\lambda_-)$.  We therefore have after making the transformation
      \[
    \widehat{L}(t) = e^{\lambda^* t}\biggl(L(t)-\frac{\mu(\{0\}) \widetilde{R}}{1- \|\mathcal{I}\|}\biggr)
    \quad\text{and}\quad
    \widehat{L}(t) = e^{\lambda^* t}\biggl(F(t)-{\mu(\{0\}) \widetilde{R}}\biggr),
  \]
  that $\widehat{L}(t)$ solves Blackwell's renewal equation (see \cite[Theorem 4.7]{Asmussen}.  If $\mathcal{F}(f(\lambda_-)) > 1,$ then the Laplace transform $\mathcal{F}$ is differentiable at $\lambda^*,$ and so the renewals have finite mean $\mu = \mathcal{F}(\lambda^*)'$.  In particular it follows that $\widehat{L}(t)$ actually converges to $\frac{1}{\mu}\int_0^\infty \widehat{F}(t) dt$, which is finite by the exponential growth condition.
  
  In the critical case $\mathcal{F}(f(\lambda_-))=1$, we observe that
  \[
    \max_{[0,t]}
    \widehat{L}(u)
    \leq
    \max_{[0,t]}
    \widehat{F}(u)
    +
    \bigl(\max_{[0,t]}
    \widehat{L}(u)
    \bigr)
    \int_0^t
    e^{f(\lambda_-)s}
    \mathcal{I}(s)
    ds,
  \]
  thus rearranging, and using that $\int_0^t
  e^{f(\lambda_-)s}
  \mathcal{I}(s)
  ds \to 1$ as $t \to \infty$ we conclude
  \[
    \max_{[0,t]}
    \widehat{L}(u)
    \leq 
    \frac{\bigl(
      \max_{[0,t]}
      \widehat{F}(u)
    \bigr)}
    {\int_t^\infty
      e^{f(\lambda_-)s}
      \mathcal{I}(s)
    ds},
  \]
  which therefore grows subexponentially.  
\end{proof}

For SGD, the $K$ and $G$ (recall \eqref{q:SGDk}) satisfy Assumption \ref{q:er} trivially with $f(\lambda) = 2\gamma \lambda$.  Hence, to esetimate the rate, by which we mean
\[
  \text{sgd-rate}(\gamma) \defas \liminf_{t\to\infty} \frac{-\log\bigl(\E_{\HH}[ f(\XX_t)] - (\E_{\HH}[ f(\XX_\infty)]\bigr)}{t},
\]
the only task is to estimate the Malthusian exponent.  
\begin{lemma}\label{q:default}
  At the default parameter for SGD, $\gamma=\tfrac{1}{\mathfrak{tr}(\mu)}$, the rate is at least $\tfrac{\lambda_-}{\mathfrak{tr}(\mu)}$.  The maximum rate over all $\gamma$ is at most $4\tfrac{ \lambda_-}{\mathfrak{tr}(\mu)}.$
\end{lemma}
\begin{proof}
  Note that the Laplace transform is given by
  \[
    \mathcal{F}_{\text{sgd}}(x)
    =
    \int_0^\infty e^{xt}\mathcal{I}_{\text{sgd}}(t)dt
    =
    \int_0^\infty \frac{\gamma^2\lambda^2}{2\gamma\lambda - x} \mu(dx).
  \]
  Note that if we choose the default parameter $\gamma=\tfrac{1}{\mathfrak{tr}(\mu)}$, then at $x=\tfrac{\lambda_-}{\mathfrak{tr}(\mu)}$ we have
  \[
    \mathcal{F}_{\text{sgd}}(\tfrac{\lambda_-}{\mathfrak{tr}(\mu)})
    =
    \frac{1}{\mathfrak{tr}(\mu)}\int_0^\infty \frac{\lambda^2}{2\lambda - \lambda_-} \mu(dx)
    \leq
    \frac{1}{\mathfrak{tr}(\mu)}
    \int_0^\infty \frac{\lambda^2}{\lambda} \mu(dx)
    \leq 1.
  \]
  It follows that we have shown that the rate at the default parameter is at least $\tfrac{\lambda_-}{\mathfrak{tr}(\mu)}$.  

  To get an upper bound over all step sizes, note that from \eqref{q:sgd}, the largest $\gamma$ we can take is determined by
  \[
    \frac{\gamma \mathfrak{tr}(\mu)}{2}
    =
    \|\mathcal{I}_{\text{sgd}}\|
    <1.
  \]
  Further, the fastest rate we can ever attain is $e^{-2\gamma \lambda_- t},$ and hence taking the largest $\gamma,$ the fastest possible rate is $e^{-\tfrac{4 \lambda_-}{\mathfrak{tr}(\mu)} t}$.
\end{proof}

\section{Momentum can be faster, SDANA}\label{sec:SDANA}
In this section, we consider the SDANA case in depth, developing approximations and limit behaviors for the differential equations for which one achieves acceleration in the non-strongly convex setting. 
Recall from the ODE for $\widehat{J}$ in \eqref{eqE:hatJ_genera} and the initial conditions \eqref{eqE:JIC_general}, 
\begin{equation}\label{eqE:hatJ}
    \begin{aligned}
    &\widehat{J}^{(3)} - \left ( \tfrac{3 \theta}{(1+t)} - 3 \gamma_2 \sigma_j^2 \right ) \widehat{J}^{(2)} - \left ( -\tfrac{5 \theta + 2 \theta^2}{(1+t)^2} + \tfrac{4 \gamma_2 \sigma_j^2 \theta}{ (1+t)} -4\gamma_1\sigma_j^2 - 2 \gamma_2^2 \sigma_j^4 \right ) \widehat{J}^{(1)} \\
    &- \left ( \tfrac{4 \theta + 4 \theta^2}{ (1+t)^3} - \tfrac{4 \gamma_2 \sigma_j^2 \theta }{ (1+t)^2} + \tfrac{4 \gamma_1 \sigma_j^2 \theta}{ (1+t) } - 4 \gamma_1 \gamma_2 \sigma_j^4 \right ) \widehat{J}\\
    &=
    \tfrac{\gamma_2^2}{\gamma_1}
    \widehat{\psi}^{(2)}
    +\bigl(
    2\gamma_2
    + 
    \bigl(\tfrac{-\theta}{(1+t)} + \gamma_2\sigma_j^2\bigr)\tfrac{\gamma_2^2}{\gamma_1}
    \bigr)
    \widehat{\psi}^{(1)}
    +\bigl(
    2\gamma_1
    + 
    \bigl(\tfrac{\theta}{(1+t)^2} + 2\gamma_1\sigma_j^2\bigr)\tfrac{\gamma_2^2}{\gamma_1}
    \bigr)
    \widehat{\psi},
    \end{aligned}
\end{equation}
where the initial conditions are given by
\begin{equation}\label{eqE:JIC}
\begin{gathered}
\widehat{J}(0) = \gamma_1^{-1}\EE\biggl[ \bigl(\nu_{0,j}-\tfrac{(\UU^T \eeta)_j}{\sigma_j}\bigr)^2\biggr],
\quad
 \widehat{J}^{(1)}(0) =
 \tfrac{\gamma_2^2}{\gamma_1} \widehat{\psi}(0)
-\widehat{J}(0) \bigl( -2\theta + 2\gamma_2 \sigma_j^2\bigr),
\quad\text{and} \\
\widehat{J}^{(2)}(0)
 =  \tfrac{\gamma_2^2}{\gamma_1} \widehat{\psi}^{(1)}(0)+ 2\gamma_2 \widehat{\psi}(0) - 2 \gamma_1 \sigma_j^2 \widehat{J}(0) + ( 2 \theta - 2 \gamma_2 \sigma_j^2) \widehat{J}^{(1)}(0) - 2\theta \widehat{J}(0).
\end{gathered}
\end{equation}

\begin{remark}\label{rk:whittaker}
It is possible to represent the solutions to the homogeneous ODE \eqref{eqE:hatJ} by
\begin{equation}
    \begin{aligned}
    \widehat{J}(t) = &c_1 \underbrace{(1+t)^{\theta} e^{-\gamma_2 \sigma_j^2  t} \,  \big ( \text{WhittakerM}(A, B, C) \big )^2}_{y_1} + c_2 \underbrace{(1+t)^{\theta} e^{-\gamma_2 \sigma_j^2 t} \big ( \text{WhittakerW} (A, B, C ) \big )^2}_{y_2}\\
    &+ c_3 \underbrace{(1+t)^{\theta} e^{-\gamma_2 \sigma_j^2 t} \big ( \text{WhittakerM}(A, B, C ) \cdot \text{WhittakerW}(A, B, C ) \big )}_{y_3}\\
    \text{where} & \quad A = \frac{\sigma_j \gamma_2 \theta}{2 \sqrt{\sigma_j^2 \gamma_2^2 - 4 \gamma_1} }, \quad B = \frac{\theta-1}{2}, \quad \text{and} \quad C = \sigma_j (1+t) \sqrt{\sigma_j^2 \gamma_2^2 - 4 \gamma_1}.
    \end{aligned}
\end{equation}
For multiple reasons, working with this representation appears to add complications: we need uniform asymptotic expansions as $\sigma$ tends to $0$. We also need estimates for the fundamental solutions with parameters in a neighborhood of the turning point $\sigma_j^2 \gamma_2^2 - 4 \gamma_1 \approx 0.$
\end{remark}

\subsection{The fundamental solutions of the scaled ODE}
To give uniform estimates as $\sigma_j \to 0,$ we will scale time $t$ by $\sigma_j$ and in doing so, we define a scaled differential equation for $\widetilde{J}(t) = \widehat{J}(t/\sigma_j)$. We develop properties of the fundamental solutions of the homogeneous version of the equation \eqref{eqE:hatJ}, given by
\begin{equation}\label{eqE:ODE}
\begin{aligned}
    \widetilde{L}[ \widetilde{J}]
    \coloneqq
    &\widetilde{J}^{(3)} - \left ( \frac{3 \theta}{\sigma_j+t} - 3 \gamma_2 \sigma_j \right ) \widetilde{J}^{(2)} - \left ( -\frac{5 \theta + 2 \theta^2}{(\sigma_j+t)^2} + \frac{4 \gamma_2 \sigma_j \theta}{\sigma_j+t} - 4 \gamma_1 - 2 \gamma_2^2 \sigma_j^2 \right ) \widetilde{J}'\\
    & - \left ( \frac{4 \theta + 4 \theta^2}{ (\sigma_j+t)^3} - \frac{4 \gamma_2 \sigma_j \theta }{ (\sigma_j+t)^2} + \frac{4 \gamma_1 \theta}{ \sigma_j+t } - 4 \gamma_1 \gamma_2 \sigma_j \right ) \widetilde{J}=0.\\
\end{aligned}
\end{equation}
One can, in principle, derive an exact solution for this ODE using Whittaker functions; the resulting solution is quite cumbersome. As such we develop families of local solutions in a neighborhood of $0$ and in neighborhood of $\infty.$

The Wronskian of this differential equation will be needed multiple times.  Due to Abel's identity, the Wronskian of any three fundamental solutions of \eqref{eqE:ODE} is (for any $t,s \in \RR$)
\begin{equation}\label{eqE:Wronskian}
\frac{\mathscr{W}(t)}{\mathscr{W}(s)}
=
\frac{(\sigma_j+t)^{3\theta}}
{(\sigma_j+s)^{3\theta}}
e^{-3\gamma_2\sigma_j(t-s)}.
\end{equation}

\paragraph{The neighborhood of infinity.}
The approach we take is to derive a local series solution for large $t$ as seen in \cite[Chapter 5]{coddington1955theory}. We observe that the coefficients in the linear ODE \eqref{eqE:ODE} are analytic in a neighborhood of $\infty$. As such, there exists a \textit{formal} solution to this ODE \citep[Chapter 5, Theorem 2.1]
{coddington1955theory} given by 
\[\widetilde{J}(t) = e^{\lambda t} (\sigma_j+t)^{\rho} P(t) = e^{\lambda t} (\sigma_j+t)^{\rho} \big ( c_0 + \tfrac{c_1}{\sigma_j+1} + \tfrac{c_2}{(\sigma_j+t)^2} + \hdots \big ),\]
where $\lambda, \rho$ are constants and $P(t)$ is an analytic function in a neighborhood of $\infty$. This formal series solution asymptotically agrees with the actual solution \citep[Chapter 5, Theorem 4.1]{coddington1955theory}, and in fact are convergent solutions for all $t \in (0, \infty)$. We now derive the constants $\lambda$ and $\rho$ by simply plugging in our guess for the solution and deriving equations for $\lambda$ and $\rho$. To make this computationally tractable, we will compute derivatives in terms of $\widetilde{J}'/\widetilde{J}$, that is, 
\begin{gather*}
    \tfrac{\widetilde{J}'}{\widetilde{J}} = \lambda + \frac{\rho}{\sigma_j+t} - \frac{c_1}{(\sigma_j+t)^2} + \mathcal{O}(t^{-3}), \quad \left ( \tfrac{\widetilde{J}'}{\widetilde{J}} \right )' = \tfrac{\widetilde{J}''}{\widetilde{J}} - \left ( \tfrac{\widetilde{J}'}{\widetilde{J}} \right )^2 = - \frac{\rho}{(\sigma_j+t)^2} + \mathcal{O}(t^{-3})\\
\text{and} \quad \left ( \tfrac{\widetilde{J}'}{\widetilde{J}} \right )'' = \tfrac{\widetilde{J}'''}{\widetilde{J}} - 3 \tfrac{\widetilde{J}'' \widetilde{J}'}{\widetilde{J}^2} + 2 \left ( \tfrac{\widetilde{J}'}{\widetilde{J}} \right )^3 = \mathcal{O}(t^{-3}).
\end{gather*}
In particular, after some simple computations, we get the following expressions
\begin{align*}
    \tfrac{\widetilde{J}'}{\widetilde{J}} &= \lambda + \frac{\rho}{\sigma_j+t} - \frac{c_1}{(\sigma_j+t)^2} + \mathcal{O}(t^{-3})\\
    \tfrac{\widetilde{J}''}{\widetilde{J}} &= \lambda^2 + \frac{2 \lambda \rho }{\sigma_j+t} + \frac{\rho^2 - 2 \lambda c_1 - \rho}{(\sigma_j+t)^2} + \mathcal{O}(t^{-3})\\
    \tfrac{\widetilde{J}'''}{\widetilde{J}} &= \lambda^3 + \frac{3\lambda^2 \rho}{\sigma_j+t} + \frac{3\lambda\rho^2-3\lambda^2c_1-3 \lambda\rho }{(\sigma_j+t)^2} + \mathcal{O}(t^{-3}).
\end{align*}
Finally we have all the pieces to get the expressions for the coefficients $\lambda$ and $\rho$ by using \eqref{eqE:ODE}
\begin{equation}\begin{gathered}
    0 = \lambda^3 + 3 \gamma_2 \sigma_j \lambda^2 + (4 \gamma_1 + 2 \gamma_2^2 \sigma_j^2) \lambda + 4 \gamma_1 \gamma_2 \sigma_j = 
    (\lambda + \sigma_j \gamma_2) ( \lambda^2 + 2 \sigma_j \gamma_2 \lambda + 4 \gamma_1)\\
    \text{and} \quad  0 = (\sigma+t)^{-1} \big [ 3 \theta \lambda^2 + 4 \gamma_2 \sigma_j^2 \theta \lambda + 4 \gamma_1 \sigma_j^2 \theta-\rho \big ( 3 \lambda^2 + 6 \gamma_2 \sigma_j^2 \lambda + 4 \gamma_1 \sigma_j^2 + 2 \gamma_2^2 \sigma_j^4 \big ) \big ]. 
\end{gathered} \end{equation}
From solving the cubic equation, we get that $\lambda = -\sigma_j \gamma_2, -\sigma_j \gamma_2 \pm \sqrt{\sigma_j^2 \gamma_2^2 - 4  \gamma_1}$. For each $\lambda$, we determine the corresponding $\rho$, that is, 
\begin{align*}
    \lambda &= - \sigma_j \gamma_2 \quad \Rightarrow \quad \rho = \theta\\
    \lambda &= -\sigma_j \gamma_2 \pm \sqrt{\sigma_j^2 \gamma_2^2 - 4 \gamma_1} \quad \Rightarrow \quad \rho = \theta \bigg (1 \mp \frac{\sigma_j \gamma_2}{\sqrt{\sigma_j^2\gamma_2^2 - 4 \gamma_1}} \bigg ).
\end{align*}
As a result, the three fundamental solutions are 
\begin{equation} \begin{aligned} \label{eq:fund_sol_near_infinity}
    j_1(t) &= e^{-\sigma_j \gamma_2 t} (\sigma_j+t)^{\theta} \left (1 + \tfrac{c_1}{\sigma_j+t} + \mathcal{O}(t^{-2})\right )\\
    j_2(t) &= e^{-\sigma_j \gamma_2 t} (\sigma_j+t)^{\theta} \exp \left (\sqrt{\sigma_j^2 \gamma_2^2 - 4  \gamma_1} \big [t - \log(\sigma_j+t) \tfrac{\theta \sigma_j \gamma_2}{\sigma_j^2 \gamma_2^2 - 4  \gamma_1} \big ] \right )\left (1 + \tfrac{c_1}{\sigma_j+t} + \mathcal{O}(t^{-2})\right )\\
    j_3(t) &= e^{-\sigma_j \gamma_2 t} (\sigma_j+t)^{\theta} \exp \left (-\sqrt{\sigma_j^2 \gamma_2^2 - 4  \gamma_1} \big [ t - \log(\sigma_j+t) \tfrac{\theta \sigma_j \gamma_2}{\sigma_j^2 \gamma_2^2 - 4  \gamma_1} \big ] \right )\left (1 + \tfrac{c_1}{\sigma_j+t} + \mathcal{O}(t^{-2})\right ).
\end{aligned}
\end{equation}

\paragraph{The neighborhood of zero.}
We may follow the same approach in a neighborhood of $\sigma_j+t=0,$ where \eqref{eqE:ODE} has a regular singular point.  The solutions are now controlled by the \emph{indicial equation} of the differential equation, which is given by
\begin{equation}\label{eqE:indicial}
\mathfrak{I}(\lambda)
\coloneqq
\lambda(\lambda-1)(\lambda-2)
-
3\theta \lambda(\lambda-1)
+
(5\theta + 2\theta^2) \lambda
-(4\theta+4\theta^2)=0.
\end{equation}
This is polynomial is explicitly factorizable by
\[
\mathfrak{I}(\lambda)
=
(\lambda-2\theta)(\lambda-(1+\theta))(\lambda - 2).
\]
Hence when $\theta$ is not an integer, there are three fundamental solutions
\begin{equation}\label{eqE:near0}
\begin{aligned}
\mathfrak{j}_1(t) &= (\sigma_j + t)^{2\theta}
\bigl(1 
+ \mathfrak{a}_{11} (\sigma_j + t) 
+ \mathfrak{a}_{21} (\sigma_j + t)^2
+ \cdots \bigr)
\eqqcolon 
(\sigma_j + t)^{2\theta}\mathfrak{a}_{1}(t), \\
\mathfrak{j}_2(t) &= (\sigma_j + t)^{1+\theta}
\bigl(1 
+ \mathfrak{a}_{12} (\sigma_j + t) 
+ \mathfrak{a}_{22} (\sigma_j + t)^2
+ \cdots \bigr)
\eqqcolon(\sigma_j + t)^{1+\theta}\mathfrak{a}_{2}(t), \\
\mathfrak{j}_3(t) &= (\sigma_j + t)^{2}
\bigl(1 
+ \mathfrak{a}_{13} (\sigma_j + t) 
+ \mathfrak{a}_{23} (\sigma_j + t)^2
+ \cdots \bigr)
\eqqcolon(\sigma_j + t)^{2}\mathfrak{a}_{3}(t). \\
\end{aligned}
\end{equation}
The coefficients of these recurrences are defined by a recurrence.  For the case of $\mathfrak{a}_1$, this recurrence is given by
\[
\mathfrak{a}_{j1}\mathfrak{I}(2\theta + j)
=
\mathfrak{a}_{(j-1)1}
\bigl\{
-3\gamma_2\sigma_j (2\theta +j-1)(2\theta + j - 2)
\bigr\} + \mathcal{O}(j),
\]
where we take coefficients $\mathfrak{a}_{k1}=0$ for $k$ negative.
The error term also depends on previous coefficients
\(
\{
\mathfrak{a}_{(j-1)1},
\mathfrak{a}_{(j-2)1},
\mathfrak{a}_{(j-3)1}
\},
\)
and the other coefficients in \eqref{eqE:ODE}.  In particular, we may bound this recurrence by
\[
|\mathfrak{a}_{j1}|
\leq
\bigl(
\max_{0 \leq k \leq j-1}
|\mathfrak{a}_{k1}|
\bigr)
\frac
{
3\gamma_2\sigma_j (2\theta +j-1)(2\theta + j - 2)
+ 
M(\gamma_1,\gamma_2,\sigma_j,\theta)(2\theta+j-2)
}
{\mathfrak{I}(2\theta + j)},
\]
where the function $M$ is a continuous function of its parameters on all $\RR^4$.
By induction, we conclude that
\[
|\mathfrak{a}_{j1}|
\leq \frac{\Gamma(2\theta + 1)}{\Gamma(2\theta+1+j)}
(3\gamma_2 \sigma_j)^j(2\theta + j)^{M(\gamma_1,\gamma_2,\sigma_j,\theta)}.
\]
Applying this argument to the other sequences, we conclude that:
\begin{lemma}\label{lemE:a}
There is a continuous function $M\coloneqq M(\gamma_1,\gamma_2,\sigma_j,\theta) \geq 0$ on $\RR^4$ so that
\[
\biggl\|
\begin{bmatrix}
\mathfrak{a}_1(t)
&
\mathfrak{a}_2(t)
&
\mathfrak{a}_3(t) \\
\mathfrak{a}_1'(t)
&
\mathfrak{a}_2'(t)
&
\mathfrak{a}_3'(t) \\
\mathfrak{a}_1''(t)
&
\mathfrak{a}_2''(t)
&
\mathfrak{a}_3''(t) \\
\end{bmatrix}
\biggl\|
\leq (2+\sigma_j + t)^{M}e^{3\gamma_2 \sigma_j t}
\quad\text{for all}\quad t \geq 0.
\]
\end{lemma}

We will also need some estimates on the fundamental matrix built from these solutions.  Define
\[
\mathfrak{P}(t)
= 
\begin{bmatrix}
\mathfrak{j}_1(t)
&
\mathfrak{j}_2(t)
&
\mathfrak{j}_3(t) \\
\mathfrak{j}_1'(t)
&
\mathfrak{j}_2'(t)
&
\mathfrak{j}_3'(t) \\
\mathfrak{j}_1''(t)
&
\mathfrak{j}_2''(t)
&
\mathfrak{j}_3''(t) \\
\end{bmatrix}
\underset{t\to -\sigma_j}{\sim}
\begin{bmatrix*}[l]
(t+\sigma_j)^{2\theta} & 
(t+\sigma_j)^{1+\theta} & 
(t+\sigma_j)^{2} \\
2\theta(t+\sigma_j)^{2\theta-1} & 
(1+\theta)(t+\sigma_j)^{\theta} & 
2(t+\sigma_j)^{1} \\
(2\theta)(2\theta+1)(t+\sigma_j)^{2\theta-2} & 
(1+\theta)\theta(t+\sigma_j)^{\theta-1} & 
2 \\
\end{bmatrix*}.
\]
In particular the Wronskian of $\mathfrak{P}$ satisfies
\[
\det \mathfrak{P}(t)
\underset{t\to -\sigma_j}{\sim}
(t+\sigma_j)^{3\theta}
\det
\begin{bmatrix*}[l]
1 & 
1 & 
1 \\
2\theta & 
1+\theta & 
2 \\
(2\theta)(2\theta+1) & 
(1+\theta)\theta & 
2 \\
\end{bmatrix*}
\eqqcolon 
p_\theta (t+\sigma_j)^{3\theta}.
\]
We conclude from \eqref{eqE:Wronskian} that for any $t,\epsilon \geq -\sigma_j,$
\begin{equation}\label{eqE:PWr}
\det \mathfrak{P}(t)
=
{\det \mathfrak{P}(\epsilon)}
\frac
{(t+\sigma_j)^{3\theta}}
{(\epsilon+\sigma_j)^{3\theta}}
e^{-3\gamma_2\sigma_j(t-\epsilon)}
\underset{\epsilon\to0}{\longrightarrow}
{(t+\sigma_j)^{3\theta}}
e^{-3\gamma_2\sigma_j(t+\sigma_j)}.
\end{equation}
We conclude that:
\begin{lemma}\label{lemE:P}
There is a continuous function $M\coloneqq M(\gamma_1,\gamma_2,\sigma_j,\theta) \geq 0$ on $\RR^4$ so that for all $t > -\sigma_j$
\[
\| \mathfrak{P}(t)\|
\leq (2+\sigma_j+t)^{M+2\theta}e^{3\gamma_2 \sigma_j t}
\quad\text{and}\quad
\| \mathfrak{P}^{-1}(t)\|
\leq M\frac{(2+\sigma_j+t)^{M+2\theta}e^{9\gamma_2 \sigma_j t}}{(\sigma_j+t)^{2\theta}}.
\]
\end{lemma}
\begin{proof}
The first bound follows directly from Lemma \ref{lemE:a}.  The second bound follows from Cram\'er's rule.  We note that some entries of the inverse matrix are singular at $0,$ but all have the form of
\[
\frac{ 
\mathfrak{j}_a^{(i-1)}\mathfrak{j}_b^{(j-1)} 
-\mathfrak{j}_b^{(i-1)}\mathfrak{j}_a^{(j-1)}
}
{\det \mathfrak{P}(t)},
\]
for some $a,b,i,j \in \{1,2,3\}$.  Hence the smallest positive power of $(t+\sigma_j)$ is achieved by taking all $a,b,i,j \in \{2,3\}$.  
\end{proof} 

\paragraph{Improved bounds at the singular point.}

When $\gamma_2^2\sigma_j^2 - 4\gamma_1 \approx 0,$ the solutions constructed near infinity degenerate.  We may however show that the solutions constructed near $0$ in fact have the correct exponential behavior at infinity.  We observe that we may always represent the differential equation \eqref{eqE:ODE} by
\begin{equation}\label{eqE:conjugatedODE}
\begin{aligned}
&(\widetilde{J}e^{\sigma_j \gamma_2 t})^{(3)}
+
\varepsilon_2 (\widetilde{J}e^{\sigma_j \gamma_2 t})^{(2)}
+
\varepsilon_1 (\widetilde{J}e^{\sigma_j \gamma_2 t})^{(1)}
+
\varepsilon_0 (\widetilde{J}e^{\sigma_j \gamma_2 t})
=0, \quad\text{where} \\
&\varepsilon_2 = 
-\frac{3\theta}{\sigma_j+t},
\quad
\varepsilon_1 = 
4\gamma_1-\gamma_2^2\sigma_j^2
+\frac{2\gamma_2\sigma_j\theta}{\sigma_j+t}
+\frac{2\theta^2+5\theta}{(\sigma_j+t)^2},
\quad\text{and}\\
&\varepsilon_0 = 
-\frac{(4\gamma_1 - \gamma_2^2\sigma_j^2)\theta}{\sigma_j+t}
+\frac{(2\theta^2-\theta)\gamma_2\sigma_j}{(\sigma_j+t)^2}
-\frac{4(\theta^2+\theta)\gamma_2\sigma_j}{(\sigma_j+t)^3}.
\end{aligned}
\end{equation}
Hence with $Y \coloneqq \widetilde{J}e^{\sigma_j \gamma_2 t},$ we can represent
\[
\begin{bmatrix}
Y^{(1)}(t) \\
Y^{(2)}(t) \\
Y^{(3)}(t) \\
\end{bmatrix}
=
\begin{bmatrix}
0 & 1 & 0 \\
0 & 0 & 1 \\
-\varepsilon_0(t) & -\varepsilon_1(t) & -\varepsilon_2(t) \\
\end{bmatrix}
\begin{bmatrix}
Y(t) \\
Y^{(1)}(t) \\
Y^{(2)}(t) \\
\end{bmatrix}
\]
Hence if we let 
\[
\delta^2 \coloneqq \max\biggl\{
{|4\gamma_1-\gamma_2^2\sigma_j^2|},
(\omega_j+t)^{-1}\biggr\}
\]
we conclude, after conjugating
\begin{equation}\label{eqE:N}
\begin{bmatrix}
Y^{(1)}(t) \\
\delta^{-1}Y^{(2)}(t) \\
\delta^{-2}Y^{(3)}(t) \\
\end{bmatrix}
=
\begin{bmatrix}
0 & \delta & 0 \\
0 & 0 & \delta \\
-\varepsilon_0(t)\delta^{-2} & -\varepsilon_1(t)\delta^{-1} & -\varepsilon_2(t) \\
\end{bmatrix}
\begin{bmatrix}
Y(t) \\
\delta^{-1}Y^{(1)}(t) \\
\delta^{-2}Y^{(2)}(t) \\
\end{bmatrix}.
\end{equation}
We define
\[
N\coloneqq\max\{
|Y|
,|\tfrac{Y'}{\delta}|
,|\tfrac{Y''}{\delta^2}|
\}
\quad\text{and}
\quad
A\coloneqq
\max\{ \delta, 
|\varepsilon_2| 
+ |\varepsilon_1\delta^{-1}|
+ |\varepsilon_0\delta^{-2}|\}
\]
which are $\ell^{\infty}$ norms of the matrix and vector that appear on the right-hand-side of \eqref{eqE:N}.  Moreover, taking the time derivative of $N$ we conclude the differential inequality
\[
N'(t) \leq A(t)N(t) + \mathbf{1}[\delta^2 = (\omega_j+t)^{-1}]\frac{N(t)}{\omega_j+t}.
\]
There is a continuous function $M$ of the parameters $(\gamma_1,\gamma_2,\sigma_j,\theta)$ so that $A(t)$ can be bounded by a multiple of $M \delta(t)$ for all $t \geq 1.$
Applying Gronwall's inequality, it follows that for any $t \geq t_0 \geq 1$ that 
\begin{equation}\label{eqE:gronwall}
N(t)
\leq
N(t_0)
\exp
\biggl(
M
\int_{t_0}^t
\delta(s)\dif s
\biggr).
\end{equation}
We use this to conclude the fundamental matrix $\mathfrak{P}$ has reasonable decay properties for $4\gamma_1-\gamma_2^2\sigma_j^2$ small.
\begin{lemma}\label{lemE:singularpoint}
Let $\omega \coloneqq 4\gamma_1-\gamma_2^2\sigma_j^2$ and suppose that $|\omega| \leq 1.$
There is a continuous function $M \coloneqq M(\gamma_1,\gamma_2,\sigma_j,\theta) \geq 0$ on $\RR^4$ so that for all $t \geq s \geq 1$ 
\[
\|\mathfrak{P}(t)\| \leq M e^{-(\gamma_2\sigma_j - M\sqrt{|\omega|})t}
\quad\text{and}\quad
\|\mathfrak{P}(t)\mathfrak{P}^{-1}(s)\| \leq M e^{-(\gamma_2\sigma_j -M\sqrt{|\omega|})(t-s)}
\]
\end{lemma}
\noindent Note that the first inequality extends to all $t \geq -\sigma_j$ using Lemma \ref{lemE:P}.
\begin{proof}
For the first bound, we apply \eqref{eqE:gronwall} to each of $\mathfrak{j}_1,\mathfrak{j}_2,\mathfrak{j}_3$
separately. As $|\omega|=\delta^2 \leq 1$ and $Y(1),Y'(1),Y''(1)$ can all be bounded using Lemma \ref{lemE:P} by a continuous function, we conclude for some possibly larger $M > 0$
\[
Y''(t) \leq Me^{M t\delta}.
\]
By integrating this bound, we conclude, by increasing $M$ as needed that there is some $M$ so that
\[
\max\{
|Y(t)|
,|{Y'(t)}|
,|{Y''(t)}|
\}
\leq Me^{M t\delta}.
\]
With $Y(t)=\mathfrak{j}_a(t)e^{\sigma_j \gamma_2 t},$ for $a \in \{1,2,3\}$ expressing the left-hand-side of the above in terms of $\mathfrak{j}_a$ and again increasing $M$ as needed, we conclude the first claimed bound.

For the second bound, the columns of $\mathfrak{P}(t)\mathfrak{P}^{-1}(s)$ solve \eqref{eqE:ODE} and they have identity initial conditions at $s$.  Hence applying \eqref{eqE:gronwall} to each, we derive the desired equation in the same fashion as above.
\end{proof}

\subsection{Near infinity}
Throughout this section, we work in the regime that 
\[|4 \gamma_1 - \gamma_2^2 \sigma_j^2| > \varepsilon\]
for some positive $\varepsilon$. This regime ensures that all the roots of the indicial equation for the ODE \eqref{eqE:ODE} are distinct near $\infty$. We are interested in deriving an expression for the kernel $K_{s}(t)$ in Corollary~\ref{cor:kernel} when $s$ and $t$ are large. Recall the three fundamental solutions near infinity that is $j_i(t)$ \eqref{eq:fund_sol_near_infinity}. We begin by defining three different boundary conditions that will aid us in finding the kernels we are interest in, that is, 
\begin{equation}\label{eqE:fuu} \begin{aligned}
\text{(Dirichlet sol., $\widetilde{\mathcal{D}}_s(t)$)} \quad \widetilde{L}[\widetilde{\mathcal{D}}_s(t)] = 0 \quad \text{where} \quad  \widetilde{\mathcal{D}}_s(s) = (1,0,0)^T\\
\text{(Neumann sol., $\widetilde{\mathcal{N}}_s(t)$)} \quad \widetilde{L}[\widetilde{\mathcal{N}}_s(t)] = 0 \quad \text{where} \quad  \widetilde{\mathcal{N}}_s(s) = (0,1,0)^T\\
\text{(2nd derivative sol., $\widetilde{\mathcal{H}}_s(t)$)} \quad \widetilde{L}[\widetilde{\mathcal{H}}_s(t)] = 0 \quad \text{where} \quad  \widetilde{\mathcal{H}}_s(s) = (0,0,1)^T
\end{aligned}
\end{equation}
We will compute the asymptotics for the Dirichlet solution $\widetilde{\mathcal{D}}_s(t)$ in full details; we leave out the details for the other two solutions noting that the same approach works. 

Using the fundamental solutions near infinity, we can write the Dirichlet solution as a linear combination of the fundamental solutions, 
\[\widetilde{\mathcal{D}}_s(t) = c_1^D(s) j_1(t) + c_2^D(s) j_2(t) + c_3^D(s) j_3(t). \]
We need to find the coefficients $c_1, c_2, c_3$ and to do so we utilize the fundamental matrix $\Phi(s)$,
\begin{align}
    \Phi(s) \defas \begin{bmatrix}
    j_1(s) & j_2(s) & j_3(s)\\
    j_1'(s) & j_2'(s) & j_3'(s)\\
    j_1''(s) & j_2''(s) & j_3''(s).
    \end{bmatrix}
\end{align}
In particular the coefficients $c_i^D$ are found by $\Phi(s) (c_1^D, c_2^D, c_3^D)^T = (1,0,0)^T$. Hence, we need to compute the inverse of $\Phi(s)$ which we do by Cramer's rule. First, we need an expression for the Wroskian $\mathscr{W}(s)$ which we wrote in \eqref{eqE:Wronskian} as a ratio. Since we are working in a neighborhood of $t = \infty$, we can compute the Wroskian using $t = \infty$ and therefore derive an expression for the $\mathscr{W}(s)$ for any $s$. A simple calculation yields the following expression
\begin{align*}
    \mathscr{W}(t) = \text{det} \,  \Phi(t) &\sim \exp((\lambda_1 + \lambda_2 + \lambda_3)t) (\sigma_j + t)^{\rho_1 + \rho_2 + \rho_3} \, \text{det} \left ( \begin{bmatrix}
    1 & 1 & 1\\
    \lambda_1 & \lambda_2 & \lambda_3\\
    \lambda_1^2 & \lambda_2^2 & \lambda_3^2
    \end{bmatrix} \right )\\
    &= \exp((\lambda_1 + \lambda_2 + \lambda_3)t) (\sigma_j + t)^{\rho_1 + \rho_2 + \rho_3} (\lambda_3-\lambda_2) (\lambda_3-\lambda_1) (\lambda_2-\lambda_1)
\end{align*}
where the pair $(\lambda_i, \rho_i)$ corresponds to the fundamental solution $j_i$ and the determinant of the matrix is the determinant of the Vandermonde matrix. It is clear that $\lambda_1 + \lambda_2 + \lambda_3 = -3 \sigma_j \gamma_2$, $\rho_1 + \rho_2 + \rho_3 = 3 \theta$, and $(\lambda_3-\lambda_2)(\lambda_3-\lambda_1)(\lambda_2-\lambda_1) = 2 (\sigma_j^2 \gamma_2^2 - 4 \gamma_1)^{3/2}$. Hence, we have that $\mathscr(t) \sim 2\exp(-3\sigma_j^2 \gamma_1 t) (\sigma_j + t)^{3 \theta} (\sigma_j^2 \gamma_2^2-4 \gamma_1)^{3/2} $. Combining this with \eqref{eqE:Wronskian}, we get that
\begin{align}\label{eqE:Ws}
    \frac{1}{\mathscr{W}(s)} = \frac{e^{3 \gamma_2 \sigma_j s} (\sigma_j + s)^{-3 \theta}}{2(\sigma_j^2 \gamma_2^2 - 4 \gamma_1)^{3/2}}.
\end{align}
Here we used that we working in the regime where the denominator is bounded away from $0$. By Cramer's rule, we have an expression for the fundamental matrix $\Phi^{-1}(s)$
\begin{align*}
    \Phi^{-1}(s) &= \frac{1}{\mathscr{W}(s)} \begin{bmatrix}
    j_2'j_3''-j_2''j_3' & j_2''j_3-j_2j_3'' & j_2j_3'-j_2' j_3\\
    j_1''j_3'-j_1' j_3'' & j_1j_3''-j_1''j_3 & j_1'j_3-j_1j_3'\\
    j_1'j_2''-j_2'j_1'' & j_2j_1''-j_1j_2'' & j_1j_2' - j_1'j_2
    \end{bmatrix} \\
    &\sim \frac{1}{2(\sigma_j^2 \gamma_2 - 4 \gamma_1)^{3/2}} \begin{bmatrix} (\lambda_2 \lambda_3^2-\lambda_2^2 \lambda_3)j_1^{-1} & (\lambda_2^2 -\lambda_3^2)j_1^{-1} &(\lambda_3-\lambda_2) j_1^{-1}\\
    (\lambda_3\lambda_1^2-\lambda_3^2\lambda_1)j_2^{-1} & (\lambda_3^2-\lambda_1^2) j_2^{-1} & (\lambda_1-\lambda_3)j_2^{-1}\\
    (\lambda_1\lambda_2^2-\lambda_1^2 \lambda_2 )j_3^{-1} & (\lambda_1^2-\lambda_2^2) j_3^{-1} & (\lambda_2-\lambda_1) j_3^{-1}
    \end{bmatrix}\\
    &= \frac{1}{2(\sigma_j^2 \gamma_2 - 4 \gamma_1)^{3/2}}
    \begin{bmatrix}
    \lambda_2 \lambda_3 (\lambda_3-\lambda_2) j_1^{-1} & (\lambda_2-\lambda_3)(\lambda_2+\lambda_3) j_1^{-1} & (\lambda_3-\lambda_2) j_1^{-1}\\
    \lambda_3\lambda_1 (\lambda_1-\lambda_3) j_2^{-1} & (\lambda_3-\lambda_1) (\lambda_3+\lambda_1) j_2^{-1} & (\lambda_1-\lambda_3) j_2^{-1}\\
    \lambda_1\lambda_2 (\lambda_2-\lambda_1)j_3^{-1} & (\lambda_1-\lambda_2)(\lambda_1+\lambda_2) j_3^{-1} &(\lambda_2-\lambda_1) j_3^{-1}
    \end{bmatrix}.
\end{align*}
We used that $j_i^{(\ell)}(s) \sim \lambda^{\ell}_i j_i(s)$ and we simplified the Wronskian by pulling out the appropriate terms. To simplify some of the terms, we let $\omega(\sigma) = 4 \gamma_1 - \gamma_2^2 \sigma^2$. We can now compute the coefficients for the various boundary solutions using that $2 (\sigma_j^2 \gamma_2^2 - 4 \gamma_1)^{3/2} = (\lambda_3-\lambda_2)(\lambda_3-\lambda_1)(\lambda_2-\lambda_1)$
\begin{align*}
    (c_1^D(s), c_2^D(s), c_3^D(s) )^T &= \left ( \frac{4\gamma_1}{\omega} j_1^{-1}, \frac{\sigma_j \gamma_2 ( \sigma_j \gamma_2 + \sqrt{-\omega})}{-2\omega} j_2^{-1}, \frac{\sigma_j \gamma_2 (\sigma_j \gamma_2 - \sqrt{-\omega})}{-2 \omega} j_3^{-1} \right )^T (1 + \mathcal{O}(\varepsilon) )\\
    (c_1^N(s), c_2^N(s), c_3^N(s) )^T &= \left (\frac{2\sigma_j \gamma_2}{\omega} j_1^{-1}, \frac{2\sigma_j \gamma_2 + \sqrt{-\omega}}{-2\omega} j_2^{-1}, \frac{2\sigma_j \gamma_2 - \sqrt{-\omega}}{-2\omega} j_3^{-1} \right )^T (1 + \mathcal{O}(\varepsilon) )\\
     (c_1^H(s), c_2^H(s), c_3^H(s) )^T &= \left ( \frac{j_1^{-1}}{\omega}, \frac{-j_2^{-1}}{2\omega}, \frac{-j_3^{-1}}{2\omega} \right )^T (1 + \mathcal{O}(\varepsilon) ).
\end{align*}

We conclude the following bounds for the fundamental matrix:
\begin{lemma}\label{lemE:fine}
Let $\varepsilon > 0$ be arbitrary and suppose that $\omega = 4 \gamma_1 - \gamma_2^2 \sigma_j^2$ satisfies $|\omega| > \varepsilon$.  There is a continuous function $M_\epsilon \coloneqq M_\epsilon(\gamma_1,\gamma_2,\theta)$ so that for all $t \geq s \geq 1,$
\[
\|\Phi(t)\Phi^{-1}(s)\|
\leq M_\epsilon e^{-(\gamma_j\gamma_2-\sqrt{\max\{-\omega,0\}})(t-s)}\frac{ (\sigma_j+t)^{\theta}}{(\sigma_j+s)^{\theta}}.
\]
If on the other hand, $s \leq 1,$ we instead have
\[
\|\Phi(t)\Phi^{-1}(s)\|
\leq M_\epsilon e^{-(\gamma_j\gamma_2-\sqrt{\max\{-\omega,0\}})(t-s)}\frac{ (\sigma_j+t)^{\theta}}{(\sigma_j+s)^{2\theta}}.
\]
Finally, we have the asymptotic representation for the fundamental solutions for $t \geq s$
\[
\begin{aligned}
\widetilde{\mathcal{H}}_s(t) 
&= 
\biggl(\tfrac{1-\cos\bigl(\sqrt{\omega}(t-s)+\log\bigl(\tfrac{\sigma_j+t}{\sigma_j+s}\bigr)\tfrac{2\gamma_2\sigma_j\theta}{\sqrt{\omega}}\bigr)+o_{s,\varepsilon}(1)}{{\omega}}\biggr)\exp\bigl(-\sigma_j\gamma_2 (t-s)\bigr)\frac{ (\sigma_j+t)^{\theta}}{(\sigma_j+s)^{\theta}}, \\
\widetilde{\mathcal{N}}_s(t) &= 2\sigma \gamma_2 \widetilde{\mathcal{H}}_s(t) + 
\biggl(\tfrac{\sin\bigl(\sqrt{\omega}(t-s)+\log\bigl(\tfrac{\sigma_j+t}{\sigma_j+s}\bigr)\tfrac{2\gamma_2\sigma_j\theta}{\sqrt{\omega}}\bigr)
+o_{s,\varepsilon}(1)}{\sqrt{\omega}}\biggr)\exp\bigl(-\sigma_j\gamma_2 (t-s)\bigr)\frac{ (\sigma_j+t)^{\theta}}{(\sigma_j+s)^{\theta}}, \\
\widetilde{\mathcal{D}}_s(t) &= \sigma_j\gamma_2 \widetilde{\mathcal{N}}_s(t) 
-(\sigma_j\gamma_2)^2 \widetilde{\mathcal{H}}_s(t)
+(1+o_{s,\varepsilon}(1))\exp\bigl(-\sigma_j\gamma_2 (t-s)\bigr)\frac{ (\sigma_j+t)^{\theta}}{(\sigma_j+s)^{\theta}}, \\
\end{aligned}
\]
where the error $o_{s,\varepsilon}(1)$ tends to $0$ uniformly with $s$ uniformly on compact sets of the parameter space where $|\omega|>\varepsilon.$
\end{lemma}
\begin{proof}
The bounds follow from estimating above the fundamental solutions \eqref{eq:fund_sol_near_infinity} and the Wronskian formula \eqref{eqE:Ws}.  By combining these with Lemma \ref{lemE:P}, we can extend the formula to $s \leq 1,$ (using
\[
\Phi(t)\Phi^{-1}(s)
=
\Phi(t)\Phi^{-1}(1)\Phi(1)\Phi^{-1}(s)
=
\Phi(t)\Phi^{-1}(1)\mathfrak{P}(1)\mathfrak{P}^{-1}(s),
\]
using uniqueness of the IVP).  The final asymptotics follow from the display above the statement of the lemma.
\end{proof}

\paragraph{The scaling limit of the fundamental solutions as $\sigma$ tends to $0$.}

We conclude with:
\begin{lemma}\label{lemE:scaling}
Let $\mathfrak{d}$ be the solution of 
\eqref{eqE:ODE} on $[0,\infty)$ with $\sigma=0$ that has the property that $\mathfrak{d}(t)t^{-2\theta} \to 1$ as $t \to 0$.  This could also be expressed as the limit as $\sigma\to 0$ of $\mathfrak{j}_1$.
For any $\varepsilon >0$
\[
\begin{aligned}
&\sup_{t \in [\varepsilon,\infty)} 
(\sigma_j + t)^{-\theta}
|\sigma^{2\theta}\cdot\widetilde{\mathcal{D}}_0(t) - \mathfrak{d}(t)|
\underset{\sigma \to 0}{\longrightarrow} 0, \\
&\sup_{t \in [\varepsilon,\infty)} 
(\sigma_j + t)^{-\theta}
|\sigma^{2\theta-1}\cdot\widetilde{\mathcal{N}}_0(t) - \mathfrak{d}(t)|
\underset{\sigma \to 0}{\longrightarrow} 0, \\
&\sup_{t \in [\varepsilon,\infty)} 
(\sigma_j + t)^{-\theta}
|\sigma^{2\theta-2}\cdot\widetilde{\mathcal{H}}_0(t) - \mathfrak{d}(t)|
\underset{\sigma \to 0}{\longrightarrow} 0, \\
\end{aligned}
\]
\end{lemma}
\begin{proof}
We can represent 
$\widetilde{\mathcal{D}}_0(t),$
$\widetilde{\mathcal{N}}_0(t),$
$\widetilde{\mathcal{H}}_0(t)$
as the entries in the first row of
\[
\Phi(t)\Phi^{-1}(0)
=
\Phi(t)\Phi^{-1}(1)
\mathfrak{P}(1)\mathfrak{P}^{-1}(0).
\]
The matrix in the middle $\Phi^{-1}(1)\mathfrak{P}(1)$ converges to a nondegenerate matrix as $\sigma \to 0$.
The first row of $\Phi(t),$ given by $j_1,j_2,j_3$ each converge to solutions $a_1,a_2,a_3$ of \eqref{eqE:ODE} as $\sigma \to 0$ in the sense that for any $\varepsilon >0$
\[
\sup_{t \in [\varepsilon,\infty)} 
(\sigma + t)^{-\theta}
|j_k(t) - a_k(t)|
\underset{\sigma \to 0}{\longrightarrow} 0.
\]
The columns of $\mathfrak{P}^{-1}(0)$ 
behave like
\[
\mathfrak{P}^{-1}(0)
\underset{\sigma \to 0}{\asymp}
\begin{bmatrix*}[l]
\sigma^{-2\theta} & 
\sigma^{1-2\theta} & 
\sigma^{2-2\theta} \\
\sigma^{-1-\theta} & 
\sigma^{-\theta} & 
\sigma^{1-\theta} \\
\sigma^{-2} & 
\sigma^{-1} & 
1 \\
\end{bmatrix*},
\]
where we mean that the ratios of the respective entries converge to a nonzero constant as $\sigma\to 0$. To see that the limits that result are always equal to $\mathfrak{d},$ we can instead represent $\widetilde{\mathcal{D}}_0(t)$,
$\widetilde{\mathcal{N}}_0(t)$,
$\widetilde{\mathcal{H}}_0(t)$ as the first row of $\mathfrak{P}(t)\mathfrak{P}^{-1}(0)$.  On taking $\sigma \to 0,$ only the multiple of $\mathfrak{j}_1$ survives.
\end{proof}

\paragraph{The fundamental solutions of the unscaled ODE.}
Finally, we relate the estimates we have made back to the unscaled differential differential equation \eqref{eqE:ODE} and \eqref{eqE:fu}.  So we set
\begin{equation}\label{eqE:fu} \begin{aligned}
\text{(Dirichlet sol., ${\mathcal{D}}_s(t)$)} \quad &{L}[{\mathcal{D}}_s(t)] = 0 \quad \text{where} \quad  {\mathcal{D}}_s(s) = (1,0,0)^T,\\
\text{(Neumann sol., ${\mathcal{N}}_s(t)$)} \quad &{L}[{\mathcal{N}}_s(t)] = 0 \quad \text{where} \quad  {\mathcal{N}}_s(s) = (0,1,0)^T,\\
\text{(2nd derivative sol., ${\mathcal{H}}_s(t)$)} \quad &{L}[{\mathcal{H}}_s(t)] = 0 \quad \text{where} \quad  {\mathcal{H}}_s(s) = (0,0,1)^T.
\end{aligned}
\end{equation}
To make the connection to \eqref{eqE:fu}, we observe that an initial value problem
\[
{L}[f(t)] = 0 \quad \text{and} \quad
\begin{bmatrix}
f(t_0) \\
f'(t_0) \\
f''(t_0) \\
\end{bmatrix}
=\begin{bmatrix}
c_1 \\
c_2 \\
c_3 \\
\end{bmatrix}
\, \, \longleftrightarrow \, \,
\widetilde{L}[f(t/\sigma_j)]
= 0 
\, \, \text{and} \, \,
\begin{bmatrix}
f(t_0/\sigma_j) \\
\partial_{t_0/\sigma_j} f(t_0/\sigma_j) \\
\partial^2_{t_0/\sigma_j} f(t_0/\sigma_j) \\
\end{bmatrix}
=\begin{bmatrix}
c_1 \\
c_2/\sigma_j \\
c_3/\sigma_j^2 \\
\end{bmatrix}.
\]
Thus we have the identification
\begin{equation}\label{eqE:DNH}
{\mathcal{D}}^{(\sigma_j)^2}_s(t) =
\widetilde{\mathcal{D}}_{s\sigma_j}(t\sigma_j),
\quad
{\mathcal{N}}^{(\sigma_j)^2}_s(t) =
\widetilde{\mathcal{N}}_{s\sigma_j}(t\sigma_j)/\sigma_j,
\quad
{\mathcal{H}}^{(\sigma_j)^2}_s(t) =
\widetilde{\mathcal{H}}_{s\sigma_j}(t\sigma_j)/\sigma_j^2.
\end{equation}
Using Lemma \ref{lemE:fine}, it is possible to give asymptotic expressions for these kernels and corresponding estimates.  

We recall that we can express the terms in the Volterra equation for SDANA as
\[
\Exp f(\XX_t) = R h_1(t) + \widetilde{R} h_0(t) + \int_0^t \mathcal{K}_s(t) \Exp f(\XX_s)\,\dif s,
\]
for a given spectral measure $\mu$ (especially,\ the empirical spectral measure or the limiting empirical spectral measure) by
\begin{equation}\label{eqE:fk}
\begin{aligned}
{G}^{(\sigma^2)}(t)
&\coloneqq
\biggl(
{\mathcal{D}}^{(\sigma^2)}_0(t)
+
(2\theta-2\gamma_2\sigma^2){\mathcal{N}}^{(\sigma^2)}_0(t)
+((2\theta-2\gamma_2\sigma^2)^2 - 2\gamma_1\sigma^2-2\theta)
{\mathcal{H}}^{(\sigma^2)}_0(t)
\biggr), \\
h_0(t) &\coloneqq \frac{1}{2\varphi^2(t)}
\int\limits_0^\infty 
{G}^{(\sigma^2)}(t)
\mu(\dif \sigma^2)
\quad
\text{and}
\quad
h_1(t) \coloneqq \frac{1}{2\varphi^2(t)}
\int\limits_0^\infty 
\sigma^2
{G}^{(\sigma^2)}(t)
\mu(\dif \sigma^2) \\
\mathcal{K}_s(t) &\coloneqq
\frac{\varphi^2(s)}{\varphi^2(t)}
\int\limits_0^\infty 
\sigma^4
\biggl(
\gamma_2^2
{\mathcal{D}}^{(\sigma^2)}_s(t)
+
\biggl(
2\gamma_2\gamma_1 + 2\gamma_2^2\bigl(\tfrac{\theta}{1+s}-\gamma_2\sigma^2\bigr)
\biggr){\mathcal{N}}^{(\sigma^2)}_s(t)
\biggr)
\mu(\dif \sigma^2) \\
&+
\frac{\varphi^2(s)}{\varphi^2(t)}
\int\limits_0^\infty 
\sigma^4
\biggl(
{\gamma_2^2} \left [ \tfrac{4\theta^2-2\theta}{(1+s)^2} - \tfrac{8\theta \gamma_2 \sigma^2}{1+s} + 4 \sigma^4 \gamma_2^2  \right ] - 8 \gamma_1\gamma_2^2 \sigma^2+ \tfrac{6 \theta \gamma_2\gamma_1}{1+s} + 2 \gamma_1^2\biggr)
{\mathcal{H}}^{(\sigma^2)}_s(t)
\mu(\dif \sigma^2) 
\end{aligned}
\end{equation}

\paragraph{Reduction to a convolution kernel.}
We work under the assumption that the support of $\mu$ is contained in $[\lambda_-,\lambda_+]$.

For the kernel, we start by using the asymptotics for $\mathcal{D},\mathcal{N},\mathcal{H}$ in Lemma \ref{lemE:fine}, which give
\begin{equation}\label{eq:DNH1}
\begin{aligned}
{\mathcal{H}}_s(t) &= \biggl(\tfrac{1-\cos\bigl(\sigma\sqrt{\omega}(t-s)-\log\bigl(\tfrac{1+t}{1+s}\bigr)\tfrac{\theta\gamma_2\sigma}{\sqrt{\omega}}\bigr)+o_{s,\varepsilon}(1)}{\sigma^2{\omega}}\biggr)\exp\bigl(-\sigma^2\gamma_2 (t-s)\bigr)\frac{ (1+t)^{\theta}}{(1+s)^{\theta}}, \\
{\mathcal{N}}_s(t) &= 2\sigma^2 \gamma_2 {\mathcal{H}}_s(t) + 
\biggl(\tfrac{\sin\bigl(\sigma\sqrt{\omega}(t-s)-\log\bigl(\tfrac{1+t}{1+s}\bigr)\tfrac{\theta\gamma_2\sigma}{\sqrt{\omega}}\bigr)
+o_{s,\varepsilon}(1)}{\sigma\sqrt{\omega}}\biggr)\exp\bigl(-\sigma^2\gamma_2 (t-s)\bigr)\frac{ (1+t)^{\theta}}{(1+s)^{\theta}}, \\
{\mathcal{D}}_s(t) &= \sigma^2\gamma_2 {\mathcal{N}}_s(t) 
-\sigma^4\gamma_2^2 {\mathcal{H}}_s(t) +
(1+o_{s,\varepsilon}(1))\exp\bigl(-\sigma^2\gamma_2 (t-s)\bigr)\frac{ (1+t)^{\theta}}{(1+s)^{\theta}}. \\
\end{aligned}
\end{equation}
To apply these asymptotics we need to cut out a window $I_\epsilon$ of $\sigma$ for which $\omega=\omega(\sigma) = 4\gamma_1 - \gamma_2^2 \sigma^2$ is small.  So for an $\epsilon >0$ let $I_\epsilon$ be those $\sigma$ for which $|\omega|<\epsilon$.  If $\omega(\sigma)$ is bounded away from $0$ on the support of $\mu$ we may simply take $\epsilon = 0$ in what follows.  By tracking the leading terms, we arrive at 
\begin{equation}\label{eqE:KS}
\begin{aligned}
\mathcal{K}_s(t)
&=
(1+o_s(1))
\frac{2\gamma_1^2\varphi(s)}{\varphi(t)}
\int_{I_\epsilon}
\sigma^2
e^{-\sigma^2\gamma_2(t-s)}
\biggl(
\tfrac{1 - \cos\bigl(\vartheta(\sigma) + \sigma\sqrt{\omega}(t-s)-\log\bigl(\tfrac{1+t}{1+s}\bigr)\tfrac{\theta\gamma_2\sigma}{\sqrt{\omega}}\bigr)}{\omega}
\biggr)
\mu(\dif \sigma^2) \\
&+\mathcal{O}\bigl(\epsilon e^{-(4\gamma_1\gamma_2^{-1} + M\epsilon)(t-s)}\bigr)\\
\end{aligned}
\end{equation}
where $\vartheta(\sigma)$ is a phase depending on $\gamma_1,\gamma_2,\sigma,$ having $\vartheta(\sigma) \sim -2\sqrt{\gamma_1 \omega(\sigma)}$ as $\omega\to0$.  The phase is defined explicitly by
\begin{equation}\label{eqE:phase}
\begin{aligned}
&\cos(\vartheta(\sigma)) = \frac{ (\omega - 2\gamma_1)^2 - 2\gamma_1^2}{2\gamma_1^2}, \\
&\sin(\vartheta(\sigma)) = \frac{ (\omega - 2\gamma_1)\sqrt{4\gamma_1-\omega}\sqrt{\omega}}{2\gamma_1^2}. \\
\end{aligned}
\end{equation}

This is essentially a convolution type Volterra kernel, and so we simplify it by using an idealized kernel.  Define
\begin{equation}\label{eqE:Ist}
\mathcal{I}_s(t)
=
\mathcal{I}(t-s)
=
2\gamma_1^2
\int_0^\infty
\sigma^2
e^{-\sigma^2\gamma_2(t-s)}
\biggl(
\frac{1 - \cos\bigl(\vartheta(\sigma) + \sigma\sqrt{\omega}(t-s)\bigr)}{\omega}
\biggr)
\mu(\dif \sigma^2).
\end{equation}
This is a convolution kernel, which is comparable in norm to $\mathcal{K}_s(t)$.  As the theory for positive convolution kernels is substantially simpler, we turn to studying the equation:
\begin{equation}\label{eqE:Psi}
\Psi(t) = R\varphi(t)h_1(t) + \widetilde{R}\varphi(t)h_0(t) + \int_0^t \mathcal{I}(t-s)\Psi(s)\,\dif s.
\end{equation}
We will reduce the asymptotics of $\psi$ to those of $\Psi$.

A simple computation gives that the $\text{L}^1$ norm of the $\mathcal{I}$ is given by
\begin{equation}\label{eqE:convergence}
\|\mathcal{I}\|
=
\int_0^\infty \mathcal{I}(t)\,dt
=
\int_{0+}^\infty 
\tfrac{\gamma_1 + \sigma^2\gamma_2^2}{2\gamma_2}
\mu(\dif \sigma^2),
\end{equation}
which gives a sufficient condition for neighborhood convergence.  Provided that the measure $\mu$ puts no mass at the critical point, the behavior of solutions  \eqref{eqE:Psi} are related to the original Volterra equation.
\begin{proposition}\label{propG:newequivalence}
Provided $\|\mathcal{I}\| < 1,$ 
\[
\Exp_{\HH} f(\XX_t)
\underset{t \to \infty}{\longrightarrow} \frac{\widetilde{R} \mu(\{0\})}{1-\|\mathcal{I}\|}.
\]
\end{proposition}
\begin{proof}
The forcing functions $h_k$ satisfy 
\[
R h_1 + \widetilde{R} h_0 \underset{t \to \infty}{\longrightarrow} \widetilde{R} \mu(\{0\}).
\]
Moreover, for any $\epsilon > 0$ it can be decomposed into two pieces, 
\[
R h_1 + \widetilde{R} h_0
=F_0 + F_\epsilon,
\]
the first of which is regularly varying and the latter of which is bounded by $\epsilon$ and tends to $0$ as $t \to \infty$.  This comes by decomposing the eigenvalues into those separated from the critical point $\{4\gamma_1/\gamma_2^2\}$ and those in a neighborhood of it.  Now it follows that solving the Volterra equation with $F_0,$
\[
X_0(t) = \varphi(t)F_0(t) + \int_0^t \mathcal{I}(t-s)X_0(s) \dif s,
\]
which from Lemma \ref{q:limit}
\[
X_0(t) \underset{t \to \infty}{\sim} \frac{\varphi(t)F_0(t)}{1-\|\mathcal{I}\|}.
\]
For the second piece, we have that for 
\[
X_1(t) = \varphi(t)F_1(t) + \int_0^t \mathcal{I}(t-s)X_1(s) \dif s,
\]
we conclude
\[
X_1(t) \leq \frac{\epsilon}{1-\|\mathcal{I}\|}
\quad\text{and}\quad
X_1/\varphi(t) \underset{t \to \infty}{\longrightarrow} 0.
\]
Combining everything we conclude that by taking $\epsilon \to 0$ 
\[
\Psi(t)/\varphi(t) 
\underset{t \to \infty}{\longrightarrow} 
\frac{\widetilde{R} \mu(\{0\})}{1-\|\mathcal{I}\|}.
\]

By taking differences, we turn to bounding 
\[
\tfrac{\Psi(t)}{\varphi(t)} - \Exp_{\HH} f(\XX_t)
=
\int_0^t \mathcal{K}_s(t)\bigl( \tfrac{\Psi(s)}{\varphi(s)} -  \Exp_{\HH} f(\XX_s)\bigr) \dif s
+
\int_0^t 
\bigl(\mathcal{I}(t-s)\tfrac{\varphi(s)}{\varphi(t)}-\mathcal{K}_s(t)\bigr) \tfrac{\Psi(s)}{\varphi(s)}\dif s.
\]
Using that there is an $\epsilon>0$ and a $C>0$ so that
\begin{equation}\label{eqG:Its}
|\mathcal{I}(t-s)\tfrac{\varphi(s)}{\varphi(t)}-\mathcal{K}_s(t)| \leq C(1+s)^{-1} \mathcal{I}(t-s) + \epsilon e^{-(4\gamma_1/\gamma_2 - C\epsilon)(t-s)},
\end{equation}
we conclude that this error term tends to $0$.  The resolvent $R_s(t)$ of $\mathcal{K}_s(t)$ is bounded in $\text{L}^1$ using standard theory (see \cite[Theorem 3]{gripenberg1980volterra}) and by comparison to $\mathcal{I}$.  Then
\[
\tfrac{\Psi(t)}{\varphi(t)} - \Exp_{\HH} f(\XX_t)
=
\int_0^t R_x(t)
\int_0^x 
\bigl(\mathcal{I}(x-s)\tfrac{\varphi(s)}{\varphi(t)}-\mathcal{K}_s(t)\bigr) \tfrac{\Psi(s)}{\varphi(s)}\dif s\dif x.
\]
Then applying Fubini
\[
\tfrac{\Psi(t)}{\varphi(t)} - \Exp_{\HH} f(\XX_t)
=
\int_0^t 
\bigl(\mathcal{I}(x-s)\tfrac{\varphi(s)}{\varphi(t)}-\mathcal{K}_s(t)\bigr) \tfrac{\Psi(s)}{\varphi(s)}
\biggl(\int_s^t
R_x(t)
\dif x\biggr)
\dif s.
\]
Bounding the integral of $R_x(t)$ and using the bound \eqref{eqG:Its}, it follows we have that
\[
\bigl(\tfrac{\Psi(t)}{\varphi(t)}\bigr)^{-1}
|\tfrac{\Psi(t)}{\varphi(t)} - \Exp_{\HH} f(\XX_t)|
\underset{t \to \infty}{\longrightarrow} 0.
\]
\end{proof}

\paragraph{Average case analysis in the strongly convex case.}
We suppose now that we have taken the limit of empirical spectral measures, and consider a measure $\mu$ with support $\{0\} \cup [\lambda^-,\lambda^+]$ for some $\lambda^- > 0$.
We suppose that $\lambda^-$ is not at the critical point, i.e.\ $4\gamma_1 - \gamma_2^2 \lambda^- \neq 0.$ We further suppose that $\mu$ has a density with regular boundary behavior at $\lambda^-$:
\begin{equation}\label{eqE:leftedge}
\mu([\lambda^-,\lambda^-+\epsilon]) \underset{\epsilon \to 0}{\sim} \ell \epsilon^{\alpha}.
\end{equation}
We need to derive the asymptotic behavior of $h_0,h_1,\mathcal{K}$.  It is convenient if we remove the effect of any point mass of $\mu$ at $0,$ which effects the eventual convergence of the algorithm.  
Set $\widetilde{h}_0 = h_0 - \mu(\{0\})$ and $\widetilde{h}_1 = h_1$. 
This leads to precise asymptotics of the forcing function $\widetilde{h}_k$, as from the asymptotics of $j_0,j_1,j_2$ we have (recalling $\omega(\sqrt{\lambda^-}) = 4\gamma_1 - \gamma_2^2\lambda^{-}$) we have for $k \in \{0,1\}$ 
\begin{equation}\label{eqE:f0}
\widetilde{h}_k(t) \underset{t\to \infty}{\sim}
\ell_k e^{-\gamma_2 \lambda^{-} t}(1+t)^{-\theta}
\begin{cases}
1+c_1\cos\bigl( \sqrt{\lambda^{-}\omega} t -\log(1+t)\tfrac{\theta \gamma_2\sqrt{\lambda^-} }{\sqrt{\omega}} + c_2 \bigr), &\text{if } \omega > 0, \\
e^{\lambda^{-}\sqrt{-\omega}t}(1+t)^{-\tfrac{\theta \gamma_2\sqrt{\lambda^-}}{\sqrt{-\omega}}}, &\text{if } \omega < 0. \\
\end{cases}
\end{equation}

\paragraph{Malthusian exponent.}
We define the Malthusian exponent $\lambda^*,$ if it exists, as the solution of 
\[
\int_0^\infty e^{\lambda^* t}\mathcal{I}(t)\dif t
=1.
\]
We observe that using \eqref{eqE:phase} we can represent for any $\lambda < \sigma^2\gamma_2$
\begin{multline}\label{eqE:malthus}
\int_0^\infty
\sigma^2
e^{\lambda t -\sigma^2\gamma_2t}
\biggl(
\frac{1 - \cos\bigl(\vartheta(\sigma) + \sigma\sqrt{\omega}t\bigr)}{\omega}
\biggr)
\dif t\\
=
\biggl(
\frac{\gamma_2 \sigma^2 \omega(\omega-2\gamma_1)-(\sigma^2\gamma_2-\lambda)( (\omega-2\gamma_1)^2-2\gamma_1^2)}
{
2\gamma_1^2( (\sigma^2 \gamma_2-\lambda)^2 + \sigma^2 \omega)
}
+
\frac{1}{\sigma^2\gamma_2-\lambda}
\biggr)\frac{\sigma^2}{\omega}.
\end{multline}
On specializing to $\lambda=0,$ we can further simplify this to
\[
\int_0^\infty
\sigma^2
e^{-\sigma^2\gamma_2t}
\biggl(
\frac{1 - \cos\bigl(\vartheta(\sigma) + \sigma\sqrt{\omega}t\bigr)}{\omega}
\biggr)
\dif t
=\frac{\gamma_1 + \sigma^2\gamma_2^2}{4\gamma_1^2\gamma_2}.
\]

Returning to \eqref{eqE:malthus} and algebraically simplifying the expression, we can can write $\lambda^*$ as the solution
\[
1=\int_0^\infty e^{\lambda^* t}\mathcal{I}(t)\dif t
=\int_0^\infty
\sigma^4
\biggl(
\frac{
\gamma_2^2(\sigma^2 \gamma_2-\lambda^*)^2 
+\gamma_2(\omega-2\gamma_1)(\sigma^2 \gamma_2-\lambda^*)
+2\gamma_1^2
}
{
( (\sigma^2 \gamma_2-\lambda^*)^2 + \sigma^2 \omega)
(\sigma^2 \gamma_2-\lambda^*)
}
\biggr)\mu(\dif \sigma^2).
\]
We let $\mathcal{F}(\lambda^*)$ be the expression on the right hand side.
We note that expression is necessarily increasing in $\lambda^*$ (which is clear from the expression $\mathcal{F}(\lambda)=\int_0^\infty e^{\lambda t}\mathcal{I}(t)\dif t)$.  Furthermore, for $\lambda > \lambda_{-},$ $\mathcal{F}(\lambda)=\infty$ as $\mathcal{I}(t)$ decays no slower than $e^{-\lambda_{-}t}$.  Provided $\alpha > 1$ (recall \eqref{eqE:leftedge}), then $\mathcal{F}(\lambda_-)<\infty,$ and thus the existence of the Malthusian exponent $\lambda^*$ is equivalent to $\mathcal{F}(\lambda_-) \geq 1$.

\begin{proposition}\label{propE:Istrconvx}
Suppose that \eqref{eqE:leftedge} holds for some $\lambda_- > 0$ with $\omega(\lambda_-) \neq 0$ and for $\alpha > 1$.  Then if $\mathcal{F}(\lambda_-) < 1$ the solution of \eqref{eqE:Psi} satisfies
\[
\frac{\Psi(t)}{\varphi(t)}  - \frac{ \widetilde{R}\mu(\{0\})}{1-\|\mathcal{I}\|} \sim c e^{-\gamma_2\lambda_- t} t^{-\alpha-\theta}
\]
for some $c>0$
or if $\mathcal{F}(\lambda_-) > 1$ then with $\lambda^*$ the unique solution of $\mathcal{F}(\lambda^*) = 1$ for some constant $c>0$
\[
\frac{\Psi(t)}{\varphi(t)} - \frac{ \widetilde{R}\mu(\{0\})}{1-\|\mathcal{I}\|} \sim c e^{- \lambda_* t}t^{-\theta}.
\]
\end{proposition}
\begin{proof}
This follows standard renewal theory machinery.  See \cite{Asmussen} or \cite[Theorem 29]{paquetteSGD2021}.
\end{proof}

\begin{lemma}\label{lemE:kernelreplacement}
Then for $\lambda_- > 0$ for which  $\omega(\lambda_-) \neq 0$ and if $\mu({0}) = 0,$
\[
\Psi(t)^{-1}
|\Psi(t)-\varphi(t)\psi(t)| \underset{t \to \infty}{\longrightarrow} 0.
\]
\end{lemma}
\begin{proof}
We start from the raw Volterra equation for $\psi$ which is given by
\[
\psi(t) = R h_1(t) + \widetilde{R} h_0(t) + \int_0^t \mathcal{K}_s(t) \psi(s)\,\dif s.
\]
Multiplying through by $\varphi(t),$ we therefore have
\[
\varphi(t)\psi(t) = R \varphi(t) h_1(t) + \widetilde{R} \varphi(t)h_0(t) + \int_0^t \mathcal{K}_s(t) \tfrac{\varphi(t)}{\varphi(s)} \varphi(s)\psi(s)\,\dif s.
\]
This allows us to express the difference $\Psi(t)-\varphi(t)\psi(t)$
as 
\begin{equation}\label{eqE:Vchange}
\Psi(t)-\varphi(t)\psi(t) 
=\int_0^t \mathcal{K}_s(t) \tfrac{\varphi(t)}{\varphi(s)} \bigl(\Psi(s)-\varphi(s)\psi(s)\bigr)\dif s
+\int_0^t \bigl(\mathcal{I}_s(t) - \mathcal{K}_s(t) \tfrac{\varphi(t)}{\varphi(s)}\bigr) \Psi(s)\dif s.
\end{equation}
We can dominate the kernel above and below by
\[
\mathcal{K}_s(t) \tfrac{\varphi(t)}{\varphi(s)}
=
(1+ \mathcal{O}(s^{-1}))\mathcal{I}_s(t),
\]
with the error uniform in $t \geq s$; this uses Lemma \ref{lemE:singularpoint} and the asymptotic representations of the fundamental solutions Lemma \ref{lemE:fine} (see also \eqref{eqE:KS}).
Let $\lambda$ be the Malthusian exponent, if it exists, or $\lambda_-$ otherwise.
The latter forcing term of \eqref{eqE:Vchange} 
can be bounded by
\[
\biggl|
\int_0^t e^{-\lambda s}\bigl(\mathcal{I}_s(t) - \mathcal{K}_s(t) \tfrac{\varphi(t)}{\varphi(s)}\bigr) e^{\lambda s}\Psi(s)\dif s
\biggr|
\leq
\int_0^t C(1+s)^{-1}e^{-\lambda s}\mathcal{I}(t-s) e^{\lambda s}\Psi(s)\dif s.
\]

In the case that $\lambda$ is the Malthusian exponent, we have that $e^{\lambda s}\Psi(s)$ is bounded
and $e^{-\lambda s}\mathcal{I}(s)$ has $\text{L}^1$--norm $1$.
It follows that
\[
\int_0^t (1+s)^{-1}e^{-\lambda s}\mathcal{I}(t-s) e^{\lambda s}\Psi(s)\dif s
\lesssim
e^{-\lambda t}
\int_0^t (1+(t-s))^{-1}e^{-\lambda s}\mathcal{I}(s) \dif s.
\]
Thus by dominated convergence, we have that the forcing term satisfies
\[
F(t) \coloneqq
e^{\lambda t}
\biggl|
\int_0^t e^{-\lambda s}\bigl(\mathcal{I}_s(t) - \mathcal{K}_s(t) \tfrac{\varphi(t)}{\varphi(s)}\bigr) e^{\lambda s}\Psi(s)\dif s
\biggr|
\to 0.
\]
From Gronwall's inequality \cite[9.8.2]{gripenberg1990volterra} we conclude there is a non-negative resolvent kernel $r(t,s)$ so that
\begin{equation}\label{eqE:Vchange2}
e^{\lambda t}|\Psi(t)-\varphi(t)\psi(t))|
\leq
\int_0^t
r(t,t-s)|F(t-s)|\dif s.
\end{equation}
We deduce that the kernel has \emph{bounded uniformly continuous type} (see \cite[Theorem 9.5.4]{gripenberg1990volterra}; see also \cite[Theorem 9.9.1]{gripenberg1990volterra}) and therefore satisfies
\[
\lim_{h \to \infty}\sup_{t \geq 0} \int_0^{t-h} r(t,u)\dif u = 0.
\]
From here it follows from \eqref{eqE:Vchange2} that \[
e^{\lambda t}|\Psi(t)-\varphi(t)\psi(t))|
\to 0
\]
as $t\to \infty$, and hence from the asymptotics of $\Psi,$ the same holds when dividing by $\Psi$.

For the case where $\lambda$ is not the Malthusian exponent, we must conclude a slightly stronger bound.  This follows from first showing that the forcing function and the kernel $\mathcal{I}(t)$ both decay like $e^{-\lambda t} t^{-\alpha}$.  By conjugating the problem by $(1+t)^{\alpha},$ we reduce the problem to the same strategy as used above.
\end{proof}

\paragraph{Average case analysis in the non--strongly convex case.}
We turn to the assumption that $\mu$ is contained in $[0,\lambda_+],$ with a possible atom at $0$ and a density that is bounded away from its endpoints and that moreover $\mu$ has regular boundary behavior at $0$ with
\begin{equation}\label{eqE:leftedge2}
\mu((0,\epsilon]) \underset{\epsilon \to 0}{\sim} \ell \epsilon^{\alpha}.
\end{equation}
In the case of Marchenko--Pastur, this $\alpha = 1/2$.
We again need the behavior of $h_k$ for $k\in\{0,1\}$.  From the boundary condition at $0$, we have that for $k \in \{0,1\}$
\[
\widetilde{h}_k(t) \sim \frac{\ell \alpha}{2\varphi^2(t)}
\int_0^\infty (\sigma^2)^{\alpha-1}\sigma^{2k}
\biggl(
{\mathcal{D}}^{(\sigma^2)}_0(t)
+
2\theta{\mathcal{N}}^{(\sigma^2)}_0(t)
+4\theta^2
{\mathcal{H}}^{(\sigma^2)}_0(t)
\biggr)
\dif(\sigma^2).
\]
Now as we are in a neighborhood of $\sigma \approx 0$ we use the scaled solutions \eqref{eqE:DNH}, due to which we can express the solution as
\[
\widetilde{h}_k(t) \sim \frac{\ell \alpha}{2\varphi^2(t)}
\int_0^\infty (\sigma^2)^{\alpha-1}\sigma^{2k}
\biggl(
\widetilde{\mathcal{D}}^{(\sigma^2)}_0(t\sigma )
+
2\sigma^{-1}\theta\widetilde{\mathcal{N}}^{(\sigma^2)}_0(t\sigma)
+4\sigma^{-2}\theta^2
\widetilde{\mathcal{H}}^{(\sigma^2)}_0(t\sigma)
\biggr)
\dif(\sigma^2).
\]
We pick a $\varepsilon > 0$ and decompose the integral according to $t\sigma > \varepsilon$ and those below.  For those $\sigma$ above, we use Lemma \ref{lemE:scaling} and conclude
\[
\begin{aligned}
\widetilde{h}_k(t)
&\sim
\frac{\ell \alpha}{2\varphi(t)}
\int_{\varepsilon^{-1}t^{-1}}^\infty (\sigma^2)^{\alpha+k-1}
\mathcal{O}\bigl(e^{-\sigma^2\gamma_2 t}\bigr)
\dif(\sigma^2) \\
&+
\frac{\ell \alpha}{2\varphi^2(t)}
\int_{\varepsilon^2t^{-2}}^{\varepsilon^{-1}t^{-1}} (\sigma^2)^{\alpha+k-1-\theta}
\mathfrak{d}(t\sigma )
\biggl(
1
+2\theta
+4\theta^2
\biggr)
\dif(\sigma^2) \\
&+
\frac{\ell \alpha}{2\varphi^2(t)}
\int_0^{\varepsilon^2t^{-2}} (\sigma^2)^{\alpha+k-1}
\biggl(
1+2t\theta+2t^2\theta^2+o_\epsilon(1)
\biggr)
\dif(\sigma^2) \\
\end{aligned}
\]
Both first and last integrals will be negligible.  For the middle integral, we change variables with $x^2 = \sigma^2 t^2$ to get
\begin{equation*}
\widetilde{h}_k(t)
\sim
\frac{\ell \alpha}{2t^{2\alpha+2k}}
\int_{\varepsilon^2}^{\varepsilon^{-1}t} (x^2)^{\alpha+k-1-\theta}
\mathfrak{d}(x)
\biggl(
1
+2\theta
+4\theta^2
\biggr)
\dif(x^2).
\end{equation*}
The integral is convergent when $2\alpha+2k -\theta< 0$ as $\mathfrak{d}$ grows like $x^{\theta}$ as $x\to \infty$ and $\mathfrak{d}$ tends to $0$ like $x^{2\theta}$ as $x\to 0$.  Thus we may take $\varepsilon \to 0$ and conclude
\begin{equation}\label{eqE:nscf0}
\widetilde{h}_k(t)
\sim
\frac{\ell \alpha}{2t^{2\alpha+2k}}
\biggl(
1
+2\theta
+4\theta^2
\biggr)
\int_{0}^{\infty} (x^2)^{\alpha+k-1-\theta}
\mathfrak{d}(x)
\dif(x^2).
\end{equation}

We again use the approximate convolution structure of the kernel, in particular the kernel $\mathcal{I}$ and the approximate Volterra equation \eqref{eqE:Psi}.
\begin{proposition}\label{prop:sqasymp}
When $\|\mathcal{I}\| < 1$ and $\widetilde{R}=0$ and $\theta > 2\alpha + 2$ it follows 
\[
\Psi(t)/\varphi(t) 
\underset{t \to \infty}{\sim} 
\frac{R\ell \alpha}{2t^{2\alpha+2}}
\frac{
\bigl(
1
+2\theta
+4\theta^2
\bigr)}
{
1-\|\mathcal{I}\|
}
\int_{0}^{\infty} (x^2)^{\alpha-\theta}
\mathfrak{d}(x)
\dif(x^2).
\]
In the case that $\widetilde{R} > 0,$
\[
\Psi(t)/\varphi(t) 
\underset{t \to \infty}{\sim} 
\frac{\widetilde{R}\ell \alpha}{2t^{2\alpha}}
\frac{
\bigl(
1
+2\theta
+4\theta^2
\bigr)}
{
1-\|\mathcal{I}\|
}
\int_{0}^{\infty} (x^2)^{\alpha-1-\theta}
\mathfrak{d}(x)
\dif(x^2).
\]
\end{proposition}
\begin{proof}
  The proposition is a corollary of Lemma \ref{q:rvrate}, using $\mathcal{I}(t) \sim c(\mu,\theta)t^{-2\alpha-4}$ and \eqref{eqE:nscf0}.
\end{proof}
Finally, we can derive the needed bound for the original problem:
\begin{lemma}\label{lemE:simplify}
When $\|\mathcal{I}\| < 1$ and \eqref{eqE:leftedge2} holds
\[
\psi(t)
\underset{t \to \infty}{\sim} \Psi(t)/\varphi(t).
\]
\end{lemma}
\begin{proof}
This follows the same strategy as Lemma \ref{lemE:kernelreplacement}.
\end{proof}

\subsection{Rate bounds}

We let $\lambda^{-}$ be the left endpoint of the support of $\mu$ restricted to $(0,\infty)$.  If $\lambda^{-} = 0,$ we are in the non--strongly convex case above, and the rate of convergence is polynomial for any choice of step size that is convergent.  We show a step size choice that gives a good rate for all $\lambda^{-} > 0$ separately.

We conclude with bounds for the convolution kernel $\mathcal{I}$ which establish bounds for the rate under step size conditions which are strictly better than \eqref{eqE:convergence}.
We shall work under that $\gamma_1$ and $\gamma_2$ satisfy the condition
\begin{equation}\label{eqE:strongrate}
\int_{0+}^\infty 
\tfrac{\gamma_1 \Delta + \sigma^2\gamma_2^2}{2\gamma_2}
\mu(\dif \sigma^2)
\leq 1
\end{equation}
for some $\Delta > 1,$ which ensures that the algorithm converges.

We recall that the Malthusian exponent is defined as the solution $\lambda^*$ of
\[
1=\mathcal{F}(\lambda^*)
\coloneqq\int_0^\infty
\sigma^4
\biggl(
\frac{
\gamma_2^2(\sigma^2 \gamma_2-\lambda^*)^2 
+\gamma_2(\omega-2\gamma_1)(\sigma^2 \gamma_2-\lambda^*)
+2\gamma_1^2
}
{
( (\sigma^2 \gamma_2-\lambda^*)^2 + \sigma^2 \omega)
(\sigma^2 \gamma_2-\lambda^*)
}
\biggr)\mu(\dif \sigma^2),
\]
if it exists.

We shall produce a bound for $\mathcal{F}(\lambda)$ for $\lambda$ sufficiently small, namely:
\begin{lemma}\label{lemE:quadupper}
Suppose that $\Delta > 1$ and $\lambda^{-} > 0$ and that
$\lambda_*$ is defined by
\[
\lambda_* 
=
\frac{
\lambda^{-}\gamma_2 + 2\tfrac{\gamma_1}{\gamma_2}
-
\sqrt{
(\lambda^{-}\gamma_2 + 2\tfrac{\gamma_1}{\gamma_2})^2
-8\gamma_1 \lambda^{-}(1-\tfrac{1}{\Delta})
}
}{2},
\]
then for all $\sigma^2 \geq \lambda^{-},$
\[
\operatorname{II} \coloneqq
\sigma^4
\biggl(
\frac{
\gamma_2^2(\sigma^2 \gamma_2-\lambda)^2 
+\gamma_2(\omega-2\gamma_1)(\sigma^2 \gamma_2-\lambda)
+2\gamma_1^2
}
{
( (\sigma^2 \gamma_2-\lambda)^2 + \sigma^2 \omega)
(\sigma^2 \gamma_2-\lambda)
}
\biggr)
\leq 
\frac{\gamma_1 \Delta + \sigma^2\gamma_2^2}{2\gamma_2}.
\]
Moreover, we have the bounds
\[
\lambda_*
\geq
\frac{ 2\gamma_1 \lambda^{-}(1-\tfrac{1}{\Delta})}
{
\gamma_2
\lambda^{-} + 2\tfrac{\gamma_1}{\gamma_2}
}.
\]
\end{lemma}

This leads immediately to a rate bound:
\begin{corollary}\label{corE:stepsize}
  At the default parameters of SDANA $\gamma_2 = (\mathfrak{tr}(\mu))^{-1}$ where $\mathfrak{tr}(\mu) = \int_0^\infty \sigma^2 \mu( \dif \sigma^2)$ and $\gamma_1=\tfrac{\gamma_2}{4}$ 
  the convergence rate is at least $\tfrac{3}{8}\min\{ (\mathfrak{tr}(\mu))^{-1} \lambda^{-},\tfrac{1}{2} \}$.  The fastest possible rate, in contrast, is no larger than $\min\{ (\mathfrak{tr}(\mu))^{-1} \lambda^{-},\tfrac{1}{2} \}.$
\end{corollary}
\begin{proof}
  For the rates, we apply Lemma \ref{q:errate}.  Lemma \ref{lemE:quadupper} gives a lower bound on the Malthusian exponent, where we take $\Delta =4.$  As for the rate of the forcing function, we have that its rate is bounded by
\[
\begin{cases}
\gamma_2 \lambda^{-}, & \text{if } \omega = 4\gamma_1 - \gamma_2^2 \lambda^{-}> 0 \text{ or }\\
\gamma_2 \lambda^{-}-\sqrt{\gamma_2^2 (\lambda^{-})^2-4\gamma_1\lambda^{-}},
& \text{otherwise.}
\end{cases}
\]
This is always bounded above by $\min\{\gamma_2 \lambda^{-},2\tfrac{\gamma_1}{\gamma_2}\}$.  For convergence we should have 
\[
  \|\mathcal{I}\|
  =
  \int_{0+}^\infty 
  \tfrac{\gamma_1 + \sigma^2\gamma_2^2}{2\gamma_2}
  \mu(\dif \sigma^2)
  \leq 1,
\]
and thus optimizing in taking $\gamma_2\lambda^{-} = 2\gamma_1/\gamma_2$ and the above norm equal to $1,$ we conclude the fastest rate is at most $\frac{4\lambda^{-}}{\lambda^{-} + 2 \mathfrak{tr}(\mu)}$.  This is in turn at most the claimed amount.
\end{proof}

\begin{proof}[Proof of Lemma]
We begin with some simplifications.  The claimed bound is equivalent to 
\[
2\gamma_2 \sigma^4
\bigl(
2\gamma_1\gamma_2 - \lambda \gamma_2^2
+\tfrac{2\gamma_1^2}{\sigma^2\gamma_2 -\lambda}
\bigr)
\leq
\bigl(
\gamma_1 \Delta + \sigma^2 \gamma_2^2
\bigr)
\bigl(
\lambda^2 - 2\lambda \sigma^2 \gamma_2 + 4\gamma_1 \sigma^2
\bigr).
\]
After cancelling terms and rearranging
\[
\tfrac{4\gamma_1^2\gamma_2 \sigma^4}{\sigma^2\gamma_2 -\lambda}
\leq
\gamma_1 \Delta
\bigl(
- 2\lambda \sigma^2 \gamma_2 + 4\gamma_1 \sigma^2
\bigr)
+
\bigl(
\gamma_1 \Delta + \sigma^2 \gamma_2^2
\bigr)\lambda^2.
\]
Hence dropping the $\lambda^2$ term and simplifying, it suffices that 
\[
\operatorname{III}
\coloneqq
\frac{2\gamma_1 \sigma^2}
{
\bigl(
2\tfrac{\gamma_1}{\gamma_2}- \lambda
\bigr)
\bigl(
{\sigma^2\gamma_2 -\lambda}
\bigr)
}
\leq
\Delta
.
\]
The map $x\mapsto \tfrac{x}{x\gamma_2 -\lambda}$ is decreasing for $x\gamma_2 > \lambda$, and hence it suffices that
\[
\operatorname{IV}
\coloneqq
\frac{2\gamma_1 \lambda^{-}}
{
\bigl(
2\tfrac{\gamma_1}{\gamma_2}- \lambda
\bigr)
\bigl(
{\lambda^{-}\gamma_2 -\lambda}
\bigr)
}
\leq
\Delta
.
\]
It follows that for all $\lambda$ less than the smallest root of
\[
\lambda^2 - \lambda(\lambda^{-}\gamma_2 + 2\tfrac{\gamma_1}{\gamma_2})+2\gamma_1 \lambda^{-}(1-\tfrac{1}{\Delta})
=0,
\]
\(\operatorname{IV} \leq \Delta.\)
Solving for the smaller root $\lambda_*$, we have
\begin{equation}\label{eqE:g1g2}
\lambda \leq 
\lambda_*
=
\frac{
\lambda^{-}\gamma_2 + 2\tfrac{\gamma_1}{\gamma_2}
-
\sqrt{
(\lambda^{-}\gamma_2 + 2\tfrac{\gamma_1}{\gamma_2})^2
-8\gamma_1 \lambda^{-}(1-\tfrac{1}{\Delta})
}
}{2}.
\end{equation}
Using concavity of the square root, we can bound $\sqrt{a+x} \leq \sqrt{a} + \tfrac{x}{2\sqrt{a}}$ and so conclude
\[
\lambda_*
\geq
\frac{ 2\gamma_1 \lambda^{-}(1-\tfrac{1}{\Delta})}
{
\gamma_2
\lambda^{-} + 2\tfrac{\gamma_1}{\gamma_2}
}.
\]
\end{proof}

\section{The general SDAHB kernel}\label{sec:dahb}
In this section, we analyze in detail a general version of SDAHB where we also include a $\gamma_2$ that is, we consider an algorithm SDA where $\gamma_1, \gamma_2 > 0$ and $\Delta(k,n) = \tfrac{\theta}{n}$. In this general setting, the log-derivative $\Phi(t) = \theta$. We recall the ODE \eqref{eq:JEQ_SHB} that describes this process (where $\sigma_j^2 = \lambda$) 
\begin{equation}
    \begin{aligned} \label{eq:JEQ_SHB_1}
    &\widehat{J}^{(3)} + \left (  -3\theta + 3 \gamma_2 \lambda \right ) \widehat{J}^{(2)} + \left ( 2 \theta^2 -4 \gamma_2 \lambda \theta +4\gamma_1\lambda + 2 \gamma_2^2 \lambda^2 \right ) \widehat{J}^{(1)} \\
    &+ \left ( - 4 \gamma_1 \lambda \theta + 4 \gamma_1 \gamma_2 \lambda^2 \right ) \widehat{J}\\
    &=
    \tfrac{\gamma_2^2}{\gamma_1}
    \widehat{\psi}^{(2)}
    +\bigl(
    2\gamma_2
    + 
    \bigl(-\theta + \gamma_2\lambda\bigr)\tfrac{\gamma_2^2}{\gamma_1}
    \bigr)
    \widehat{\psi}^{(1)}
    +\bigl(
    2\gamma_1
    +  2\gamma_1\lambda \tfrac{\gamma_2^2}{\gamma_1}
    \bigr)
    \widehat{\psi}.
    \end{aligned}
\end{equation}
The initial conditions are given by
\begin{equation}
\begin{gathered}
\widehat{J}(0) = \gamma_1^{-1}\EE\biggl[ \bigl(\nu_{0,j}-\tfrac{(\UU^T \eeta)_j}{\sigma_j}\bigr)^2\biggr],
\quad
 \widehat{J}^{(1)}(0) =
 \tfrac{\gamma_2^2}{\gamma_1} \widehat{\psi}(0)
-\widehat{J}(0) \bigl( -2\theta + 2\gamma_2 \lambda \bigr),
\quad\text{and} \\
\widehat{J}^{(2)}(0)
 =  \tfrac{\gamma_2^2}{\gamma_1} \widehat{\psi}^{(1)}(0)+ 2\gamma_2 \widehat{\psi}(0) - 2 \gamma_1 \lambda \widehat{J}(0) + ( 2\theta - 2 \gamma_2 \lambda) \widehat{J}^{(1)}(0).
\end{gathered}
\end{equation}
We note that the ODE in \eqref{eq:JEQ_SHB_1} is constant coefficient and therefore can be solved by finding the characteristic polynomial, that is,
\begin{align*}
    0 &= \xi^3 + (3 \gamma_2 \lambda - 3 \theta) \xi^2 + (2\theta^2 -4 \gamma_2 \lambda \theta + 4 \gamma_1 \lambda + 2 \gamma_2^2 \lambda^2) \xi + 4\gamma_1 \gamma_2 \lambda^2 - 4\gamma_1 \lambda \theta \\
    0 &= (\xi + \lambda \gamma_2 - \theta) (\xi^2 + (2\lambda \gamma_2 - 2 \theta) \xi + 4 \lambda \gamma_1 ) \\
    \xi &= \theta - \lambda \gamma_2 \quad \text{and} \quad \xi = -(\lambda \gamma_2 - \theta) \pm \sqrt{(\lambda \gamma_2-\theta)^2-4 \lambda \gamma_1}.
\end{align*}
It immediately follows that the solutions to \eqref{eq:JEQ_SHB_1} are linear combinations of $\text{exp}(-(\lambda \gamma_2-\theta)t)$ and $\text{exp}(-(\lambda \gamma_2-\theta) \pm \sqrt{(\lambda \gamma_2-\theta)^2-4\lambda_j \gamma_1})$. We now write the Dirichlet, Neumann, and 2nd-derivative solutions for which we will use to derive the kernel and the forcing term. For convenience, we denote $ \omega \defas 4 \lambda \gamma_1- (\lambda \gamma_2 - \theta)^2$ and $\rho \defas \lambda \gamma_2 - \theta$. Taking derivatives, we get the following expressions for $K_s(t)$:
\begin{align*}
    K_s(t) 
    &= \exp(-t\rho) \big (c_1 + c_2 \exp(-t\sqrt{-\omega}) + c_3 \exp(t \sqrt{-\omega}) \big ) \\
    \tfrac{d}{dt} K_s(t) 
    &= -\rho \exp(-t\rho) \big (c_1 +c_2 \exp(-t\sqrt{-\omega}) +c_3 \exp(t \sqrt{-\omega}) \big )\\
    & \quad + \sqrt{-\omega} \exp(-t \rho) \big (c_3 \exp(t \sqrt{-\omega}) - c_2 \exp(-t\sqrt{-\omega}) \big )\\
    \tfrac{d}{dt^2} K_s(t) 
    &= \rho^2 \exp(-t \rho) \big (c_1 + c_2 \exp(-t\sqrt{-\omega}) + c_3\exp(t \sqrt{-\omega}) \big )\\
    &\quad + 2 \rho \sqrt{-\omega} \exp(-t\rho) \big (c_2 \exp(-t \sqrt{-\omega}) - c_3 \exp(t \sqrt{-\omega})  \big )\\
    & \quad -\omega \exp(-t \rho) \big ( c_2 \exp(-t \sqrt{-\omega}) + c_3 \exp(t \sqrt{-\omega}) \big ).
\end{align*}
Provided that $\omega \neq 0$, we can now solve for $c_1, c_2, c_3$ for the Dirichlet, Neumann, and 2nd-derivative solutions,
\begin{equation}\begin{aligned}
\text{(Dirichlet sol., ${\mathcal{D}}_s(t)$)} \quad &{L}[{\mathcal{D}}_s(t)] = 0 \quad \text{where} \quad  {\mathcal{D}}_s(s) = (1,0,0)^T,\\
\text{(Neumann sol., ${\mathcal{N}}_s(t)$)} \quad &{L}[{\mathcal{N}}_s(t)] = 0 \quad \text{where} \quad  {\mathcal{N}}_s(s) = (0,1,0)^T,\\
\text{(2nd derivative sol., ${\mathcal{H}}_s(t)$)} \quad &{L}[{\mathcal{H}}_s(t)] = 0 \quad \text{where} \quad  {\mathcal{H}}_s(s) = (0,0,1)^T.
\end{aligned}
\end{equation}
To distinguish these solutions, we denote the coefficients by $c_i^D, c_i^N, c_i^H$ for $i = 1,2,3$. We begin by find the coefficients for $\mathcal{D}_s(t)$: 
\begin{align*}
    (c_1^{D}, c_2^{D}, c_3^{D}) &= \bigg ( \exp(s\rho) (1+\tfrac{\rho^2}{\omega} ), \tfrac{1}{2}\exp(s(\rho+\sqrt{-\omega})) \big ( \tfrac{-\rho^2}{\omega} - \tfrac{\rho}{\sqrt{-\omega}} \big ), \tfrac{1}{2} \exp(s (\rho - \sqrt{-\omega})) \big (\tfrac{-\rho^2}{\omega} + \tfrac{\rho}{\sqrt{-\omega}} \big ) \bigg )\\
    (c_1^{N}, c_2^{N}, c_3^{N}) &= \left ( \exp(s\rho)\tfrac{2\rho}{\omega}, \tfrac{1}{2} \exp(s(\rho + \sqrt{-\omega})) \big ( \tfrac{-2\rho}{\omega} - \tfrac{1}{\sqrt{-\omega}} \big ), \tfrac{1}{2} \exp(s (\rho - \sqrt{-\omega}) ) \big (\tfrac{-2\rho}{\omega} + \tfrac{1}{\sqrt{-\omega}} \big )  \right )\\
    (c_1^{H}, c_2^{H}, c_3^H) &= \bigg (\exp(s \rho) \tfrac{1}{\omega}, -\tfrac{1}{2 \omega} \exp(s(\rho + \sqrt{-\omega})), -\tfrac{1}{2 \omega} \exp(s(\rho - \sqrt{-\omega})) \bigg ).
\end{align*}
We recall $J =\gamma_1e^{-2 \theta t} \widehat{J}$ and Corollary~\ref{cor:kernel_aleph} that
\begin{align*}
    J(t) = \gamma_1 e^{-2\theta t} \widehat{J}_0(t) + \gamma_1 e^{-2 \theta t} \int_0^t K_s(t) \widehat{\psi}(s) \, \dif s. 
\end{align*}
Using the coefficients in Corollary~\ref{cor:kernel_aleph}, we write an expression for the forcing term 
\begin{align} \label{eq:aleph_forcing}
    \widehat{J}_0(t) = \frac{1}{2} \left (1+\tfrac{\rho^2}{\omega} \right ) \frac{1}{\gamma_1} \mathbb{E}[(\nu_{0,j}- \tfrac{(\UU^T \eeta)_j}{\sigma_j} )^2] e^{-t\rho} (1 + \cos(t \sqrt{\omega} + \vartheta_1)),
\end{align}
where the phase shift satisfies
\begin{equation} \begin{aligned} \label{eq:vartheta_1_alpha}
    \cos(\vartheta_1) &= \frac{ \big(1+ \tfrac{\rho}{\sqrt{-\omega}} \big )^2 + \big (1- \frac{\rho}{\sqrt{-\omega}} \big )^2}{2 \big (1 + \tfrac{\rho^2}{\omega} \big )} = \frac{\omega-\rho^2}{\rho^2 + \omega}\\
    \sin(\vartheta_1) &= \frac{ \big(1- \tfrac{\rho}{\sqrt{-\omega}} \big )^2 - \big (1+ \frac{\rho}{\sqrt{-\omega}} \big )^2}{2 \big (1 + \tfrac{\rho^2}{\omega} \big )} i = \frac{2 \rho \sqrt{\omega}}{\rho^2 + \omega}.
\end{aligned} \end{equation}
We now give an expression for the kernel $K_s(t)$:
\begin{align} \label{eq:aleph_kernel}
   & K_s(t) = \frac{1}{2} \left ( \frac{\gamma_2^2}{\gamma_1} + \frac{\gamma_2^2 \rho^2}{ \omega \gamma_1} +\frac{4}{\omega} (\gamma_1-\gamma_2 \rho) \right ) e^{-(t-s) \rho} \left(1 + \cos((t-s)\sqrt{\omega}+\vartheta_2) \right ) 
\end{align}
where we have 
\begin{equation} \begin{aligned} \label{eq:vartheta_2_alpha}
\cos(\vartheta_2) &= \frac{2 \left (\frac{\gamma_2^2}{4 \gamma_1} - \frac{\gamma_2^2 \rho}{4 \omega \gamma_1} + \frac{\gamma_2 \rho}{\omega} - \frac{\gamma_1}{\omega} \right )}{\frac{1}{2} \left ( \frac{\gamma_2^2}{\gamma_1} + \frac{\gamma_2^2 \rho^2}{\gamma_1 \omega} +\frac{4}{\omega} (\gamma_1-\gamma_2 \rho) \right )} = \frac{ \gamma_2^2\omega - \gamma_2^2 \rho + 4 \gamma_2 \rho \gamma_1 - 4 \gamma_1^2}{ \gamma_2^2 \omega + \gamma_2^2 \rho^2 -4\gamma_2 \gamma_1 \rho +4\gamma_1^2}\\
    \sin(\vartheta_2) &= \frac{\frac{\gamma_2^2 \rho}{\gamma_1\sqrt{\omega}} - \frac{2 \gamma_2}{\sqrt{\omega}}}{\frac{1}{2} \left ( \frac{\gamma_2^2}{\gamma_1} + \frac{\gamma_2^2 \rho^2}{\gamma_1 \omega} +\frac{4}{\omega} (\gamma_1-\gamma_2 \rho) \right )} = \frac{2( \gamma_2^2 \rho \sqrt{\omega} - 2 \gamma_2 \gamma_1 \sqrt{\omega})}{\gamma_2^2 \omega + \gamma_2^2 \rho^2 -4\gamma_2 \gamma_1 \rho+4\gamma_1^2 }.
\end{aligned} \end{equation}
It follows that $J(t)$ is the sum of \eqref{eq:aleph_forcing} and \eqref{eq:aleph_kernel}. We now recall that $\widehat{\psi}(s) = \frac{2 \lambda \exp(2\theta s) \psi^{(n)}(s)}{n}$. Finally we arrive at the Volterra equation
\begin{align*}
\psi(t) = \frac{1}{2} \int_0^{\infty} \lambda \gamma_1 e^{-2 \theta t} \widehat{J}_0^{(\lambda)}(t) \, \dif \mu(\lambda) + \gamma_1  \int_0^t \int_0^\infty \sigma^4 e^{-2\theta (t-s)} K_s^{(\lambda)}(t) \dif \mu(\lambda) \, \psi(s) \, \dif s.
\end{align*}

\begin{proposition}[Volterra equation for general SDAHB with parameters$(\gamma_1, \gamma_2, e^{\theta t})$ ] The Volterra equation for the general SDAHB with step size parameters $\gamma_1, \gamma_2 > 0$ and $\varphi(t) = e^{\theta t}$ is 
\begin{equation} \begin{aligned}
    G^{(\lambda)}(t) &= \frac{1}{4} \left (1 + \frac{\rho^2}{\omega} \right ) e^{-t(\rho + 2 \theta)}(1+\cos(t \sqrt{\omega} + \vartheta_1)) \\
    K_s^{(\lambda)}(t) &=  \frac{\lambda^2}{2} \left ( \gamma_2^2 + \frac{\gamma_2^2 \rho^2}{ \omega } +\frac{4}{\omega} (\gamma_1^2-\gamma_2 \gamma_1 \rho) \right ) e^{-(t-s) (\rho + 2 \theta)} \left(1 + \cos((t-s)\sqrt{\omega}+\vartheta_2) \right ),
\end{aligned} \end{equation}
where $\omega = 4 \lambda \gamma_1 - (\lambda \gamma_2 - \theta)^2$, $\rho = \lambda \gamma_2-\theta$, and $\vartheta_1$ and $\vartheta_2$ are defined in \eqref{eq:vartheta_1_alpha} and \eqref{eq:vartheta_2_alpha} respectively. 
\end{proposition}

\begin{corollary}[Volterra equation for SDAHB] The Volterra equation for SDAHB with step size parameters $\gamma_1 > 0$, $\gamma_2 = 0$, and $\varphi(t) = e^{\theta t}$ is 
\begin{equation} \begin{aligned}
    G^{(\lambda)}(t) &= \frac{1}{4} \left (1 + \frac{\theta^2}{\omega} \right ) e^{-t\theta}(1-\cos(t \sqrt{\omega} + \vartheta_1)) \\
    \text{and} \quad K_s^{(\lambda)}(t) &=  \frac{2 \gamma_1^2 \lambda^2}{\omega}  e^{-(t-s) \theta} \left(1 - \cos((t-s)\sqrt{\omega}) \right ),
\end{aligned} \end{equation}
where $\omega = 4 \lambda \gamma_1 - \theta^2$ and $\vartheta_1$ is defined by 
\begin{equation}
    \cos(\vartheta_1) = \frac{\theta^2-\omega}{\theta^2+\omega} \qquad \text{and} \qquad \sin(\vartheta_1) = \frac{2 \theta \sqrt{\omega}}{\theta^2 + \omega}.
\end{equation}
\end{corollary}

\subsection{ Convergence analysis for SDAHB }

The interaction kernel for SDAHB is therefore of convolution type, and we have
\[
F(t) =
\int_0^\infty G^{(\lambda)}(t) \mu(\dif \lambda)
\quad
\text{and}
\quad
\mathcal{I}(t) =
\int_0^\infty K_0^{(\lambda)}(t) \mu(\dif \lambda).
\]
The loss of homogenized SGD then satisfies
\[
\Exp_{\HH} f(\XX_t)
= F(t) + \int_0^t \mathcal{I}(t-s)\Exp_{\HH} f(\XX_s) \dif s.
\]
Computing the Laplace transform of this kernel for all $x$ sufficiently small,
\[
\mathcal{F}(x)
\defas
\int_0^\infty 
e^{x t}
\mathcal{I}(t) 
\dif t
=
\int_0^\infty
\frac{2 \gamma_1^2 \lambda^2}
{
(\theta - x)(x^2-2\theta x + 4\gamma_1 \lambda)
}
\mu(\dif \lambda),
\]
and recall that the Malthusian exponent $\lambda^{*}$ is defined as the root of $\mathcal{F}(\lambda^{*})=1$, if it exists. 
In particular evaluating at $x=0,$ we compute the norm
\begin{equation}\label{eqG:norm}
\|\mathcal{I}\|
=
\int_0^\infty
\frac{2 \gamma_1^2 \lambda^2}
{
\theta (4\gamma_1 \lambda)
}
\mu(\dif \lambda)
=
\frac{\gamma_1}{2\theta}
\int_0^\infty
\lambda
\mu(\dif \lambda)
=
\frac{\gamma_1}{2\theta}
\mathfrak{tr}(\mu)
.
\end{equation}

\subsection{ Convergence analysis for SDAHB }

We now suppose we have passed to a limiting measure $\mu$ with a support $\{0\} \cup [\lambda^{-},\lambda^{+}]$ that satisfies
\begin{equation}\label{eqG:leftedge}
\mu([\lambda^-,\lambda^-+\epsilon]) \underset{\epsilon \to 0}{\sim} \ell \epsilon^{\alpha}.
\end{equation}
The forcing function satisfies, with $\omega=\omega(\lambda^{-}) = 4\lambda^{-}\gamma_1 - \theta^2$
and for some constants $c,c_1,c_2$ depending on the algorithm parameters,
\[
F(t) 
\underset{t \to \infty}{\sim}
\begin{cases}
c t^{\alpha-1}e^{-t(\theta - \sqrt{\theta^2 - 4\gamma_1 \lambda^{-}})},
&\text{ if }\quad  \omega < 0, \\
c t^{\alpha+1}e^{-t\theta},
&\text{ if }\quad  \omega = 0, \\
(c_1t^{\alpha-1} + c_2t^{\alpha-1}\cos(t \sqrt{\omega}))e^{-t \theta},
&\text{ if }\quad  \omega > 0.
\end{cases}
\]

From standard renewal theory (Lemma \ref{q:errate}), we have that 

\begin{proposition}\label{propG:Istrconvx}
If $\mathcal{F}(\lambda^-) < 1$ and \eqref{eqG:leftedge} holds
\[
\Exp_{\HH} f(\XX_t)  - \frac{ \widetilde{R}\mu(\{0\})}{1-\|\mathcal{I}\|} 
=F(t)e^{o(t)}
\]
or if $\mathcal{F}(\lambda_-) > 1$ then with $\lambda^*$ the unique solution of $\mathcal{F}(\lambda^*) = 1$ 
\[
\Exp_{\HH} f(\XX_t) - \frac{ \widetilde{R}\mu(\{0\})}{1-\|\mathcal{I}\|} = e^{-(\lambda_*+o(1))t}.
\]
\end{proposition}
We note that if we take the default parameters, we come within a factor of the maximum rate.
\begin{proposition}\label{propG:default}
Suppose we take the default parameters for SDAHB, that is
\[
\theta=2
\quad\text{and}\quad
\gamma_1 
= \frac{\theta}{\mathfrak{tr}(\mu)}
= \frac{\theta}{\int\lambda \mu(\dif \lambda)},
\]
the rate of convergence is at least
\[
\lambda_* \geq \frac{\gamma_1\lambda^{-}\theta}{2\gamma_1 \lambda^{-} +\theta^2}
=
\frac{2\gamma_1\lambda^{-}}{2\gamma_1 \lambda^{-} +4}.
\]
The fastest possible rate is at most $\frac{4 \lambda^{-}}{\mathfrak{tr}(\mu)}.$
\end{proposition}
\begin{proof}
We just need to bound $\mathcal{F}(x) \leq 1$
for $x\leq \min\{ \tfrac{2\lambda^{-}\gamma_1}{\theta}, \tfrac{\theta}{2}\}$ and with the parameter choices made.
By monotonicity
\[
\mathcal{F}(x)
\leq 
\int_0^\infty
\frac{2 \gamma_1^2 \lambda \lambda^{-}}
{
(\theta - x)(x^2-2\theta x + 4\gamma_1 \lambda^{-})
}
\mu(\dif \lambda)
\leq 
\frac{2 \gamma_1 \theta \lambda^{-}}
{
(\theta - x)(x^2-2\theta x + 4\gamma_1 \lambda^{-})
}.
\]
We bound further from above by dropping the $x^2$ and then solving the result quadratic, i.e. $\mathcal{F}(x) \leq 1$ if 
\[
x \leq \frac{4\gamma_1 \lambda^{-} +2\theta^2
-
\sqrt{
(4\gamma_1 \lambda^{-} +2\theta^2)^2-8\theta(2 \gamma_1 \lambda^{-}\theta)
}
}
{
4\theta
}.
\]
By concavity of the square root, it suffices to have
\[
x \leq 
\frac{\gamma_1\lambda^{-}\theta}{2\gamma_1 \lambda^{-} +\theta^2}.
\]

The rate of $F$ is at most $\min\{ \tfrac{2\gamma_1 \lambda^{-}}{\theta}, \theta\}$, and so optimizing this over $\tfrac{\gamma_1}{2\theta} \mathfrak{tr}(\mu) \leq 1$, we arrive at $\theta = \frac{4 \lambda^{-}}{\mathfrak{tr}(\mu)}$
\end{proof}

\subsection{Average-case rates non--strongly convex}
We instead suppose the support is given by $[0, \lambda^+]$ and that
\[
\mu( (0,\varepsilon) ) 
\underset{\epsilon \to 0}{\sim}
\ell \epsilon^{\alpha}.
\]
The forcing function, for any $\theta > 0$ then behaves like 
\begin{equation}\label{eqG:sdahb}
F(t) 
\underset{t \to \infty}{\sim}
c(Rt^{-\alpha-1} + \widetilde{R} t^{-\alpha}).
\end{equation}
It follows using Lemma \ref{q:rvrate} that when $\|\mathcal{I}\| < 1,$ the same rate holds for $\Exp f(\XX_t)$ up to multiplication by $(1-\|\mathcal{I}\|)^{-1}$.

\subsection{Degeneration to SGD}
\label{sec:SDAHBsGD}

\begin{theorem}\label{thmG:degeneration}
Suppose the homogenized SGD diffusions for SHB and SGD are chosen so that $\gamma^{\text{sgd}}=\tfrac{\gamma^{\text{shb}}}{\theta^{\text{shb}}}$.  Suppose that $n \to \infty$ and that $\HH$ is chosen so that $\lambda^+_\HH$ is bounded in $n$. Then for any $t >0$
\[
|\Exp_{\HH} f(\XX^{\text{shb}}_t) -\Exp_{\HH} f(\XX^{\text{sgd}}_t)| 
\underset{n \to \infty}{\longrightarrow} 0.
\]
\end{theorem}
 \begin{proof}
 The homogenized SGD diffusion for SHB is the same as the diffusion for SDAHB with parameters
 \(
 (\theta^{\text{sdahb}},\gamma^{\text{sdahb}})
 =
 (n \theta^{\text{shb}},
 n\gamma^{\text{shb}})
 \)
 An elementary computation shows that uniformly in compact sets of $\lambda$ and $t$, the forcing function and interaction kernel ($G^{(\lambda)}$ and $K^{(\lambda)}$) of SDAHB with these parameters satisfy
 \[
 G^{(\lambda)}(t) 
 \underset{n \to \infty}{\longrightarrow}
 e^{-2\gamma^{\text{sgd}} \lambda t}
 \quad\text{and}\quad
 K^{(\lambda)}(t) \underset{n \to \infty}{\longrightarrow}
 \gamma^2 \lambda^2 e^{-2\gamma^{\text{sgd}} \lambda t}.
 \]
 Thus under the assumption that the eigenvalues of $\HH$ remain bounded as $n \to \infty$, the forcing function and kernel for each of 
 $\Exp f(\XX^{\text{SHB}}_t)$
 and $\Exp f(\XX^{\text{sGD}}_t)$ differ by an error that  goes to $0$ as $n\to\infty$ uniformly on compact sets of time.
 \end{proof}

\begin{proposition}\label{propG:sGDdefault}
For SGD, with default parameters
\(
\gamma = \frac{1}{\int \lambda \mu(\dif \lambda)},
\)
the Malthusian exponent is at least
\(
\lambda_* \geq \gamma \lambda^{-}.
\)
\end{proposition}
\begin{proof}
For SGD, the Malthusian exponent is given simply as the root of
\[
1=
\mathcal{F}(x)
=\int_0^\infty 
\frac{\gamma_1^2 \lambda^2 }
{
(2\gamma \lambda - x)
}
\mu(\dif \lambda)
\leq 
\frac{\gamma \lambda^- }{2\gamma \lambda^- - x}.
\]
(See \cite{paquetteSGD2021} or send $\theta \to \infty$ with $\gamma_1 = \gamma \theta$ in SDAHB).
Thus for $x =\gamma \lambda^{-}$ we have $\mathcal{F}(x) \leq 1,$ and so $\lambda_* \geq \gamma\lambda^{-}$.
\end{proof}

\begin{proposition}\label{prop:SDAHB}
For $\theta$ sufficiently large, and when $\mathcal{F}(\lambda^-) \leq 1,$ SDAHB with parameters $(\gamma_1,\theta)$ is faster than SGD with parameters $(\gamma = \tfrac{\gamma_1}{\theta})$ but never more than a factor of $2$ than SGD at its default parameter.
\end{proposition}
\begin{proof}
Note that for large $\theta$, with $\omega < 0$ we always have that $F(t)$ has rate
\[
F(t) \sim 
c t^{\alpha-1}e^{-t(\theta - \sqrt{\theta^2 - 4\gamma_1 \lambda^{-}})}.
\]
The rate for $F$ satisfies 
\[
\theta - \sqrt{\theta^2 - 4\gamma_1\lambda^{-}}
>
\theta - \bigl(
\theta - \tfrac{4\gamma_1\lambda^{-}}{2\theta}
\bigr)
=2\gamma \lambda^{-}.
\]
Moreover, the expression on the left is monotone decreasing $\theta$ until the argument of the radical becomes negative.  Hence, we maximize the rate by taking the smallest admissible $\theta,$ which at the convergence threshold is given by
\[
\frac{\theta}{\gamma_1}=\frac{\int \lambda \mu(\dif \lambda)}{2}.
\]
Substituting this ratio into $\theta - \sqrt{\theta^2 - 4\gamma_1\lambda^{-}}$ to remove $\theta$ and then maximizing gives
\(
\tfrac{2\lambda^-}{\int \lambda \mu(\dif \lambda)},
\)
which is no more than a factor of $2$ than SGD at its default parameter.
\end{proof}

\begin{figure}[t!]
    \centering
    \includegraphics[scale = 0.25]{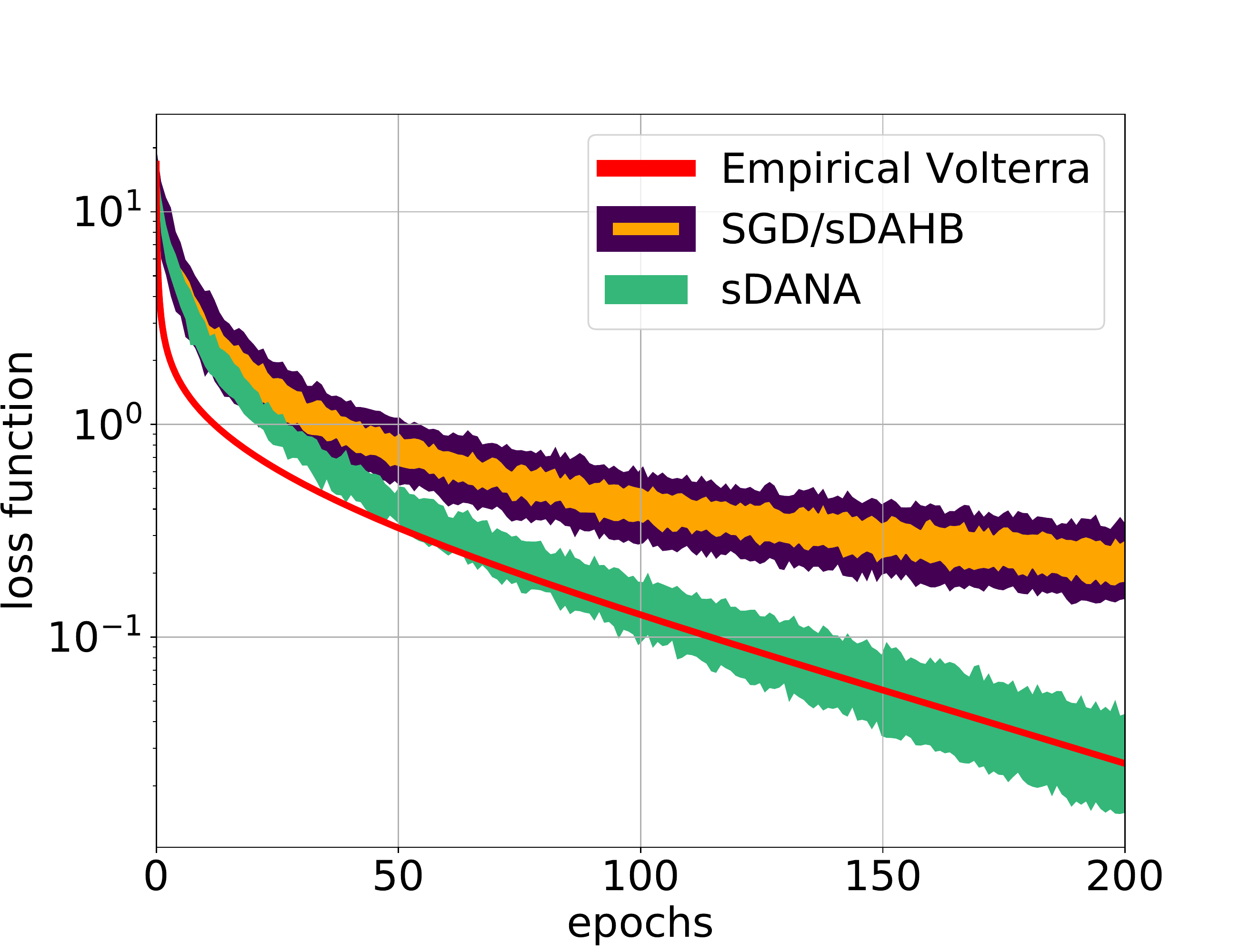}
      \includegraphics[scale = 0.25]{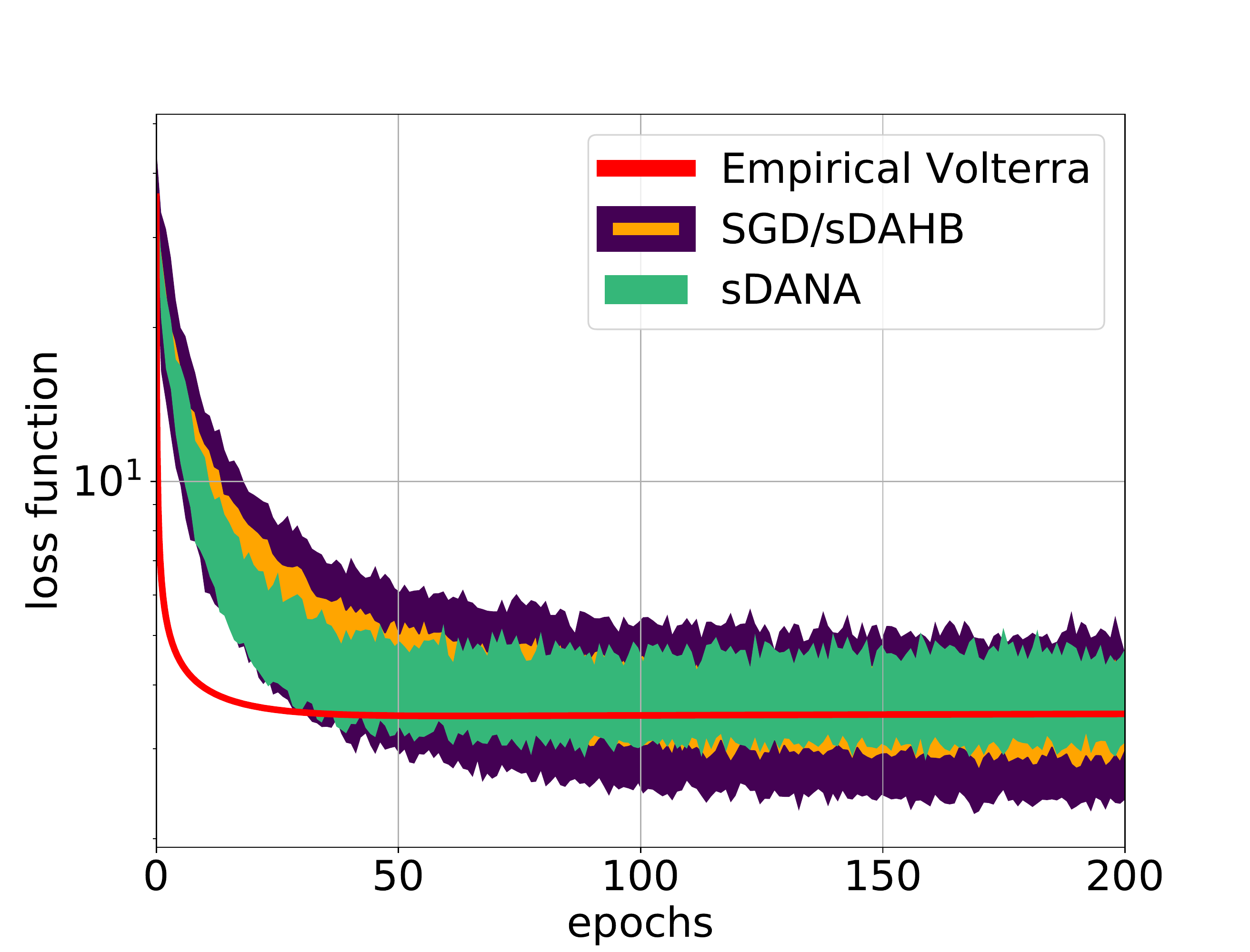}
   \caption{\textbf{SDANA \& SGD vs Theory on MNIST.} 
     MNIST ($60000\times 28 \times 28$ images) \citep{lecun2010mnist} is reshaped into $30$ (left) and $60$ (right) matrices of dimension $1000\times 1568(784)$, representing $1000$ samples of groups of $2$ or $1$ digits, respectively (preconditioned to have centered rows of norm-1).  First digit of each 2 or 1 is chosen to be the target $\bb$. Algorithms were run 30(60) times with default parameters (without tuning) to solve \eqref{eq:lsq}. 80\%--confidence interval is displayed. Volterra (SDANA) is generated with eigenvalues from the first MNIST data matrix with a ratio of signal-to-noise of 6-to-1. Volterra predicts the convergent behavior of SDANA in this non-idealized setting. SDANA outperforms equivalent SGD/SDAHB. 
     }.
    \label{fig:MNIST_extra}
\end{figure}

\section{Numerical simulations}\label{sec: numerical_simulation}
To illustrate our theoretical results, we report simulations using SGD, stochastic heavy-ball (SHB) \citep{Polyak1962Some}, SDAHB, SDANA (Table~\ref{table:stochastic_algorithms}) on the least squares problem. In all simulations of the random least squares problem, the vectors $\xx_0,$ $\widetilde{\xx},$  and $\eeta,$ are sampled i.i.d. from a standard Gaussian $N(0, \tfrac{R}{2d} \II)$, $N(0, \tfrac{R}{2d} \II)$ and $N(0, \tfrac{\widetilde{R}}{n} \II)$ respectively and the entries of $\AA$, $A_{ij} \sim N(0,\tfrac{1}{d})$.  Figures \ref{fig:concentration_SHB}, \ref{fig:SHB_equals_SGD}, \ref{fig:concentration} are with noise; the first two have $R=\widetilde{R} = 1$ and the last is $R = 1=100 \widetilde{R}$.  Figure \ref{fig:SDANA_faster} is with noise 0.  

\paragraph{Volterra equation.} The forcing term $F(t)$ in \eqref{thm:hSGD} is solved by a Runge-Kutta method after which we applied a Chebyshev quadrature rule to approximate the integral with respect to the Marchenko-Pastur distribution. The Chebyshev quadrature is also used to derive a numerical approximation for the kernel, $\mathcal{I}(t-s)$, \eqref{eqF:conv}. Next, to generate the solution $\psi(t)$ of the Volterra equation, we implement a Picard iteration which finds a fix point to the Volterra equation by repeatedly convolving the kernel and adding the forcing term. 

Despite the numerical approximations to integrals, the resulting solution to the Volterra equation ($\psi$, red lines in the plots) models the true behavior of all the stochastic algorithms analyzed in this paper remarkably well (see Fig.~\ref{fig:concentration_SHB}, \ref{fig:SHB_equals_SGD}, and \ref{fig:concentration}). Notably, it captures the oscillatory trajectories in the momentum methods often is seen in practice due to their overshooting (see Fig.~\ref{fig:concentration}).  We note that the Volterra equation for SDANA reliably undershoots simulations of SDANA for small time (say $t < 10$), but matches for larger times ($t > 100$).  This is due in part because the convolution Volterra equation is only an approximation for SDANA that holds as time grows larger, and hence the undershoot is consistent with theory.

\paragraph{Real data.} The MNIST examples (Figures~\ref{fig:MNIST} and \ref{fig:MNIST_extra}) are shown to demonstrate that large--dimensional random matrix predictions often work for large dimensional real data. Figure~\ref{fig:MNIST} is strongly convex as $\lambda^- = 0.041$. This corresponds to a similar convexity structure as $r = 1.44$ in Marchenko-Pastur. Under this convexity, we do not expect SDANA to be faster than SGD/SDAHB. This is reflected in the figure as both SDANA and SGD are parallel to each other after $t > 50$. We chose to include the $6$-sequential images in the main paper in order to show multiple properties of the algorithms in the same image: (1) empirical Volterra and SDANA matched and (2) SGD and SDAHB have similar dynamics. To see the behavior of the algorithms on a pure MNIST dataset, see Figure~\ref{fig:MNIST_extra}. As mentioned above, the Volterra equation always initially underestimates the dynamics of SDANA.

\end{document}

%% file: macros.tex
\def\N{\mathbb{N}}

\def\aa{{\boldsymbol a}}
\def\bb{{\boldsymbol b}}

\def\xx{{\boldsymbol x}}

\def\XX{{\boldsymbol X}}

\def\ZZ{{\boldsymbol Z}}
\def\aa{{\boldsymbol a}}
\def\bb{{\boldsymbol b}}

\def\WW{{\boldsymbol W}}

\def\II{{\boldsymbol I}}
\def\yy{{\boldsymbol y}}

\def\ww{{\boldsymbol w}}

\def\BB{{\boldsymbol B}}
\def\AA{{\boldsymbol A}}

\def\OO{{\boldsymbol O}}
\def\PP{{\boldsymbol P}}

\def\UU{{\boldsymbol U}}
\def\VV{{\boldsymbol V}}
\def\HH{{\boldsymbol H}}

\def\SSigma{{\boldsymbol \Sigma}}

\def\nnu{{\boldsymbol \nu}}

\def\eeta{{\boldsymbol \eta}}

\def\dif{\mathop{}\!\mathrm{d}}

\def\MP{\mu_{\mathrm{MP}}}
\def\RR{{\mathbb R}}
\def\EE{{\mathbb E}\,}
\renewcommand{\vec}{\mathbf{vec}}
\DeclareMathOperator{\tr}{tr}
\def\defas{\stackrel{\text{def}}{=}}

\DeclareDocumentCommand{\Prto} {o} {
  \IfNoValueTF {#1}
  {\overset{\Pr}{\longrightarrow}}
  { \xrightarrow[ #1 \to \infty]{\Pr }}
}

\DeclareDocumentCommand{\Asto} {o} {
  \IfNoValueTF {#1}
  {\overset{\text{\rm a.s.}}{\longrightarrow}}
  { \xrightarrow[ #1 \to \infty]{\text{\rm a.s.} }}
}

\DeclareDocumentCommand{\law} {o} {
  \IfNoValueTF {#1}
  {\overset{\text{law}}{=}}
  { \xrightarrow[ #1 \to \infty]{\Pr }}
}